\documentclass[reqno,11pt,a4paper]{amsart}
\usepackage{exscale,amsmath,amssymb,amsthm,color,graphicx,accents,mathrsfs}
\usepackage{hyperref}
\usepackage[margin=1in]{geometry}
\usepackage{enumitem}
\setlist[enumerate,1]{leftmargin=1cm,label=(\roman*), ref=(\roman*)}
\setlist[itemize,1]{leftmargin=1cm}

\input xy
\xyoption{all}

\theoremstyle{plain}

\newtheorem{theorem}{Theorem}%[section]

\newtheorem{proposition}[theorem]{Proposition}
\newtheorem{corollary}[theorem]{Corollary}
\newtheorem{lemma}[theorem]{Lemma}

\theoremstyle{definition}
\newtheorem{definition}[theorem]{Definition}

\theoremstyle{remark}
\newtheorem{remark}[theorem]{Remark}

\makeatletter
 \DeclareRobustCommand{\checkarg}{\@ifnextchar[{\@witharg}{}}
 \DeclareRobustCommand{\@witharg}[1][]{\ensuremath{\left(#1\right)}}
 \DeclareRobustCommand{\scaleGen}[1]{\@ifnextchar[{\@scalewithargs{#1}}{\odot^{}_{#1}}}
 \def\@scalewithargs#1[#2][#3]{#2 \odot^{}_{#1} #3}
\makeatother

\def\IPspace{\mathcal{I}}
\def\dI{d_{\IPspace}}

\def\cS{\mathcal{S}}

\def\cT{\mathcal{T}}

\def\Concat{ \mathop{ \raisebox{-2pt}{\Huge$\star$} } }
\def\ConcatIL{ \mbox{\huge $\star$} }
\def\concat{\star}
\newcommand{\IPLT}{\mathscr{D}}
\newcommand{\wh}[1]{\widehat{#1}}
\def\BR{\mathbb{R}}				% reals
\def\to{\rightarrow}

\def\cf{\mathbf{1}}				% characteristic fctn
\def\cF{\mathcal{F}}			% sigma-alg
\def\cB{\mathcal{B}}	% Borel sigma-alg
\def\distribfont#1{\texttt{\upshape #1}}%\textnormal{#1}}}
\def\ExpDist{\distribfont{Exponential}\checkarg}
\def\GammaDist{\distribfont{Gamma}\checkarg}
\def\InvGammaDist{\distribfont{InverseGamma}\checkarg}

\def\BetaDist{\distribfont{Beta}\checkarg}

\def\PoiDir{\distribfont{PD}\checkarg}

\newcommand{\besq}{\distribfont{BESQ}}
\def\BESQ{\distribfont{BESQ}\checkarg}
\def\PDIP{\distribfont{PDIP}\checkarg}
\def\cadlag{c\`adl\`ag}

\makeatletter
\let\save@mathaccent\mathaccent
\newcommand*\if@single[3]{%
  \setbox0\hbox{${\mathaccent"0362{#1}}^H$}%
  \setbox2\hbox{${\mathaccent"0362{\kern0pt#1}}^H$}%
  \ifdim\ht0=\ht2 #3\else #2\fi
  }
\newcommand*\rel@kern[1]{\kern#1\dimexpr\macc@kerna}
\newcommand{\widebar}{}% initialize
\DeclareRobustCommand*\widebar[1]{\@ifnextchar^{\wide@bar{#1}{0}}{\wide@bar{#1}{1}}}
\newcommand*\wide@bar[2]{\if@single{#1}{\wide@bar@{#1}{#2}{1}}{\wide@bar@{#1}{#2}{2}}}
\newcommand*\wide@bar@[3]{%
  \begingroup
  \def\mathaccent##1##2{%
    \let\mathaccent\save@mathaccent
    \if#32 \let\macc@nucleus\first@char \fi
    \setbox\z@\hbox{$\macc@style{\macc@nucleus}_{}$}%
    \setbox\tw@\hbox{$\macc@style{\macc@nucleus}{}_{}$}%
    \dimen@\wd\tw@
    \advance\dimen@-\wd\z@
    \divide\dimen@ 3
    \@tempdima\wd\tw@
    \advance\@tempdima-\scriptspace
    \divide\@tempdima 10
    \advance\dimen@-\@tempdima
    \ifdim\dimen@>\z@ \dimen@0pt\fi
    \rel@kern{0.6}\kern-\dimen@
    \if#31
      \overline{\rel@kern{-0.6}\kern\dimen@\macc@nucleus\rel@kern{0.4}\kern\dimen@}%
      \advance\dimen@0.4\dimexpr\macc@kerna
      \let\final@kern#2%
      \ifdim\dimen@<\z@ \let\final@kern1\fi
      \if\final@kern1 \kern-\dimen@\fi
    \else
      \overline{\rel@kern{-0.6}\kern\dimen@#1}%
    \fi
  }%
  \macc@depth\@ne
  \let\math@bgroup\@empty \let\math@egroup\macc@set@skewchar
  \mathsurround\z@ \frozen@everymath{\mathgroup\macc@group\relax}%
  \macc@set@skewchar\relax
  \let\mathaccentV\macc@nested@a
  \if#31
    \macc@nested@a\relax111{#1}%
  \else
    \def\gobble@till@marker##1\endmarker{}%
    \futurelet\first@char\gobble@till@marker#1\endmarker
    \ifcat\noexpand\first@char A\else
      \def\first@char{}%
    \fi
    \macc@nested@a\relax111{\first@char}%
  \fi
  \endgroup
}
\makeatother

\newcommand{\parent}[1]{\accentset{\leftarrow}{#1}}
\newcommand{\longparent}[1]{\accentset{\longleftarrow}{#1}}
\DeclareRobustCommand{\Ast}[1]{\accentset{*}{#1}}

\newcommand{\wi}{{\tilde{\textit{\i}}}}
\newcommand{\wj}{{\tilde{\textit{\j}}}}

\newcommand{\ft}{\mathbf{t}}

\newcommand{\ff}{\mathbf{f}}
\newcommand{\fm}{\mathbf{m}}

\newcommand{\bN}{\mathbb{N}}

\newcommand{\fT}{\mathbf{T}}

\newcommand{\cR}{\mathcal{R}}

\newcommand{\cI}{\mathcal{I}}

\newcommand{\cX}{\mathcal{X}}

\newcommand{\cU}{\mathcal{U}}
\newcommand{\cV}{\mathcal{V}}
\newcommand{\cW}{\mathcal{W}}

\newcommand{\bP}{\mathbb{P}}
\newcommand{\bR}{\mathbb{R}}
\newcommand{\bE}{\mathbb{E}}
\newcommand{\bT}{\mathbb{T}}

\newcommand{\bX}{\mathbb{X}}
\newcommand{\bY}{\mathbb{Y}}

\newcommand{\TInt}{\bT}%^{\textnormal{int}}}
\newcommand{\TShape}{\bT^{\textnormal{shape}}}
\newcommand{\bTInt}{\widebar{\bT}}%^{\textnormal{int}}}
\newcommand{\tdTInt}{\widetilde{\bT}}%^{\textnormal{int}}}

\newcommand{\block}{\textsc{block}}

\newcommand{\wbD}{\widebar{D}}
\newcommand{\wbT}{\widebar{\cT}}

\newcommand{\bTMarkk}{\Ast\bTInt_k}
\newcommand{\tdTMarkk}{\Ast\tdTInt_k}
\newcommand{\TMarkk}{\Ast\bT_{k}}%{\textnormal{mark}}_{[k]}}

\newcommand{\wt}[1]{\widetilde{#1}}

\allowdisplaybreaks[3]

\begin{document}

\title[A projective system of $k$-tree evolutions]{Aldous diffusion I: %construction via 
	\\ a projective system of continuum $k$-tree evolutions}

\author{N\MakeLowercase{\sc oah} F\MakeLowercase{\sc orman,} S\MakeLowercase{\sc oumik} P\MakeLowercase{\sc al,} D\MakeLowercase{\sc ouglas} R\MakeLowercase{\sc izzolo, and}  M\MakeLowercase{\sc atthias} W\MakeLowercase{\sc inkel}}

\address{\hspace{-0.42cm}N.~Forman\\ Department of Mathematics\\ University of Washington\\ Seattle WA 98195\\ USA\\ Email: noah.forman@gmail.com}             

\address{\hspace{-0.42cm}S.~Pal\\ Department of Mathematics\\ University of Washington\\ Seattle WA 98195\\ USA\\ Email: soumikpal@gmail.com}

\address{\hspace{-0.42cm}D.~Rizzolo\\ 531 Ewing Hall\\ Department of Mathematical Sciences\\ University of Delaware\\ Newark DE 19716\\ USA\\ Email: drizzolo@udel.edu}

\address{\hspace{-0.42cm}M.~Winkel\\ Department of Statistics\\ University of Oxford\\ 24--29 St Giles'\\ Oxford OX1 3LB\\ UK\\ Email: winkel@stats.ox.ac.uk}             



\date{\today}

\thanks{This research is partially supported by NSF grants DMS-1204840, DMS-1444084, DMS-1612483, EPSRC grant EP/K029797/1, and the University of Delaware Research Foundation}

%\author{Noah Forman \and Soumik Pal \and Douglas Rizzolo \and Matthias Winkel}             

\begin{abstract} The Aldous diffusion is a conjectured Markov process on the space of real trees that is the continuum analogue of discrete Markov chains on binary trees. We construct this conjectured process via a consistent system of stationary evolutions of binary trees with $k$ labeled leaves and edges decorated with diffusions on a space of interval partitions constructed in previous work by the same authors. This pathwise construction allows us to study and compute path properties of the Aldous diffusion including evolutions of projected masses and distances between branch points. A key part of proving the consistency of the projective system is Rogers and Pitman's notion of intertwining. 
\end{abstract}

\maketitle

\vspace{-0.5cm}

\section{Introduction}

The Aldous chain \cite{Aldous00,Schweinsberg02} is a Markov chain on the space of (rooted) binary trees with $n$ labeled leaves. Each transition of the chain, called a down-up move, has two stages. First, a uniform random leaf is deleted and its parent branch point is contracted away. Next, a uniform random edge is selected, we insert a new branch point into the middle of that edge, and we extend a new leaf-edge out from that branch point. This is illustrated in Figure \ref{fig:AC_move} where $n=6$ and the leaf labeled $3$ is deleted and re-inserted. The stationary distribution of this chain is the uniform distribution on rooted binary trees with $n$ labeled leaves. 
\begin{figure}[bth]
 \centering
 \scalebox{0.9}{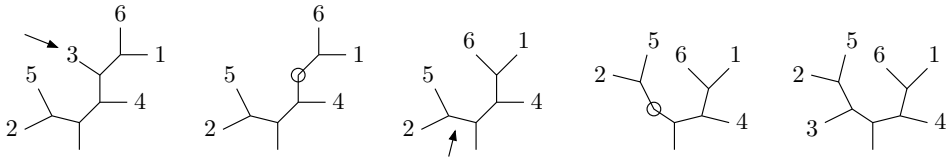}
 \caption{From left to right, one Aldous down-up move.\label{fig:AC_move}}
\end{figure}

Aldous observed that if $b$ branch points are selected in the initial tree and the leaves are partitioned into sets based on whether or not they are in the same component of the tree minus these branch points then, up until one of the selected branch points disappears, the scaling limit of the evolution of the sizes of these sets under the Aldous chain is a Wright--Fisher diffusion, running at quadruple speed, with some positive parameters and some negative ones.  He conjectured that these diffusions were recording aspects of a diffusion on continuum trees and posed the problem of constructing this process \cite{ADProb2, AldousDiffusionProblem}.  This process would be stationary under the law of the Brownian continuum random tree (BCRT) because this is the scaling limit of the stationary distribution of the Aldous chain \cite{AldousCRT1, AldousCRT2, AldousCRT3}. Building on \cite{Paper0,Paper1,Paper2,Paper3}, the present paper resolves this conjecture via a pathwise construction of a \textit{consistent finite-dimensional projected} system of $k$-tree evolutions that are described below.  In a follow-up paper \cite{Paper5} we establish some important properties of this process.  The idea of constructing a BCRT as a projective system of finite random binary trees goes back to the original construction in \cite{AldousCRT1}. The critical extension here is that the projective consistency now has to hold for the entire stochastic process.  

The concept of a $k$-tree is best understood by generating a random $k$-tree from the BCRT. Consider a BCRT $(\cT,d,\rho,\mu)$. Roughly, $\left(\cT, d\right)$ is a metric space that is tree-shaped and has a distinguished element $\rho$ called the root and an associated Borel probability measure $\mu$ called the leaf mass measure. Of course, these quantities are all random. Given an instance of this tree, let $\Sigma_n$, $n\!\ge\!1$, denote a sequence of leaves sampled conditionally i.i.d.\ with law $\mu$. Denote by $\cR_k$ the subtree of $\cT$ spanned by $\rho,\Sigma_1,\ldots,\Sigma_k$, and by $\cR_k^\circ$ the subtree of $\cR_k$ spanned by the set ${\rm Br}(\cR_k)$ of branch points and the root $\rho$ of $\cR_k$. Almost surely, $\cR_k$ is a binary tree.  

Let $[k]\!:=\!\{1,\ldots,k\}$ and suppose that $k\geq 2$. The \emph{Brownian reduced $k$-tree}, denoted by $\cT_k\!=\!(\ft_k,(X_j^{(k)}\!,j\!\in\![k]),(\beta_E^{(k)}\!,E\!\in\!\ft_k))$, is defined as follows.
\begin{itemize}
  \item The combinatorial \emph{tree shape} $\ft_k$ is in one-to-one correspondence with ${\rm Br}(\cR_k)$ or with the branches of $\cR_k^\circ$. Specifically, for each \emph{branch}, i.e.\ each connected component $\cB\subseteq\cR_k^\circ\setminus{\rm Br}(\cR_k)$, the corresponding element of $\ft_k$ is the set $E$ of all $i\in[k]$ for which leaf $\Sigma_i$ is ``above $\cB$,'' i.e.\ is in the component of $\cR_k\setminus\cB$ not containing the root. We denote by $\cB_E$ the branch corresponding to $E$. Each element of $\ft_k$ is a set containing at least two leaf labels; we call the elements of $\ft_k$ the \emph{internal edges}. %\vspace{-0.15cm}
  \item For $j\in[k]$, the \emph{top mass} $X_j^{(k)}$ is the $\mu$-mass of the component of $\cT\setminus\cR_k^\circ$ containing $\Sigma_j$.
  \item Each edge $E\in\ft_k$ is assigned an interval partition $\beta_E^{(k)}$ as in \cite{GnedPitm05,PitmWink09}. Consider the connected components of $\cT\setminus\cR^\circ_k$ that attach to $\cR_k^\circ$ along the interior of $\cB_E$. These are totally (but not sequentially) ordered by decreasing distance between their attachment points on $\cB_E$ and the root. Then $\beta_E^{(k)}$ is the interval partition whose \textit{block} lengths are the $\mu$-masses of components in this total order.
\end{itemize}
See Figure \ref{fig:B_k_tree_proj} for an illustration of such a $k$-tree.

\begin{figure}[t!]\centering
  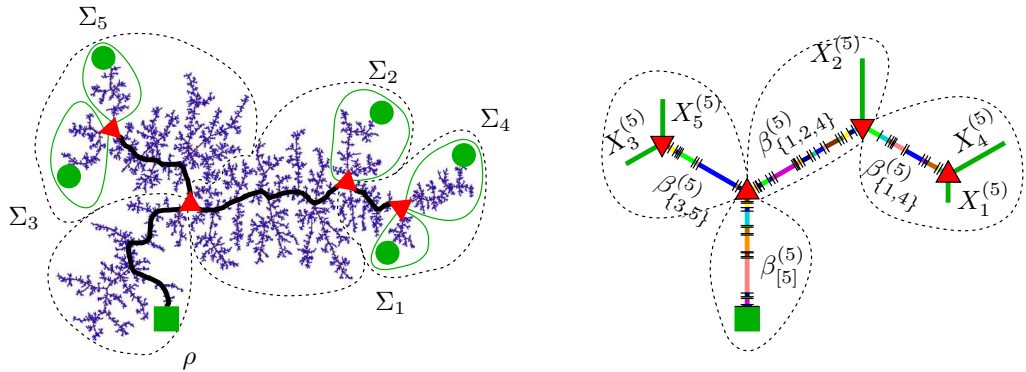
  \caption{Left: Simulation of a BCRT (courtesy of Igor Kortchemski) %here with root $\rho$, and $k=5$ leaves $\Sigma_1,\ldots,\Sigma_5$, decomposition $\cR_5^\circ$ with subtrees; associated Brownian reduced $k$-tree $(\ft_5,(X_1^{(5)},\ldots,X_5^{(5)}),(\beta_{[5]}^{(5)},\beta_{\{1,2,4\}}^{(5)},\beta_{\{1,4\}}^{(5)},\beta_{\{3,5\}}^{(5)}))$.\label{fig:B_k_tree_proj}}
  with root $\rho$, and $k=5$ leaves $\Sigma_1,\ldots,\Sigma_5$. The bold lines and triangles are the branches and vertices of $\cR_5^\circ$.  Right: The associated Brownian reduced $k$-tree.\label{fig:B_k_tree_proj}}%180917
\end{figure}

We later provide more precise definitions of sets of $k$-trees that support the laws of the Brownian reduced $k$-trees. There is a natural projection map $\pi_{k}$ from $(k+1)$-trees to $k$-trees for every $k\ge 2$ such that $\pi_k(\cT_{k+1})=\cT_k$. There is also a natural map $S_k$ taking $k$-trees to rooted metric measure trees such that, almost surely, $S_k(\cT_k)$ is the rooted metric measure tree that results from projecting the leaf mass measure of $\cT$ onto $\cR_k^\circ$.  To see this, observe that since the combinatorial tree shape of $\cR_k^\circ$ is $\ft_k$ and atoms of the projected measure are given by the top masses and intervals of the interval partitions, one only needs to be able to recover the metric structure of $\cR_k^\circ$ from $\cT_k$.  This is accomplished using the diversity of the interval partitions on the edges as in \cite{PitmWink09}, see also Equation \eqref{eq diversitydef} below for a precise definition of diversity.  Given any system of $k$-trees $\mathbf{T}=(T_k, k\geq 2)$, we define $S(\mathbf{T}) = \lim_{k\to\infty} S_k(T_k)$, where the limit is taken with respect to the rooted Gromov--Hausdorff--Prokhorov metric \cite{M09}, provided this limit exists, and $S(\mathbf{T})$ is the trivial one-point tree if the limit does not exist.  From the definition of $\cR^\circ_k$, it is clear that $S((\cT_k,k\geq 1)) = \cT$ almost surely, see e.g. \cite{PitmWink09} for a similar argument in the context of bead-crushing constructions.

The following collection of results summarizes the main contribution of this paper.

\begin{theorem}\label{thm:intromain}
 There is a projective system $(\widebar\cT^u, u\geq 0)=((\widebar\cT^u_{\!\!k},k\geq 2), u\geq 0)$ such that the following hold.
 \begin{enumerate}[itemsep=.1cm]
  \item \label{main markov} For each $k$, $(\widebar{\cT}_{\!\!k}^u,u\geq 0)$ is a $k$-tree-valued Markovian evolution. %\vspace{.1cm}
  \item \label{main cons} These processes are consistent in the sense that $(\widebar\cT^u_{\!\!k},u\geq 0) = (\pi_k(\widebar\cT^u_{\!\!k+1}),u\geq 0)$. %\vspace{.1cm}
  \item \label{main stat} The law of the consistent family of Brownian reduced $k$-trees $(\widebar\cT_{\!\!k}, k\geq 2)$ is a stationary law for the process $(\widebar\cT^u, u\geq 0)$.
  \item \label{main Wright--Fisher} Let $\widebar\cT^u_{\!\!k} = (\ft^u_k,(X^u_j\!,j\!\in\![k]),(\beta^u_E,E\!\in\!\ft^u_k))$, let $\|\beta^u_E\|$ be the mass of the interval partition $\beta^u_E$, i.e.\ the sum of the lengths of the intervals in the partition, and let $\tau$ be the first time either a top mass or the mass of an interval partition is $0$.    For $u<\tau$ we have $\ft^u_k=\ft^0_k$.  The process $((( X^{u/4}_j\!,j\!\in\![k]),(\|\beta^{u/4}_E\|,E\!\in\!\ft^0_k)), 0\leq u <\tau)$ is a Wright--Fisher diffusion, killed when one of the coordinates vanishes, with the parameters proposed by Aldous, which are $-1/2$ for coordinates corresponding to top masses and $1/2$ for coordinates corresponding to masses of interval partitions.
  \item  \label{main GHPAldous} If $((\widebar\cT^u_{\!\!k},k\geq 2), u\geq 0)$ is running in stationarity, then $(S(\widebar\cT^u), u\geq 0)$ is a stationary stochastic process of rooted metric measure trees whose stationary distribution is the BCRT.
 \end{enumerate}
\end{theorem}

For fixed $k$, we give a pathwise definition of $(\widebar{\cT}_{\!\!k}^u,u\geq 0)$ in Definition \ref{def:resamp_1}.  Properties \ref{main markov} and \ref{main stat} are proved in Theorem \ref{thm:dePoi}.  Property \ref{main cons} follows from Theorem \ref{thm:consistency} and the Kolmogorov consistency theorem.  Property \ref{main Wright--Fisher} is proved in Corollary \ref{cor:WF}.  Property \ref{main GHPAldous} is an immediate consequence of Property \ref{main stat} and the discussion immediately following the definition of $S$ above.

\begin{definition}\label{def:Aldousdiff}
The \textit{Aldous diffusion} is the process $(S(\cT^y), y\geq 0)$ defined in Theorem \ref{thm:intromain}\ref{main GHPAldous}.
\end{definition}

In the follow-up paper \cite{Paper5} we show that the Aldous diffusion has the Markov property and a continuous modification.  The techniques used to show these properties are quite different from those used in the current paper.

The current paper and its companion \cite{Paper5} are the culmination of several ideas that we have developed in our previous joint work. Although it is not necessary to read all the previous papers to follow the mathematics here, in the interest of the ``big picture'', Table \ref{tbl:project_outline} outlines their dependence structure. The $k$-tree evolutions described here require Markovian evolutions on spaces of interval partitions. Their existence and properties have been worked out in \cite{Paper1}. As mentioned above, the leaf masses of the continuum-tree-valued process are captured by the interval lengths of the evolving interval partitions. What is not obvious is that the evolution of the metric structure of the continuum-tree-valued process is also captured by the so-called \textit{diversity} of the interval-partition-valued process. This has been dealt with in \cite{Paper0}. The $2$-tree-valued process, which is an important building block in the current work, has been separately constructed in \cite{Paper3}. The consistency of the $k$-tree-valued process is subtle and requires a non-trivial labeling of the $k$-tree shapes and a \textit{resampling} mechanism when the randomly evolving tree shape drops a leaf label. In \cite{Paper2}, a similar labeling and resampling scheme has been worked out for the Aldous Markov chain where it has been proved to lead to consistent projections. The proofs of consistency in \cite{Paper2} and the current work are related in the sense that they both employ a combination of two well-known criteria: intertwining and Dynkin--Kemeny--Snell. However, the proof of \cite{Paper2} cannot be directly generalized in the continuum and it necessitates new arguments that are developed here tying together the threads of \cite{Paper1, Paper3, Paper0}.

\begin{table}
 \centering
 \begin{tabular}{cc}
  Uniform control of local times of stable processes \cite{Paper0} & \\
  $\downarrow$ & \\
  %Interval partition evolutions \cite{Paper1} & \\
  Interval-partition-valued diffusions \cite{Paper1} & \\
  $\downarrow$ %& \\%Consistent projections \\
%  2-tree evolution under Aldous diffusion \cite{Paper3} & of the Aldous chain \cite{Paper2}
  & Consistent projections \\
  Interval partition evolutions with emigration \cite{Paper3} & of the Aldous chain \cite{Paper2}
 \end{tabular}
 
 $\underbrace{\hphantom{Interval partition evolutions with emigration [20] of the Aldous chain [20] fillerfil}}$\\
 $\downarrow$\\
 Projective system of continuum $k$-tree evolutions {[this paper]}\\
 $\downarrow$\\
 Properties of the continuum-tree-valued process \cite{Paper5}
 
 \vphantom{text}
 
 \caption{Outline of the present authors' construction of Aldous the diffusion.\label{tbl:project_outline}} \vspace{-.7cm}
\end{table}

%In Section \ref{sec:dePoi}
%We conclude in Section \ref{sec:down_from_infty} with 

\subsection{Related work of L\"ohr, Mytnik, and Winter}

Recently L\"ohr, Mytnik, and Winter \cite{LohrMytnWint18} used a martingale problem to independently prove the existence of a process on a new space of trees, which they call algebraic measure trees.  They named their process the \textit{Aldous diffusion on binary algebraic non-atomic measure trees}, which they abbreviated to \textit{Aldous diffusion}, but for clarity, we will abbreviate as the \textit{algebraic Aldous diffusion}.  Algebraic measure trees can be thought of as the structures that remain when one forgets the metric on a metric measure tree but retains knowledge of the branch points; cf.\ mass-structural equivalence in \cite{IPTrees}. Equivalence classes of algebraic trees form the state space for the algebraic Aldous diffusion.  They showed that, on this state space with a new topology that they introduce, the Aldous chain with time rescaled by $N^2$ converges to the algebraic Aldous diffusion as the number of leaves tends to infinity.  In explaining the move to a new state space that does not include the metric structure of the trees they say that, although it is traditionally useful to include the metric structure and this was part of the original problem, the metric might not evolve naturally under the Aldous chain because the quadratic variation of some functions of the distance scale differently than one would anticipate from time scaling of $N^2$ predicted by Aldous and further supported by Pal \cite{Pal13}.  We believe, however, that this is not the case.  Our results show that the evolution of the metric can be described using the local time of a L\'evy process.  One does not expect this to be a semimartingale and we believe this is the cause of the difficulty identified in \cite{LohrMytnWint18}. 

Changing to a new state space and topology allows the authors of \cite{LohrMytnWint18} to use standard martingale problem techniques to establish the existence of the algebraic Aldous diffusion, however it also results in the loss of some information that we believe formed an important aspect of Aldous's conjecture.  In particular, in order to remove information about the metric, it seems that one also must sacrifice some detailed information about the distribution of leaves.  Indeed, from Pitman and Winkel \cite{PitmWink09}, a detailed knowledge of the distribution of leaves allows you to reconstruct the metric.  This forms a central part of our construction here.  The new state space and topology also seem insufficient for recovering the mixed Wright--Fisher diffusions that Aldous conjectured were recording aspects of the underlying continuum-tree-valued diffusion.  Rather, in \cite{LohrMytnWint18} the authors recover the annealed behavior of how the mass split around a typical branch point behaves, thus recovering some aspects of the negative Wright--Fisher diffusion in an averaged sense (the methods they use may be able to recover the behavior of the mass split among multiple branch points in the same averaged sense, but they have not done this).  In contrast, our approach shows clearly how the mixed Wright--Fisher diffusions observed by Aldous embed into the process we construct on continuum trees. More generally, our pathwise construction grants access to sample path properties of the diffusion, which may be difficult to study from the martingale problem approach.

There is a natural conjecture relating the processes we call the Aldous diffusion and the algebraic Aldous diffusion.  In particular, in this paper in Theorem \ref{thm:intromain} we construct a consistent system of $k$-tree evolutions $((\cT^y_k)_{y\geq 0})_{k\geq 0}$ that captures, in a consistent way, the mixed Wright--Fisher diffusions observed by Aldous.  In Definition \ref{def:Aldousdiff} we define the Aldous diffusion by mapping $\cT^y_k$ to a metric measure tree using the diversities and block sizes in the interval partitions to determine branch lengths and masses of atoms, then taking a projective limit as $k\to \infty$ in the Gromov--Hausdorff--Prokhorov metric.  If instead of mapping $\cT^y_k$ to a metric measure tree we map it to an algebraic measure tree and take the limit in the topology of \cite{LohrMytnWint18} (which is easily seen to exist), we conjecture that the resulting process is the algebraic Aldous diffusion.

\subsection{Structure of the paper}

In Section \ref{sec:type012} we briefly recall the main objects and some properties from \cite{Paper1,Paper3}. In Section \ref{sec:def_thm} we formally introduce the space of $k$-trees, define various $k$-tree-valued Markov processes, and state our main results. In particular, we begin by defining processes in which the total mass of the tree fluctuates and eventually the tree dies out; we obtain processes with constant mass by ``de-Poissonization.'' We prove some properties of these processes for fixed $k$, including ``pseudo-stationarity'' results for the processes with fluctuating mass, in Section \ref{sec:k_fixed}. We apply these pseudo-stationarity results in Section \ref{sec:consistency} to prove that these Markov processes are projectively consistent under an additional hypothesis that we remove in Section \ref{sec:non_acc}. Section \ref{sec:const:other} then includes proofs of additional projective consistency results.

%\vspace{-.1cm}
\section{Preliminaries on type-0, type-1, and type-2 evolutions}\label{sec:type012}

Type-0, type-1, and type-2 evolutions are Markov processes introduced with pathwise constructions in \cite{Paper1,Paper3}. In \cite[Section 1.1]{Paper3}, we argue via a connection to ordered Chinese restaurant processes that the type-2 evolution is a continuum analogue of a certain 2-tree projection of the discrete Aldous chain discussed in \cite[Appendix A]{Paper2}. %For a related model, see \cite{Petrov09}.
By the same argument, the $k$-tree projection of the Aldous chain discussed in \cite[Appendix A]{Paper2} can be decomposed into parts whose evolutions are analogous to type-0/1/2 evolutions. In Figure \ref{fig:B_k_tree_proj}, the dashed lines separate parts of the $k$-tree that evolve as type-0/1/2 evolutions; see Definition \ref{def:killed_ktree}. %180917
%The relationships between type-0 and type-1 evolutions and the Aldous chain, which motivate their roles in the forthcoming Definition \ref{def:killed_ktree} of $k$-tree evolutions, can be inferred from this same argument in \cite{Paper3}. 

The aforementioned pathwise construction brings a lot of symmetry to light, and it makes many calculations accessible. In this paper, we will not delve into this construction, and in fact, only a few key properties of these processes are needed. For the sake of completeness, we include enough information to fully specify the distributions of these processes.

An \emph{interval partition} (IP) is a set $\alpha$ of disjoint open intervals that cover some compact interval $[0,M]$ up to a Lebesgue-null set. $M$ is called the mass of $\alpha$ and denoted by $\|\alpha\|$. We call the elements $(a,b)\in\alpha$ the \emph{blocks} of the partition. A block $(a,b)$ has \emph{mass} $b-a$. We say that $\alpha$ has the \emph{$\frac12$-diversity property} if the following limit exists for every $t\ge0$:
\vspace{ -.1cm}
\begin{equation}\label{eq diversitydef}
 \IPLT_{\alpha}(t) := \sqrt{\pi}\lim_{h\to 0}\sqrt{h}\#\{(a,b)\in\alpha\colon |b-a|>h, b < t\}.
 \vspace{ -.1cm}
\end{equation}

Let $\cI$ denote the space of interval partitions with the $\frac12$-diversity property. We denote the left-to-right concatenation of interval partitions by the operator $\concat$ or $\ConcatIL$, as in $\alpha\concat \beta$ or $\ConcatIL_{n\ge 1}\alpha_n$. Then $\cI$ is closed under pairwise concatenation.

This space is equipped with a metric $d_{\cI}$ \cite[Definition 2.1]{Paper1}. %under which two partitions $\alpha$ and $\beta$ are close if there is a one-to-one correspondence between the ``large'' blocks of the two partitions that respects left-to-right order, and under which: (i) corresponding blocks $(a,b)\leftrightarrow (a',b')$ are close in mass, $|b-a|\approx |b'-a'|$; (ii) the diversities to the left of corresponding blocks are close, $\IPLT_\alpha(a)\approx \IPLT_\beta(a')$; and (iii) the aggregate mass of the blocks left out of the correspondences is small. Thus, in 
In a continuous process $(\alpha^y,y\ge0)$ in this metric topology, large blocks of the partition persist over time, with continuously fluctuating endpoints $(a^y,b^y)$, while smaller blocks may be born or die out (enter and exit continuously at mass 0) in such a way that the diversity $\IPLT_{\alpha^y}(t)$ up to any point $t$ in the partition evolves continuously in $y$ as well.
%a manner that tends to maintain the diversity of the partition.

We follow the convention of \cite{Paper3} that type-$i$ evolutions, for $i=0,1,2$, are valued in $[0,\infty)^i\times \cI$ (we emphasized a different, $\cI$-valued, representation of type-1 evolutions in \cite{Paper1}). We refer to the real-valued first coordinates of type-1 and type-2 evolutions as \emph{top blocks} or \emph{top masses}. Each block in a type-$i$ evolution, including these top blocks, has mass that fluctuates as a squared Bessel diffusion \BESQ[-1]. For $\theta\in\BR$ and $m\ge0$, the $\besq_m(\theta)$ diffusion satisfies
\vspace{ -.1cm}
\begin{equation}
 Z_t = m + \theta t + \int_{u=0}^t 2\sqrt{|Z_u|}dB_u \qquad \text{for }t\le \inf\{u\ge0\colon Z_u = 0\}.
 \vspace{ -.1cm}
\end{equation}

See \cite{RevuzYor}. We take the convention that for $\theta\le0$, $\besq(\theta)$ diffusions are absorbed at zero.

Informally, when a top block of a type-1 or type-2 evolution hits mass zero, the leftmost blocks of the interval partition component of the evolution are successively (informally speaking, as the blocks are not well-ordered) pulled out of the interval partition to serve as new top blocks, until their \BESQ[-1] masses are absorbed at zero. %Type-0/1 evolutions are specified by the following, along with Proposition \ref{prop:012:pred}.

\begin{proposition}[Propositions 4.30, 5.4, 5.16 of \cite{Paper1}]\label{prop:012:transn}
 Fix a time $y>0$ and a state $\beta\in\cI$. For each block $(a,b)\in\beta$,
 \begin{itemize}
  \item let $\chi_{(a,b)}\sim \distribfont{Bernoulli}(1-e^{-w/2y})$, where $w = b-a$,
  \item let $L_{(a,b)}$ be a $(0,\infty)$-valued random variable with probability density function
  \begin{equation}
   \bP\{L_{(a,b)}\in du\} = \frac{1}{\sqrt{2\pi}}\frac{\sqrt{y}}{u^{3/2}}\frac{e^{-u/2y}}{e^{w/2y}-1}\left( 1 - \cosh\!\left(\frac{\sqrt{wu}}{y}\right) + \frac{\sqrt{wu}}{y} \sinh\!\left(\frac{\sqrt{wu}}{y}\right) \right)du,
  \end{equation}
  \item let $S_{(a,b)}\sim\ExpDist[(2y)^{-1/2}]$, and
  \item let $R_{(a,b)}$ be a subordinator with Laplace exponent 
  %\begin{equation}
   $\lambda\mapsto \left(\lambda+\frac{1}{2y}\right)^{1/2}-\left(\frac{1}{2y}\right)^{1/2}$.
  %\end{equation}
 \end{itemize}
 We take these objects to all be jointly independent. Let $(R_0,S_0)$ denote an additional independent pair, with this same law as that given for $(R_{(a,b)},S_{(a,b)})$. For the purpose of the following, we take
 $$\textsc{ip}(R,S) := \big\{(R(t-),R(t))\colon t < S,\ R(t-)<R(t)\big\}.$$
 
 The map that send an initial state $\beta$ and time $y$ to the law of 
 %If $(\alpha^z,z\ge0)$ is a type-0 evolution with initial state $\beta$, then
 \begin{equation}\label{eq:type1:kernel}
  \Concat_{U\in \beta\colon \chi_U=1}\big((0,L_U)\concat \textsc{ip}(R_U,S_U)\big)
 \end{equation}
 gives a transition semigroup on $\cI$. The same holds for the map sending $\beta$ and $y$ to the law of 
 %If $((m^z,\alpha^z),z\ge0)$ is instead a type-1 evolution with initial state $(x,\beta)$, then
 \begin{equation}\label{eq:type0:kernel}
  \textsc{ip}(R_0,S_0)\concat \Concat_{U\in \beta\colon \chi_U=1}\big((0,L_U)\concat \textsc{ip}(R_U,S_U)\big).
 \end{equation}
\end{proposition}

\begin{definition}\label{def:type01}
 \emph{Type-0 evolutions} are right-continuous Markov processes on $\cI$ with transition kernel \eqref{eq:type0:kernel}. \emph{IP-valued type-1 evolutions} are right-continuous Markov processes $((\gamma^y),y\ge0)$ on $\cI$ with transition kernel \eqref{eq:type1:kernel}. Such an IP-valued type-1 evolution is associated with a \emph{(pair-valued) type-1 evolution} $((m^y,\alpha^y),y\ge0)$ where $m^y$ is the mass of the leftmost block of $\gamma^y$, or zero if there is no leftmost block, while $\alpha^y$ comprises the rest of the partition, so that $(0,m^y)\concat\alpha^y = \gamma^y$.
 
 A type-1 evolution is said to \emph{degenerate} when it is absorbed at $(0,\emptyset)$. A type-0 evolution is never absorbed and is said to have degeneration time $\infty$.
\end{definition}

It is helpful to bear in mind that, for any interval partition $\beta$ and $y>0$, for $\chi_{(a,b)}$ as defined in Proposition \ref{prop:012:transn}, the sum $\sum_{U\in\beta}\chi_U$ is a.s.\ finite; this is easily proved using the Borel--Cantelli lemma. In other words, for a type-0/1 evolution with initial state $\beta$ a.s.\ only finitely many of $(0,L_U)\concat \textsc{ip}(R_U,S_U)$ contribute to the state of the type-0/1 evolution at the later time. From this it is seen that the IP-valued type-1 evolution a.s.\ has a leftmost block at any given time, unless it has degenerated before that time.

\begin{definition}\label{def:type2}
 Let $(x_1,x_2,\beta)\in [0,\infty)^2\times\cI$ with $x_1+x_2>0$. A \emph{type-2 evolution} starting from $(x_1,x_2,\beta)$ is a process of the form 
$((m_1^y,m_2^y,\alpha^y),y\ge 0)$, with $(m_1^0,m_2^0,\alpha^0)=(x_1,x_2,\beta)$. Its law is specified by the following construction.
 
 Let $\left(\fm^{(0)}, \gamma^{(0)} \right)$ be a type-1 evolution starting from $( x_2, \beta)$ independent of $\ff^{(0)}\sim \besq_{x_1}(-1)$. Let $Y_0:=0$ and $Y_1=\zeta(\ff^{(0)})$. For $0\le y \le Y_1$, define the type-2 evolution as
 \[
 \left( m_1^y, m_2^y, \alpha^y  \right):=\left(\ff^{(0)}(y), \fm^{(0)}(y), \gamma^{(0)}(y)  \right), \quad 0\le y \le Y_1.
 \]
 Now proceed inductively. Suppose, for some $n\ge 1$, these processes have been constructed until time $Y_n$ with $m_1^{Y_n}+m_2^{Y_n}>0$. Conditionally given this history, consider a type-1 evolution $(\fm^{(n)}, \gamma^{(n)})$ with initial condition $(0,\alpha^{Y_n}) = (0,\gamma^{(n-1)}(Y_n-Y_{n-1}))$ that is independent of $\ff^{(n)}$, a $\besq(-1)$ diffusion with initial value $\fm^{(n-1)}(Y_n-Y_{n-1})$. The latter equals $m_2^{Y_n}$ if $n$ is odd or $m_1^{Y_n}$ if $n$ is even.
 %$\ff^{(n)}\sim \besq_{\fm^{(n-1)}(Y_n-Y_{n-1})}(-1)$. Note that...
 Set $Y_{n+1}=Y_n+\zeta(\ff^{(n)})$. For $y\in(0,Y_{n+1}-Y_n]$, define
 \begin{align*}
  &\left(m_1^{Y_n+y},m_2^{Y_n+y},\alpha^{Y_n+y}\right):= \left\{\begin{array}{ll}
  (\fm^{(n)}(y),\ff^{(n)}(y),\gamma^{(n)}(y)), &\mbox{if $n$ is odd},\\[3pt]
  %\mbox{or}\quad&\left(m_1^{Y_n+y},m_2^{Y_n+y},\alpha^{Y_n+y}\right):=
  (\ff^{(n)}(y),\fm^{(n)}(y),\gamma^{(n)}(y)), &\mbox{if $n$ is even.}
  \end{array}\right.
 \end{align*}
 If, for some $n\ge 1$, $m_1^{Y_n}+m_2^{Y_n}=0$, set $(m_1^y,m_2^y,\alpha^y):=(0,0,\emptyset)$ for all $y>Y_n$ and $Y_{n+1} := \infty$. This time $Y_n$ is the \emph{lifetime} of the type-2 evolution. We also define the earlier time $D\in [Y_{n-1},Y_n)$ that corresponds to the degeneration of the type-1 evolution $(\fm^{(n)},\gamma^{(n)})$ at time $D-Y_{n-1}$. Equivalently, $D$ is the time at which either $m_1^y+\|\alpha^y\|$ or $m_2^y+\|\alpha^y\|$ hits zero and is absorbed. This is the \emph{degeneration time} of the type-2 evolution. 
\end{definition}

Though this fact is obscure in the above definition, the roles of the two top masses in a type-2 evolution are symmetric.

\begin{lemma}[Lemma 19 of \cite{Paper3}]\label{lem:type2:symm}
 If $\big(\big(m_1^y,m_2^y,\alpha^y\big),y\ge0\big)$ is a type-2 evolution then so is $\big(\big(m_2^y,m_1^y,\alpha^y\big),y\ge0\big)$. In particular, if $Y := \inf\big\{y>0\colon m_2^y = 0\big\}$ then given $\big(m_2^0,m_1^0,\alpha^0\big)$, the process $\big(m_2^y,y\in [0,Y]\big)$ is a $\BESQ_{m_2^0}(-1)$ stopped when it hits zero, while $\big(\big(m_1^y,\alpha^y\big),y\in [0,Y]\big)$ is conditionally distributed as an independent type-1 evolution stopped at time $Y$.
\end{lemma}

%Though they are not clear from the description above, the following propositions are easy consequences of the results cited below and the pathwise constructions of the type-0/1/2 evolutions.

\begin{proposition}[Theorem 1.4 of \cite{Paper1}, Theorem 2, Corollary 16 of \cite{Paper3}]\label{prop:012:pred}
 \begin{enumerate}
  \item Type-0/1/2 evolutions and IP-valued type-1 evolutions are Borel right Markov processes.\label{item:pred:Markov}
  \item Type-0 evolutions and IP-valued type-1 evolutions are path-continuous.\label{item:pred:0}
  %\item %If $((m^y,\alpha^y),y\ge0)$ is a type-1 evolution then $((0,m^y)\concat\alpha^y,y\ge0)$ is a diffusion. Also, 
  	%$m^y$ can only equal zero at times when $\alpha^y$ has no leftmost block.\label{item:pred:1}
  \item If $\big(\big(m_1^y,m_2^y,\alpha^y\big),y\ge0\big)$ is a type-2 evolution and $I(y) :\equiv \max\{n\ge 0\colon Y_n\le y\}\text{ mod }2$ is $\{1,2\}$-valued, where $(Y_n,n\ge0)$ is as in Definition \ref{def:type2}, then the \emph{IP-valued type-2 evolution} $\big(\big(0,m_{3-I(y)}^y\big)\concat \big(0,m_{I(y)}^y\big)\concat\alpha^y,y\ge0\big)$ is a diffusion. Also, %if $\alpha^y$ has a leftmost block for some $y\ge0$, then $m_1^y,m_2^y>0$. 
  	each of $m_1^y$ and $m_2^y$ can only equal zero when $\alpha^y$ has no leftmost block, and they can only both equal zero if $\alpha^y = \emptyset$.\label{item:pred:2}
 \end{enumerate}
 \noindent In particular type-0/1/2 evolutions are predictable Markov processes: the value of a type-0/1/2 evolution at any stopping time $Y$ is a (deterministic) function of its left limit at that time.
\end{proposition}

\begin{proposition}[Concatenation properties; Proposition 9 of \cite{Paper3}]\label{prop:012:concat}
 Let $((m^y,\alpha^y),y\ge0)$ denote a type-1 evolution.
 \begin{enumerate}[label=(\roman*),ref=(\roman*)]
  \item \label{item:012concat:BESQ+0}
  	Let $\zeta$ denote the first time that $m^y$ hits zero. Then $(m^y\cf\{y\le\zeta\},y\ge0)$ is a \BESQ[-1] and $(\alpha^y,y\in [0,\zeta])$ distributed as an independent type-0 evolution stopped at $\zeta$.
  \item \label{item:012concat:0+1}
  	If $(\widetilde\alpha^y,y\ge0)$ is an independent type-0 evolution then $(\widetilde\alpha^y\concat(0,m^y)\concat\alpha^y,y\ge0)$ is a type-0 evolution.
  \item \label{item:012concat:1+1}
  	Suppose instead that $((\widetilde m^y,\widetilde\alpha^y),y\ge0)$ is an independent type-1 evolution and let $\widetilde D$ denote its degeneration time. Then the following process is a type-1 evolution.
  \begin{equation}\label{eq:012concat:1+1}
   \left\{\begin{array}{ll}
    (\widetilde m^y,\widetilde\alpha^y\concat(0,m^y)\concat\alpha^y)	& \text{for }y\in [0,\widetilde D),\\
    (m^y,\alpha^y)	& \text{for }y\ge \widetilde D.
   \end{array}\right.
  \end{equation}
  \item \label{item:012concat:2+1}
  	Suppose instead that $((\widetilde m_1^y,\widetilde m_2^y,\widetilde\alpha^y),y\ge0)$ is an independent type-2 evolution. Let $\widetilde D$ denote its degeneration time. Let $(\widehat x_1,\widehat x_2)$ equal $(\widetilde m_1^{\widetilde D},m^{\widetilde D})$ if $m_2^{\widetilde D} = 0$ (i.e.\ if label 2 is the label that degenerates at time $\widetilde{D}$), or equal $(m^{\widetilde D},\widetilde m_2^{\widetilde D})$ otherwise (if label 1 degenerates). Let $((\widehat m_1^y,\widehat m_2^y,\widehat\alpha^y),y\ge0)$ be a type-2 evolution with initial state $(\widehat x_1,\widehat x_2, \alpha^{\widetilde D})$, conditionally independent of the other processes given its initial state. The following is a type-2 evolution:
  \begin{equation}\label{eq:012concat:2+1}
   \left\{\begin{array}{ll}
    (\widetilde m_1^y,\widetilde m_2^y,\widetilde\alpha^y\concat(0,m^y)\concat\alpha^y)	& \text{for }y\in [0,\widetilde D),\\
    (\widehat m_1^{y-\widetilde D},\widehat m_2^{y-\widetilde D},\widehat\alpha^{y-\widetilde D})	& \text{for }y\ge \widetilde D.
   \end{array}\right.
  \end{equation}
  %Then $((\widetilde m_1^y,\widetilde m_2^y,\widetilde\alpha^y\concat(0,m^y)\concat\alpha^y),y\in [0,\widetilde D))$ is a stopped type-2 evolution.
  %Then the process that equals $(\widetilde m^y,\widetilde\alpha^y\concat(0,m^y)\concat\alpha^y)$ up to time $\widetilde D$ and subsequently equals $(m^y,\alpha^y)$ is a type-1 evolution.
 \end{enumerate}
 Moreover, the concatenated evolutions constructed in \ref{item:012concat:0+1}, \ref{item:012concat:1+1}, and \ref{item:012concat:2+1} each possess the strong Markov property in the larger filtrations generated by their constituent parts.
\end{proposition}

\begin{remark}\label{rmk:decomp_ker}
 Note that in assertion \ref{item:012concat:2+1}, we may consider the conditional joint distribution of $((m^y,\alpha^y),y\in [0,\wt D))$ and $((\widetilde m_1^y,\widetilde m_2^y,\widetilde\alpha^y),y\in [0,\wt D))$ given the path of the concatenated process of display \eqref{eq:012concat:2+1} prior to time $\wt D$. Such a regular conditional distribution exists because $(\cI,d_{\cI})$ is Lusin \cite[Theorem 2.7]{Paper1} and these evolutions have c\`adl\`ag paths. We can do the same for the assertions \ref{item:012concat:0+1} and \ref{item:012concat:1+1} of Proposition \ref{prop:012:concat}. In each case, the independence of the two constituent evolutions is lost under this conditioning; it is only recovered after integrating over the law of the concatenated type-$i$ evolution.
\end{remark}
%As mentioned above, the type-0/1/2 evolutions are given pathwise constructions in \cite{Paper1,Paper3}. For the purpose of this paper, we denote by  $(\cN_i,\Sigma(\cN_i),\bP^i_\mu)$, $i=0,1,2$, the canonical probability space of \emph{type-$i$ data}, on which a type-$i$ evolution $(\Gamma^y,y\ge0)$ with $\Gamma^0\sim\mu$ can be constructed. This space comes equipped with a filtration, which we will denote by $(\cF^y_i,y\ge0)$, so that $(\Gamma^y,y\ge0)$ possess the strong Markov property in this filtration. The particulars of these spaces and the type-$i$ data objects are unimportant here; details of this formalism can be read from Section 1.3 and Definition 17 in \cite{Paper3} and Definitions 3.36, 4.16, and 4.20 in \cite{Paper1}. What is important is that on these spaces, we can carry out the decompositions noted in Proposition \ref{prop:012:concat}, decomposing about any block in the evolution, starting from any time.

Recall from \cite{Paper1} that a Poisson--Dirichlet interval partition with parameters $\big(\frac12,\frac12\big)$, called \PDIP[\frac12,\frac12], is an interval partition whose ranked block sizes have law $\PoiDir[\frac12,\frac12]$, with the blocks exchangeably ordered from left to right. Let $A\sim \BetaDist[\frac12,\frac12]$, $(A_1,A_2,A_3)\sim\distribfont{Dirichlet}\big(\frac12,\frac12,\frac12\big)$, and $\bar\beta\sim\PDIP[\frac12,\frac12]$ independent of each other. A probability distribution on $\cI$ is said to be a \emph{pseudo-stationary law for the type-0 evolution} if it is the law of $M\bar\beta$, i.e.\ $\bar\beta$ scaled by $M$, for an independent random mass $M>0$. Likewise a law on $[0,\infty)\times\cI$,
respectively $[0,\infty)^2\times\cI$, is a \emph{pseudo-stationary law for the type-1, {\rm resp.} type-2, evolution} if it is the law of any independent multiple of $(A,(1-A)\bar\beta)$, resp.\ $(A_1,A_2,A_3\bar\beta)$. This language is in reference to the following proposition.

\begin{proposition}[Theorem 6.1 of \cite{Paper1}, Proposition 33 of \cite{Paper3}]\label{prop:012:pseudo}
 For $i=0,1,2$, if a type-$i$ evolution has a pseudo-stationary initial distribution then, given that it does not degenerate prior to time $y$, its conditional law at time $y$ is also pseudo-stationary. In the special case that its initial mass has law $\GammaDist[\frac{1+i}{2},\gamma]$, then its mass at time $y$ has conditional law $\GammaDist[\frac{1+i}{2},\gamma/(2\gamma y+1)]$.
\end{proposition}

\begin{proposition}[Theorem 1.5 of \cite{Paper1}, Theorem 3 of \cite{Paper3}]\label{prop:012:mass}
 For $i=0,1,2$, the total mass process for a type-$i$ evolution is a \BESQ[1-i]. In particular, this total mass process is Markovian in the filtration of the type-$i$ evolution.
\end{proposition}

%\section{Tree shapes, $k$-trees, and evolutions}\label{sec:def_thm}
\section{Definitions of $k$-tree evolutions and statements of main results}\label{sec:def_thm}

\subsection{State spaces and killed $k$-tree evolutions}

A \emph{tree shape} is a rooted binary combinatorial tree with leaves labelled by a non-empty finite set $A\subset\bN$, with the convention that the root vertex has degree one. We refer to the root as the ancestor of the other vertices, and the edge incident to the root as the ancestor of other edges. 
We will use genealogical language, such as ``child/parent/sibling/uncle,'' to describe relations between vertices or between edges of the tree shape. %, while also refering to the ultimate ancestor of all other vertices as ``root,'' and the childless descendants as ``leaves.'' NOTE: possibly come back to this 
Think of each edge of a tree shape (and the branch point at the end of the edge farthest from the root) as being labeled by the set of labels of leaves that are separated from the root by that edge. This collection of edge labels specifies the tree shape. I.e.\ a tree shape can be represented as a collection $\mathcal{H}$ of subsets of $A$ with certain properties, including: $A\in\mathcal{H}$, representing the edge incident to the root, and $\{i\}\in\mathcal{H}$ for each $i\in A$, representing edges incident to the leaves. 
%
%A common representation of a tree shape is as a binary hierarchy on $A$, i.e.\ a collection of subsets of $A$, including $A$ itself, all singleton subsets of $A$ and for every non-singleton member $B$ the two parts of a unique binary partition of $B$. 
%
We use the related representation that omits the singletons. %In other words, we label every branch point (and its unique edge towards the root) by the set of labels that, when the branch point (edge) is removed, are in components other than the root component. Our terminology sees this as a set of \emph{edges} labelled by non-singletons, including the label $A$ for the edge from the root to the first branch point.
We denote the set of such (representations of) tree shapes by $\TShape_A$. %In this way, the formalism is such that
For example, 
$$\bT_{[3]}^{\rm shape}=\{\{[3],\{2,3\}\},\ \{[3],\{1,3\}\},\ \{[3],\{1,2\}\}\},\qquad\mbox{edge sets of }\ 
	\parbox{0.9cm}{\includegraphics[height=1cm]{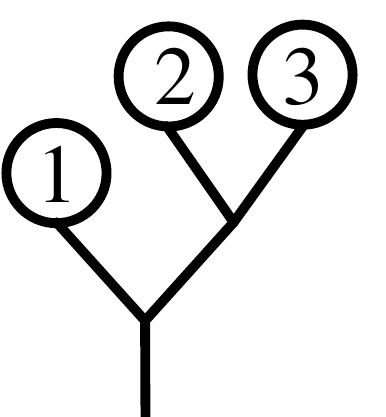}}\;,\;\;%
    \parbox{0.9cm}{\includegraphics[height=1cm]{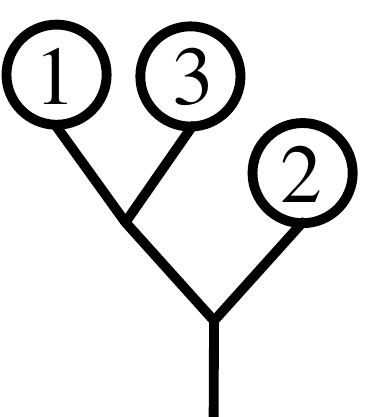}}\;,\;\;%
    \parbox{0.9cm}{\includegraphics[height=1cm]{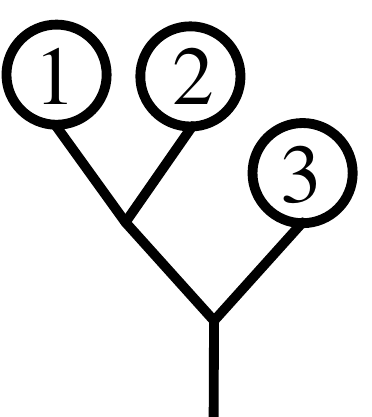}}.$$

Given a tree shape $\ft\in\TShape_A$, for $E\in \ft\cup \{\{i\}\colon i\in A\}$, $E\neq A$, we write $\parent{E}$ to denote the parent edge,
$$\parent{E} = \bigcap_{F\in \ft\colon E\subsetneq F}F.$$ 
We say an internal edge $E\in\ft$ is of type 0, 1, or 2, if 0, 1, or 2 of its children are leaf edges, respectively. For example, in the tree shape $\{[3],\{1,3\}\}$, the edge $\{1,3\}$ is of type 2 with leaves 1 and 3 as its children, while $[3]$ is of type 1 with leaf 2 as its child.

For a finite, non-empty set $A\subset\bN$, an \emph{$A$-tree} is a tree shape $\ft\in\bT_{A}^{\rm shape}$ equipped with non-negative weights on leaf edges and interval partitions marking the internal edges:
\begin{equation}
 \bTInt_{A}=\bigcup_{\ft\in\bT^{\rm shape}_{A}}\{\ft\}\times[0,\infty)^A\times\cI^\ft.
\end{equation}
For $k\ge1$ we refer to elements of $\bTInt_{k} := \bTInt_{[k]}$ as \emph{$k$-trees}. Consider $T=(\ft,(x_i,i\in A),(\beta_E,E\in\ft))$ $\in \bTInt_{A}$. We write $\|T\|=\sum_{i\in A}x_i+\sum_{E\in\ft}\|\beta_E\|$. Think of this representation in connection with Figure \ref{fig:B_k_tree_proj} and the description of the $k$-tree projection of a Brownian CRT in the introduction. The $x_i$ represent masses of subtrees corresponding to leaves of the tree represented by $\ft$, while the $\beta_E$ represent totally ordered collections of subtree masses. % along a path in the CRT corresponding to an internal edge of $\ft$. 
%The $x_i$ represent weights on the leaf-edges of (the tree shape represented by) $\ft$. Each interval partition $\beta_E$ represents a totally ordered collection of subtree masses branching off along a path within a rooted, weighted $\BR$-tree. See Figure \ref{fig:k_tree_proj}. 
In this interpretation, the intervals in $\beta_E$ that are closer to $0$ represent subtrees that are farther from the root of the CRT.

We refer to each top mass $x_i$, $i\in A$, and each interval in each of the partitions $\beta_E$, $E\in\ft$, as a \emph{block} of $T$. Formally, we denote the set of blocks by
\begin{equation}
 \block(\ft,(x_i,i\in A),(\beta_E,E\in\ft)) := A\cup\{(E,a,b)\colon E\in\ft,\,(a,b)\in \beta_E\}.
\end{equation}
We will write $\|\ell\|$ for the \emph{mass} of $\ell\in \block(T)$; i.e.\ for the top masses $\|\ell\|:=x_\ell$, $\ell\in A$, and for the other blocks $\|\ell\|=\|(E,a,b)\|:=b-a$. %, $\ell\in\{(E,a,b)\colon (a,b)\in\beta_E,E\in\ft\}$. 
Then $\sum_{\ell\in \block(T)}\|\ell\|=\|T\|$.

For each label set $A$ and each $\ft\in\TShape_A$, we topologize the set of $A$-trees with shape $\ft$ by the product over the topologies in the components. This can be metrized by setting
\begin{equation}\label{eq:ktree:metric_1}
 d_{\bT}(T,T') = \sum_{i\in A} |x_i-x'_i| + \sum_{E\in\ft}\dI(\beta_E,\beta'_E)
\end{equation}
for $T,T'\in\bTInt_A$ with shapes $\ft=\ft'$. Within the set of trees with a given label set $A$ and shape, there is a single $A$-tree of zero total mass; we topologize the space of all $A$-trees, for all finite label sets $A$, by identifying all of these trees of zero mass, thereby gluing these spaces together. This is metrized by
\begin{equation}\label{eq:ktree:metric_2}
 d_{\bT}(T,T') = \sum_{i\in A} x_i + \sum_{i\in A'}x'_i + \sum_{E\in\ft}\dI(\beta_E,\emptyset) + \sum_{E\in\ft'}\dI(\beta'_E,\emptyset)
\end{equation}
for $T\in\bTInt_A$, $T'\in\bTInt_{A'}$ with differing tree shapes. We note that $\dI(\beta,\emptyset) = \max\{\|\beta\|,\IPLT_\beta(\infty)\}$ for any $\beta\in\cI$.
%NOTE: Matthias says minimize this discussion & the following proposition, perhaps fold it into pf of Borel right Markov.

\begin{proposition}\label{prop:Lusin}
 $\big(\big(\bigcup_A \bTInt_A \big)\big/ \big\{T\in \bigcup_A \bTInt_A\colon \|T\|=0\big\},d_{\bT}\big)$ is a Lusin space.
\end{proposition}

\begin{proof}
 From \cite[Theorem 2.7]{Paper1}, $(\cI,\dI)$ is Lusin. Thus, so is the above gluing of countably many product topologies of this with the Euclidean topology.
\end{proof}

We are interested in $k$-tree-valued Markov processes that avoid certain degenerate states. For example, states with multiple zero top masses will be inaccessible by our evolutions. We also exclude states having a zero top mass with an empty partition on its parent edge. These latter states will arise as left limits but force jumps ``away from the boundary.'' Specifically,
\begin{align}
  \tdTInt_A &:= \left\{ T = (\ft,(x_i,i\in A),(\beta_E,E\in\ft))\in\bTInt_A\, \middle|
    \begin{array}{l}
      x_i+x_j>0\mbox{ for all }E=\{i,j\}\in\ft\\
      \mbox{and }x_i\!+\!\big\|\beta_{\parent{\{i\}}}\big\| \!=\! 0\mbox{ for at most one }i\!\in\! A
    \end{array}\!\!\right\}\!\notag\\ \label{eq:k_tree_spaces}
  \TInt_{A} &:= \left\{ T = \left(\ft,(x_i,i\in A),(\beta_E,E\in\ft)\right)\in\tdTInt_A \,\middle|\, 
  		x_i+\big\|\beta_{\parent{\{i\}}}\big\| > 0\mbox{ for all }i\in A.\right\}.
\end{align}
Let $I\colon\tdTInt_A\rightarrow A\cup\{\infty\}$ record $I(T)=i$ if $x_i+\|\beta_{\parent{\{i\}}}\| = 0$ and set $I(T)=\infty$ if $T\in\TInt_A$. In the former case, we say that \emph{label $i$ is degenerate in $T$}.

Because we will only ever consider single leaf trees in the case where the leaf has label 1, we take the convention that $\tdTInt_{1} = [0,\infty)$ and $\TInt_{1} = (0,\infty)$, with this real number representing the mass on the leaf 1 component, which is then the total mass of the tree. We define $\TInt_\emptyset = \{0\}$.

As noted above Proposition \ref{prop:Lusin}, we identify all trees of zero mass. We take the convention of writing $0$ to denote such a tree. 
Furthermore, let $\partial\notin \bigcup_A\bTInt_A$ denote an isolated cemetary state.

\begin{definition}[Killed $A$-tree evolution]\label{def:killed_ktree}
 Consider an $A$-tree $T = (\ft,(m^0_i,i\in A),(\alpha^0_E,E\in\ft))$ $\in\TInt_A$ for some finite $A\subset\bN$ with $\#A\ge 2$.  
 \begin{itemize}
  \item For each type-2 edge $E = \{i,j\}\in \ft$, let $((m_i^y,m_j^y,\alpha_E^y),y\ge0)$ denote a type-2 evolution from initial state $(m_i^0,m_j^0,\alpha_E^0)$, and let $D_E$ denote its degeneration time.
  \item For each type-1 edge $E = \parent{\{i\}}\in \ft$, let $((m_i^y,\alpha_E^y),y\ge0)$ denote a type-1 evolution from initial state $(m_i^0,\alpha_E^0)$, and let $D_E$ denote its absorption time in $(0,\emptyset)$.
  \item For each type-0 edge $E\in\ft$, let $(\alpha_E^y,y\ge0)$ denote a type-0 evolution from initial state $\alpha^0_E$ and define $D_E=\infty$.
 \end{itemize}
 We take these evolutions to be jointly independent. Let $D = \min_{E\in\ft}D_E$. Define $\cT^y = (\ft,(m_i^y,i\in A),(\alpha_E^y,E\in\ft))$ for $y\in [0,D)$ and $\cT^y = \partial$ for $y\ge D$. This is the \emph{killed $A$-tree evolution from initial state $T$}. We call $D$ the \emph{degeneration time} of the evolution.
 
For $\#A=1$, define $(\cT^y)$ to be a \BESQ[-1] starting from $T$, killed upon hitting zero.
\end{definition}

In light of this construction, in a $k$-tree, we refer to each type-2 edge partition together with its two top masses, $(x_i,x_j,\beta_{\{i,j\}})$, as a \emph{type-2 compound}. Likewise, for a type-1 edge $E = \parent{\{i\}}$, we call $(x_i,\beta_E)$ a \emph{type-1 compound}, and for each type-0 edge $F$, the partition $\beta_F$ is a \emph{type-0 compound}. In Figure \ref{fig:B_k_tree_proj}, $\beta_{[5]}^{(5)}$ is a type-0 compound, $\big(X_2^{(5)},\beta_{\{1,2,4\}}^{(5)}\big)$ is a type-1 compound, and $\big(X_3^{(5)},X_5^{(5)},\beta_{\{3,5\}}^{(5)}\big)$ and $\big(X_1^{(5)},X_4^{(5)},\beta_{\{1,4\}}^{(5)}\big)$ are type-2 compounds.%180917

%\texttt{NOTE: need alternative language for ``components'' -- sometimes use this term to refer to individual top masses or spinal partitions; sometimes to refer to the collection of top-masses and spinal partitions that collectively evolve as a type-$i$ evolution between degen times}

\subsection{Label swapping and non-resampling $k$-tree evolution}\label{sec:non_resamp_def}

In the theory of Borel right Markov processes, \emph{branch states} are states that are not visited by the right-continuous Markov process but may be attained as a left limit, triggering an instantaneous jump. We will now define non-resampling $k$-tree evolutions with branch states in $\tdTInt_A\setminus\TInt_A$. When a type-1 or type-2 compound in an $A$-tree degenerates, we project this compound down and the evolution proceeds with one fewer leaf label.

%In the theory of Borel right Markov processes, \emph{branch states} are states that are not visited by the right-continuous Markov process but may be attained as a left limit, triggering an instantaneous jump. We will now define non-resampling and resampling $k$-tree evolutions with branch states in $\tdTInt_A\setminus\TInt_A$. In the non-resampling evolution, when a component of the tree degenerates, we apply the swap-and-reduce map $\varrho$ to drop the degenerate component, and the evolution proceeds with one less leaf label. In the resampling evolution at degeneration times we first apply $\varrho$ and then jump into a random state according to the resampling kernel $\Lambda$, which reinserts the label lost in degeneration, so that the evolution always retains all $k$ labels.

However, in \cite{Paper2} we found that in the discrete regime, in order to construct a family of projectively consistent Markov processes, it was necessary to have degenerate labels sometimes swap places with other nearby, higher labels before dropping the degenerate component and its label with it. The following two definitions %describe this mechanism.
lead to an analogous construction in the present setting. The role of this mechanism in preserving consistency will be evident in the proof of Proposition \ref{prop:Dynkin:killed}.

\begin{figure}
 \centering
 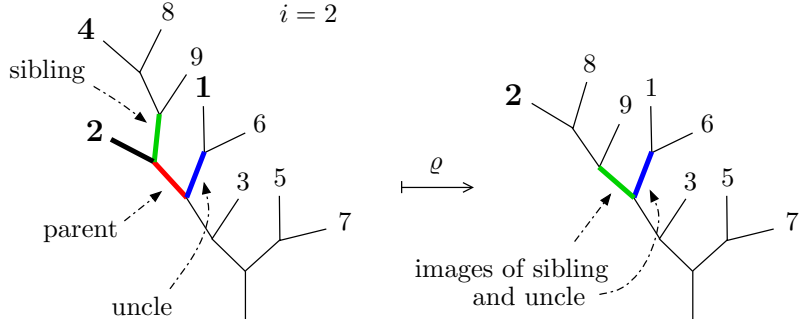
 \caption{Example of the swap-and-reduce map on a tree shape. Least labels in the two subtrees descended from sibling and uncle of leaf edge $\{i\}$ are shown in bold.\label{fig:label_swap}}
\end{figure}
\medskip

\textbf{Swap-and-reduce map for tree shapes. }
%\begin{definition}[Swap-and-reduce map for tree shapes]\label{def:shape:swapred}
 Consider a tree shape $\ft\in\TShape_A$ on some label set with $\#A\ge2$ and label $i\in A$. Let $J(\ft,i) = \max\{i,a,b\}$ where $a$ and $b$ are, respectively, the smallest labels of sibling and uncle of leaf edge $\{i\}$ in $\ft$. In the special case that its parent $\parent{\{i\}}$ is $A$, in which case $\{i\}$ has no uncle, we define $b=0$. In the example in Figure \ref{fig:label_swap},
 $$\ft = \big\{\, [9],\, \{5,7\},\, \{1,2,3,4,6,8,9\},\, \{1,2,4,6,8,9\},\, \{1,6\},\, \{2,4,8,9\},\, \{4,8,9\},\, \{4,8\}  \,\}.$$
 Leaf edge $\{2\}$ has sibling $\{4,8,9\}$ and uncle $\{1,6\}$, so $a=4$, $b=1$, $J(\ft,2) = \max\{2,4,1\} = 4$.
 
 We define a \emph{swap-and-reduce map on tree shapes}, $\widetilde\varrho\colon \TShape_A\times A \to \bigcup_{a\in A}\TShape_{A\setminus\{a\}}$ mapping $(\ft,i)$ to the tree shape $\ft'$ obtained from $\ft$ by first swapping labels $i$ and $j=J(\ft,i)$, then deleting the leaf subsequently labeled $j$ and contracting away its parent branch point. Formally, $\ft'$ is the image of $\ft\setminus\big\{\parent{\{i\}}\big\}$ under the map $\phi_{\ft,i}$ that modifies label sets $E\in\ft$ by first deleting label $i$ from the sets, and then replacing label $j$ by $i$. 
 %In general, $\phi$ is a two-to-one map on the power set of $A$, where $\phi(B) = \phi(B\cup\{i\})$ for any $B\subseteq A$. This means that for $\ft\in\TShape_A$, the parent and sibling of $\{i\}$ have the same image, but the restriction of $\phi$ to $\ft\setminus\big\{\parent{\{i\}}\big\}$ is injective. 
 %So, for example, if $i=2$ and $j = 4$, then
 %$$\{4,8,9\} \mapsto \{2,8,9\},\quad  \{2,3\}\mapsto \{3\} ,\quad \{2,4,6,9\}\mapsto \{2,6,9\} ,\quad \{1,6\}\mapsto \{1,6\}.$$
 In the example in Figure \ref{fig:label_swap}, $\phi_{\ft,2}$ is
 \newcommand{\mapsdown}{\rotatebox[origin=c]{-90}{\ensuremath{\mapsto}}}
 \newcommand{\pcom}{\hphantom{,}}
 \definecolor{darkblue}{rgb}{0.25,0.25,0.7}
 \definecolor{mygreen}{rgb}{0,0.75,0}
 \definecolor{gold}{rgb}{.8,0.7,0}
 $$
  \begin{array}{r@{}c@{}c@{}c@{}c@{}c@{}c@{}c@{}c@{}l}
   \ft = \big\{ & \pcom[9], & \pcom\{5,7\}, & \pcom\{1,2,3,4,6,8,9\}, & \pcom\{1,2,4,6,8,9\}, & \pcom\textcolor{blue}{\{1,6\}}, & \pcom\textcolor{red}{\{2,4,8,9\}}, & \pcom\textcolor{mygreen}{\{4,8,9\}}, & \pcom\{4,8\}\pcom&\big\}\\
   	& \mapsdown & \mapsdown & \mapsdown & \mapsdown & \mapsdown &  & \mapsdown & \mapsdown & \\
   \ft' = \big\{& \pcom[9]\!\setminus\!\{4\}, & \pcom\{5,7\}, & \pcom\{1,2,3,6,8,9\}, & \pcom\{1,2,6,8,9\}, & \pcom\textcolor{blue}{\{1,6\}}, & & \pcom\textcolor{mygreen}{\{2,8,9\}}, & \pcom\{2,8\}\pcom &\big\}.
  \end{array}
 $$
%\end{definition}

Note that in the preceding definition, $\phi_{\ft,i}(E_1) = \phi_{\ft,i}(E_2)$ if and only if $E_1\setminus\{i\} = E_2\setminus\{i\}$. But the only distinct edges $E_1\neq E_2$ in $\ft$ with this relationship are the sibling and parent of leaf edge $\{i\}$. Thus, by excluding $\parent{\{i\}}$ from its domain, we render $\phi_{\ft,i}$ injective and ensure that the range of this map is an element of $\TShape_{A\setminus\{J(\ft,i)\}}$.

The swap-and-reduce map on tree shapes induces a corresponding map for degenerate $A$-trees, where labels are swapped and the degenerate component is projected away, but everything else remains unchanged.
\medskip

\textbf{Swap-and-reduce map for $A$-trees. }
%\begin{definition}[Swap-and-reduce map for $A$-trees]\label{def:ktree:swapred}
 Consider $T = (\ft,(x_a,a\in A),(\beta_E,E\in\ft))\in\tdTInt_A\setminus\TInt_A$. Recall that for such an $A$-tree, $I(T)$ denotes the unique index $i\in A$ for which $x_i+\big\|\beta_{\parent{\{i\}}}\big\| = 0$. We define $J\colon\tdTInt_A\setminus\TInt_A\rightarrow A$ by  
 $J(T) = J(\ft,I(T))$, as defined above. The \emph{swap-and-reduce map on $A$-trees} is the map
 $$\varrho\colon \tdTInt_A\setminus \TInt_A \rightarrow \bigcup_{j\in A}\TInt_{A\setminus\{j\}}$$
 that sends $T$ to 
 $(\widetilde\varrho(T),(x'_a, a\in A\setminus\{J(T)\}),(\beta'_E, E\in \widetilde\varrho(T)))$ where: (1) $x'_a = x_a$ for $a\neq I(T)$, (2) $x'_{I(T)} = x_{J(T)}$ if $I(T)\neq J(T)$, and (3) $\beta'_{E} = \beta_{\phi_{\ft,I(T)}^{-1}(E)}$ for each $E\in \widetilde\varrho(T)$, where $\phi_{\ft,I(T)}$ is the injective map defined above.%in Definition \ref{def:shape:swapred}. %Following \cite{Paper2}, 
%\end{definition}

\begin{definition}[Non-resampling $k$-tree evolution]\label{def:nonresamp_1}
 Set $A_1 = [k]$ and fix some $\cT^0_{(1)} =T\in \TInt_A$. Inductively for $1\le n<k$, let $(\cT^y_{(n)},y\in [0,\Delta_n))$ denote a killed $A_{n}$-tree evolution from initial state $\cT_{(n)}^0$, run until its degeneration time $\Delta_n$, conditionally independent of $(\cT_{(j)},j < n)$ given its initial state. If $n<k-1$, we then set $A_{n+1} = A_n\setminus \{J(\cT_{(n)}^{\Delta_n-})\}$ and let $\cT_{(n+1)}^0 = \varrho(\cT_{(n)}^{\Delta_n-})$.
 
 For $1\le n\le k$ we define $D_n = \sum_{j=1}^n\Delta_j$ and set $D_0=0$. For $y\in [D_{n-1},D_{n})$ we define $\cT^y = \cT_{(n)}^{y-D_{n-1}}$. For $y\ge D_k$ we set $\cT^y = 0 \in \TInt_{\emptyset}$. Then $(\cT^y,y\ge0)$ is a \emph{non-resampling $k$-tree evolution from initial state $T$}. We say that at each time $D_n$, label $I(\cT^{D_n-})$ has \emph{caused degeneration} and label $J(\cT^{D_n-})$ is \emph{dropped in degeneration}.
\end{definition}

\subsection{Resampling $k$-tree evolutions and de-Poissonization}\label{sec:resamp_def}

We now define a resampling $k$-tree evolution in which at degeneration times we first apply $\varrho$ and then jump into a random state according to a resampling kernel, which reinserts the label lost in degeneration, so that the evolution always retains all $k$ labels.
\medskip

\textbf{Label insertion operator $\oplus$. }%\label{def:insertion}
 \emph{For tree shapes.} Consider $\ft\in\TShape_A$. Given an edge $F\in\ft\cup\{\{a\}\colon a\in A\}$, we define $\ft\oplus (F,j)$ to be the tree shape with labels $A\cup\{j\}$ formed by replacing edge $F$ by a path of length 2, and inserting label $j$ as a child of the new branch point in the middle of the path. Formally, for each $E\in\ft$ we define (a) $\phi(E)=E\cup\{j\}$ if $F\subsetneq E$ and (b) $\phi(E)=E$ otherwise. Then $\ft\oplus (F,j)$ equals $\phi(\ft)\cup \{F\cup\{j\}\}$.
 
 \emph{For $A$-trees.} Consider an $A$-tree $T = (\ft,(x_m,m\in A),(\beta_E,E\in\ft))$, a label $i\in A$, and a 2-tree 
 %$U = (y_i,y_j,\gamma)$ with $\|U\|=1$. We define $T\oplus (i,j,U)$ to be the $A\cup\{j\}$ tree formed by replacing the leaf block $i$ with weight $x_i$ by the rescaled 2-tree $x_i\odot U$, thereby introducing label $j$. Formally, 
 $U = (y_1,y_2,\gamma)\in \TInt_{2}$ with $\|U\| = 1$, where we have dropped the tree shape because all elements of $\TInt_{2}$ have the same shape. We define $T\oplus (i,j,U)$ to be the $(A\cup\{j\})$-tree formed by replacing the leaf block $i$ and its weight $x_i$ by the rescaled 2-tree in which label $i$ gets weight $x_iy_1$, a new label $j$ gets weight $x_iy_2$, and their new parent edge bears partition $x_i\gamma$. 
 %$x_i\odot U$, thereby introducing label $j$. 
 Formally, 
  $T\oplus (i,j,U) = (\ft\oplus(\{i\},j),(x'_m,m\in A\cup\{j\}),(\beta'_E,E\in\ft\oplus(\{i\},j)))$, where: (i) $(x'_i,x'_j,\beta'_{\{i,j\}}) = x_i U$, (ii) $x'_m=x_m$ for $m\notin \{i,j\}$, and (iii) $\beta'_E = \beta_{\phi^{-1}(E)}$ for $E\neq \{i,j\}$, where $\phi$ is as above.
 
 Now consider a block $\ell = (F,a,b)\in \block(T)$. This block splits $\beta_F$ into $\beta_{F,0}\concat (0,b-a)\concat \beta_{F,1}$. We define $T\oplus (\ell,j,U)$ to be the $A\cup\{j\}$ tree formed by inserting label $j$ into block $\ell$. In this definition, $U$ is redundant. Formally, $T\oplus (\ell,j,U) = (\ft\oplus (F,j),(x'_m,m\in A\cup\{j\}),(\beta'_E,E\in\ft\oplus(F,j)))$, where: (i) $x'_m=x_m$ for $m\neq j$, (ii) $\beta'_E = \beta_{\phi^{-1}(E)}$ for $E\notin \{F,F\cup\{j\}\}$, and (iii) $(\beta'_F,x'_j,\beta'_{F\cup\{j\}}) = (\beta_{F,0},b-a,\beta_{F,1})$.
\medskip

\textbf{Resampling kernel for $A$-trees. }%\label{def:ktree:resamp}
 For finite non-empty $A\subset\bN$ and $j\in\bN\setminus A$, we define the \emph{resampling kernel} as the distribution of the tree obtained by inserting label $j$ into a block chosen at random according to the masses of blocks and, if the chosen block is a top mass $x_i$, then replacing the block by a rescaled Brownian reduced $2$-tree. More formally, we define a kernel $\Lambda_{j,A}$ from $\TInt_A$ to $\TInt_{A\cup\{j\}}$ by
 \begin{equation}
  \int_{T^\prime\in\bT_{A\cup\{j\}}^{\rm int}}\varphi(T^\prime)\Lambda_{j,A}(T,dT^\prime)=\sum_{\ell\in \block(T)}\frac{\|\ell\|}{\|T\|}\int_{U\in\TInt_{2}}\varphi(T\oplus(\ell,j,U))Q(dU),
 \end{equation}
 where $Q$ denotes the distribution of a Brownian reduced $2$-tree with leaf labels $\{1,2\}$, as defined in the introduction.
%\end{definition}

In \eqref{eq:B_ktree_resamp}, we describe how these resampling kernels generate a Brownian reduced $k$-tree.

\begin{definition}[Resampling $k$-tree evolution]\label{def:resamp_1}
 Fix some $\cT^0_{(1)} =T\in \TInt_{k}$. Inductively for $n\ge1$, let $(\cT^y_{(n)},y\in [0,\Delta_n))$ denote a killed $k$-tree evolution from initial state $\cT_{(n)}^0$, run until its degeneration time $\Delta_n$, conditionally independent of $(\cT_{(j)},j< n)$ given its initial state. We define $\cT_{(n+1)}^0$ to have conditional distribution $\Lambda_{J_n,[k]\setminus\{J_n\}}\big(\varrho(\cT_{(n)}^{\Delta_n-}),\cdot\,\big)$ given $(\cT_{(j)},j\le n)$, where $J_n = J(\cT_{(n)}^{\Delta_n-})$.
 
 We define $D_n = \sum_{j=1}^n\Delta_j$ and set $D_0=0$. For $y\in [D_{n-1},D_{n})$ we define $\cT^y = \cT_{(n)}^{y-D_{n-1}}$. For $y\!\ge\! D_\infty:=\sup_{n\ge 0}D_n$ we set $\cT^y\! =\! 0\! \in\! \TInt_{\emptyset}$. Then $(\cT^y,y\!\ge\!0)$ is a \emph{resampling $k$-tree evolution with initial state $T$}.
\end{definition}

\begin{theorem}\label{thm:Markov}
 Killed $A$-tree evolutions and non-resampling and resampling $k$-tree evolutions are Borel right Markov processes but are not quasi-left-continuous; thus they are not Hunt processes.
\end{theorem}

\begin{proposition}\label{prop:non_accumulation}
 For resampling $k$-tree evolutions with degeneration times $D_n$, $n\ge 1$, the limit $D_\infty = \lim_{n\rightarrow\infty}D_n$ equals $\inf\{y\ge0\colon \|\cT^{y-}\|=0\}$, and this is a.s.\ finite.
\end{proposition}

\begin{theorem}\label{thm:total_mass}
 For a non-resampling or resampling $k$-tree evolution $(\cT^y,y\ge0)$ with initial state with mass $\|\cT^0\| = m$, the total mass process $(\|\cT^y\|,y\ge0)$ has law $\BESQ_m(-1)$.
\end{theorem}

We prove Theorem \ref{thm:Markov} and a partial form of Theorem \ref{thm:total_mass} at the start of Section \ref{sec:k_fixed}. In particular, we prove the assertion of Theorem \ref{thm:total_mass} for non-resampling evolutions, and we reduce the resampling case to Proposition \ref{prop:non_accumulation}, which we then prove in Section \ref{sec:non_acc} and Appendix \ref{sec:non_acc_2}.

\begin{proposition}\label{prop:pseudo:resamp}
 Let $(\cT^y,y\ge 0)$ be a resampling $k$-tree evolution starting from a scaled Brownian reduced $k$-tree of total mass $M$, and let $B\sim\besq_M(-1)$. Then at any fixed time $y\ge0$, $\cT^y$ has the distribution of a scaled Brownian reduced $k$-tree of mass $B(y)$.  
\end{proposition}

In light of this result, we refer to the laws of independently scaled Brownian reduced $k$-trees as the \emph{pseudo-stationary laws for resampling $k$-tree evolutions.} We prove this proposition in Section \ref{sec:pseudo}.

%The resampling case follows from the same argument and Proposition \ref{prop:non_accumulation}, which we prove in Section \ref{sec:non_acc}.

\subsection{De-Poissonized $k$-tree evolutions}

Given a \cadlag\ path $\fT = (T^y,y\!\ge\!0)$ in $\bigcup_A\tdTInt_{A}$, let
\begin{equation}\label{eq:dePoi:time_change}
 \rho(\fT)\colon[0,\infty)\rightarrow[0,\infty],\qquad\rho_u(\fT) := \inf\left\{y\ge 0\colon\int_0^y\|T^z\|^{-1}dz>u\right\},\quad u\ge 0.
\end{equation}
This process is continuous and strictly increasing until it is possibly absorbed at $\infty$. If the total mass process $(\|T^y\|,y\ge0)$ evolves as a $\besq(-1)$, as in Theorem \ref{thm:total_mass}, then $\rho(\fT)$ is bijective onto $[0,\zeta)$, where $\zeta = \inf\{y\ge0\colon \|T^y\| = 0\}$ is a.s.\ finite; see e.g.\ \cite[p.\ 314-5]{GoinYor03}. 

Let $\TInt_{k,1} := \big\{T\in\TInt_{k}\colon \|T\| = 1\big\}$.

\begin{definition}\label{def:dePoi}
 Let $\fT = (\cT^y,y\ge0)$ denote a resampling (respectively, non-resampling) $k$-tree evolution from initial state $T\in\TInt_{k,1}$. Then
 $$\widebar\cT^u := \big\|\cT^{\rho_u(\fT)}\big\|^{-1} \cT^{\rho_u(\fT)},\quad u\ge0$$
 is a \emph{de-Poissonized resampling} (resp.\ \emph{non-resampling}) \emph{$k$-tree evolution from initial state $T$}.
 %
 %For any law $\widebar\mu$ on $\TInt_{k,1}$, let $\widebar{\bP}_{\widebar{\mu}}^{\,k,+}$ denote the law on $\bD([0,\infty),\TInt_{k,1})$ of a de-Poissonized resampling $k$-tree evolution with initial distribution $\widebar\mu$. 
\end{definition}

\begin{theorem}\label{thm:dePoi}
 De-Poissonized resampling and de-Poissonized non-resampling $k$-tree evolutions are Borel right Markov processes. The former are stationary with the laws of the Brownian reduced $k$-trees. The latter are eventually absorbed at the state $1\in\TInt_{1,1}$ of the tree whose only leaf is $1$, carrying unit weight.
\end{theorem}

Analogous de-Poissonization results have been given for type-0/1/2 evolutions in \cite[Theorem 1.6]{Paper1} and \cite[Theorem 4]{Paper3}. The results stated in Section \ref{sec:resamp_def} suffice to prove this Theorem by the same method.

\begin{proof}
 The right-continuity of sample paths is preserved by both the time change and the normalization. The rest of the proof of \cite[Theorem 1.6]{Paper1}, including the auxiliary results, from Proposition 6.7 to Theorem 6.9 of that paper, are easily adapted, with the Markov property, total mass, and pseudostationarity results of Theorems \ref{thm:Markov} and \ref{thm:total_mass} and Proposition \ref{prop:pseudo:resamp} serving in place of Proposition 5.20 and Theorems 1.5 and 6.1 in \cite{Paper1}.
\end{proof}

We also obtain the following result for de-Poissonised resampling $k$-tree evolutions.

\begin{corollary}\label{cor:WF}
Let $\widebar \fT = (\widebar \cT^u,u\!\ge\!0)= ((\widebar \ft^u_k,(\widebar X^u_j ,j\!\in\![k]),(\widebar\beta^u_E,E\!\in\!\widebar\ft^u_k)), u\!\geq\! 0)$ denote a de-Poissonised resampling $k$-tree evolution from initial state $T= (\ft_k,(X_j,j\!\in\![k]),(\beta_E,E\!\in\!\ft_k))$ $\in\TInt_{k,1}$, and let $\tau$ be the first time either a top mass or an interval partition has mass $0$.  Observe that $\tau \leq \widebar D_1$, where $\widebar D_1$ is the first time $\widebar \fT$ resamples, so for $u<\tau$, $\widebar \ft^u_k=\ft_k$.  Then $((\widebar X^{u/4}_j,j\!\in\![k]),(\|\widebar\beta^{u/4}_E\|,E\!\in\!\ft_k)), 0\leq u <\tau)$ is a Wright--Fisher diffusion, killed when one of the coordinates vanishes, with parameters $-1/2$ for coordinates corresponding to top masses and $1/2$ for coordinates corresponding to masses of interval partitions.
\end{corollary}

\begin{proof}
 Let $\fT$ be a resampling $k$-tree evolution started from $T$.  By Propositions \ref{prop:012:pred} and \ref{prop:012:concat}, up until the first time a top mass or the mass of an interval partition is zero, the top masses evolve as \BESQ[-1] processes and the masses of internal interval partitions evolve as \BESQ[1] processes, and all of these are independent.  The effect of de-Poissonisation procedure in Definition \ref{def:dePoi} on these evolutions is identical to the de-Poissonisation procedure \cite{Pal11, Pal13} used to construct Wright--Fisher diffusions, and the result follows.\qedhere
\end{proof}

%After we have established some of the other properties of $k$-tree evolutions, the above theorem can be proved in the same manner as the analogous result for type-0/1 evolutions, \cite[Theorem 1.6]{Paper1}. We discuss this analogy in detail in Section \ref{sec:dePoi}.

\subsection{Projective consistency results}

\begin{definition}[Projection maps for $A$-trees]\label{def:proj}
 For $j\in\bN$ and $A\subseteq\bN$ with $\#(A\setminus\{j\})\in [1,\infty)$, we define a projection map
 $$\pi_{-j}\colon \TInt_A\to\TInt_{A\setminus\{j\}}$$
 %\bigcup_{A\subseteq\bN\colon 1\le \#A\setminus\{j\}<\infty}\TInt_A \to \bigcup_{A\subseteq\bN\setminus\{j\} \colon 1\le \#A<\infty}\TInt_A$$
 to remove label $j$ from an $A$-tree, as follows. Let $T = (\ft,(x_i,i\in A),(\beta_E,E\in\ft))\in \TInt_A$. If $j\notin A$ then $\pi_{-j}(T) = T$. Otherwise, let $\phi$ denote the map $E\mapsto E\setminus\{j\}$ for $E\in\ft\setminus\big\{\parent{\{j\}}\big\}$. As noted for a similar map in Section \ref{sec:non_resamp_def}, this map is injective. Then $\pi_{-j}(T) := (\ft',(x'_i,i\in A\setminus\{j\}),(\beta'_E,E\in\ft'))$, where
 \begin{enumerate}[label=(\roman*),ref=(\roman*)]
  \item $\displaystyle\ft' = \phi(\ft) = \big\{E\setminus\{j\}\colon E\in \ft\setminus\big\{\parent{\{j\}}\big\}\big\}$,
  \item if $E = \parent{\{j\}}$ is a type-1 edge in $\ft$ then $\beta'_{E\setminus\{j\}} = \beta_{E\setminus\{j\}}\concat (0,x_j)\concat\beta_E$,\label{item:proj:merge}
  \item if $\parent{\{j\}} = \{a,j\}$ is a type-2 edge in $\ft$ then $x_a' = x_a+x_j+\|\beta_{\{a,j\}}\|$, \label{item:proj:add}
  \item if $i\in A\setminus\{j\}$ is \emph{not} the sibling of $\{j\}$ in $\ft$, then $x'_i = x_i$, and
  \item if $E\in \ft'$ is \emph{not} the sibling of $\{j\}$ in $\ft$, then $\beta'_E = \beta_{\phi^{-1}(E)}$.
 \end{enumerate}
 
 For $k\ge1$ and any finite $A\subseteq\bN$ with $A\cap[k]\neq\emptyset$, we define $\pi_k\colon \TInt_A\to\TInt_{A\cap [k]}$ to be the composition $\pi_{-(k+1)}\circ\pi_{-(k+2)}\circ\cdots\circ\pi_{-\max(A)}$. It is straightforward to check that this composition of projection maps commutes.
\end{definition}

These projections are illustrated in Figure \ref{fig:k_tree_proj}. Note how, in that example, in passing from $T$ to $\pi_5(T)$, the condition of item \ref{item:proj:merge} of the above definition applies, whereas in passing from $\pi_5(T)$ to $\pi_4(T)$, item \ref{item:proj:add} applies.

\begin{figure}
 \centering
 \scalebox{.95}{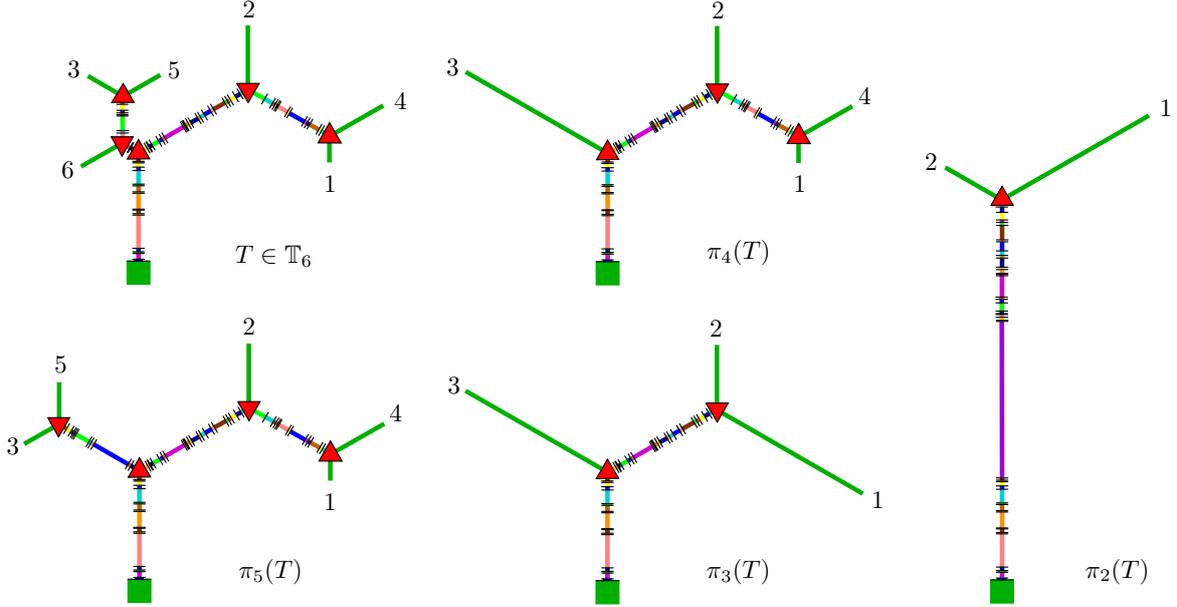}
 \caption{Projections of a $6$-tree.\label{fig:k_tree_proj}}
\end{figure}

\begin{theorem}\label{thm:consistency}
 \begin{enumerate}[label=(\roman*),ref=(\roman*)]
  \item \label{item:cnst:nonresamp}
    Let $2\le j<k$. For $(\cT^y,y\ge0)$ any non-resampling $k$-tree evolution, $(\pi_j(\cT^y),y\ge0)$ is a non-resampling $j$-tree evolution.
  \item \label{item:cnst:resamp}
    If $(\cT^y,y\ge0)$ is a resampling $k$-tree evolution and $\cT^0$ satisfies
   \begin{equation}\label{eq:cnst:init}
    \bE [\varphi(\cT^0)] = \int_{(T_i)_{i=j+1}^k \in \TInt_{j+1}\times\cdots\times \TInt_{k}}\Lambda_{j+1,[j]}(T_j,dT_{j+1})\cdots\Lambda_{k,[k-1]}(T_{k-1},dT_k)\varphi(T_k)
   \end{equation}
   for some $T_j\in \TInt_{j}$, then $(\pi_j(\cT^y),y\ge0)$ is a resampling $j$-tree evolution.
  \item \label{item:cnst:dePoi} These same results hold for de-Poissonized versions of these processes.
 \end{enumerate}
\end{theorem}

For $j\ge1$, we say that $(T_k,k\ge j)$ is a consistent family of $k$-trees if $\pi_{k-1}(T_k)=T_{k-1}$ for all $k > j$. We say that a family of $k$-tree evolutions $(\cT_k^y,y\ge 0)$, $k\ge j$, is consistent if $(\cT_k^y,k\ge j)$ is consistent for each $y\ge 0$. This next result follows from Theorem \ref{thm:consistency} by the Kolmogorov consistency theorem.

\begin{corollary}\label{cor:consistent_fam}
 \begin{enumerate}[label=(\roman*), ref=(\roman*)]
    \item For every consistent family $T_k\in\TInt_{k}$, $k\ge 1$, there are consistent families of non-resampling $k$-tree evolutions $(\cT_k^{y},y\ge 0)$, $k\ge 1$, with $\cT_k^0 = T_k$ for each $k$.\label{item:cnstfam:nonresamp}
    \item For any fixed $j\ge 1$ and $T\in\TInt_{j}$, there exists a consistent family of resampling $k$-tree evolutions $(\cT_k^y,y\ge0)$, $k\ge j$, $\cT_j^0 = T$.\label{item:cnstfam:resamp}
    \item Assertions \ref{item:cnstfam:nonresamp} and \ref{item:cnst:resamp} hold for de-Poissonized versions of these processes; in particular, there is a consistent family of stationary de-Poissonised resampling $k$-tree evolutions, $(\widebar{\cT}_k^{u},u\ge 0)$, $k\ge 1$.\label{item:cnstfam:dePoi}
  \end{enumerate}
\end{corollary}

We note one more consistency result. For the following result, we require notation $\pi_A$ generalising Definition \ref{def:proj} of $\pi_k$ to a general finite non-empty label set $A\subset \bN$. This is defined in the obvious, analogous manner.

\begin{proposition}\label{prop:resamp_to_non}
 Suppose $(\cT^y,y\ge0)$ is a resampling $k$-tree evolution. Then there exists a process $((A_y,B_y,\sigma_y),y\ge0)$ that is constant between degeneration times, such that $A_y$ and $B_y$ are subsets of $[k]$ and $\sigma_y$ is a bijection between them, and such that $\sigma_y\circ\pi_{A_y}(\cT^y)\in\TInt_{B^y}$, $y\ge0$, is a non-resampling $k$-tree evolution.
\end{proposition}

In Section \ref{sec:consistency}, we prove that the consistency of Theorem \ref{thm:consistency}\ref{item:cnst:resamp} holds until time $D_\infty$. We apply this result in Section \ref{sec:non_acc} to prove Proposition \ref{prop:non_accumulation}, which in turn completes the proof of Theorem \ref{thm:consistency}\ref{item:cnst:resamp}. The remaining results of Theorem \ref{thm:consistency} and Proposition \ref{prop:resamp_to_non} are proved in Section \ref{sec:const:other}.

\section{Proofs of results for fixed $k$}\label{sec:k_fixed}

\subsection{Markov property and total mass}\label{sec:Markov_mass}

\begin{proof}[Proof of Theorem \ref{thm:Markov}]
 Killed $A$-tree evolutions are Borel right Markov processes as they are effectively tuples of independent type-0/1/2 evolutions, which are themselves Borel right Markov processes as noted in Proposition \ref{prop:012:pred}\ref{item:pred:Markov}, killed at a stopping time. Note that because these evolutions in the various type-0/1/2 compounds in the tree are independent and their degeneration times are continuous random variables, almost surely one of them degenerates before all of the others.  Since each type-$i$ compound has $i$ positive top masses and positive interval partition mass at almost all times before its degeneration time, $\cT^{D_1-}\in\tdTInt_{k}$ a.s.. Therefore, the non-resampling and resampling $k$-tree evolutions are well-defined. Moreover, the type of construction undertaken in Definitions \ref{def:nonresamp_1} and \ref{def:resamp_1} of non-resampling and resampling $k$-tree evolutions is well-studied; it yields a Borel right Markov process by Th\'eor\`eme 1 and the Remarque on p.\ 474 of Meyer \cite{Mey75}.
 
 Finally, we use stopping times $S_n := \inf\big\{y\ge0\colon \min_{i\in A}(m_i^y+\|\beta_{\parent{\{i\}}}^y\|) < 1/n\big\}$, $n\ge1$, to show that these processes fail to be quasi-left-continuous.  
 These times are eventually strictly increasing (as soon as they exceed 0) and they converge to the (first) degeneration time. Thus, the first degeneration time is an increasing limit of stopping times, so it is visible in the left-continuous filtration and is a time at which the three processes are discontinuous.
\end{proof}

\begin{proposition}\label{prop:total_mass_0}
 The total mass process of a non-resampling $A$-tree evolution is a \BESQ[-1]. The total mass process of a resampling $k$-tree evolution is also a \BESQ[-1], killed at the random time $D_\infty := \sup_nD_n$.
\end{proposition}

This is a partial form of Theorem \ref{thm:total_mass}. In Section \ref{sec:non_acc}, we prove Proposition \ref{prop:non_accumulation}, that $D_\infty = \inf\{y\ge0\colon \|\cT^{y-}\|=0\}$, completing the proof of Theorem \ref{thm:total_mass}.

\begin{proof}
 Let $(\cT^y,y\ge0)$ denote a non-resampling $k$-tree evolution. Up until its first degeneration, its total mass $\|\cT^y\|$ is the sum of the total masses of $k-1$ type-0/1/2 evolutions -- one compound for each internal edge $E\in\ft$ of the tree shape. In particular, the sum of the ``type numbers'' of these compounds is $k$: if we let $n_i$ denote the number of type-$i$ compounds, $i\in\{0,1,2\}$, then
 $$k-1 = n_0+n_1+n_2 \qquad\text{and} \qquad k = 0\times n_0 + 1\times n_1 + 2\times n_2 = n_1 + 2n_2.$$ 
 This gives $n_2 = n_0+1$. By Proposition \ref{prop:012:mass}, %From \cite[Theorem 1.5]{Paper1} and \cite[Theorem 3]{Paper3}, 
 the total mass process of a type-$i$ evolution is a \BESQ[1-i]. Then, by the additivity of squared Bessel processes \cite[Proposition 1]{PitmWink18}, the sum of these total masses, $(\|\cT^y\|,y\in [0,D_1])$, evolves as a squared Bessel process with parameter $n_0\times 1 + 0\times n_1 - 1\times n_2 = -1$, stopped at a stopping time in a filtration to which the squared Bessel process is adapted. Moreover, the same argument and the strong Markov property show that the total mass continues to evolve as a \BESQ[-1] between the first and second degeneration times, and so on. Thus, the process evolves as a \BESQ[-1] until its absorption at $0$. The same argument proves the above assertion for the resampling $k$-tree evolution.
\end{proof}

\subsection{The Brownian reduced $k$-tree}\label{sec:B_ktree}

For a tree shape $\ft\in \TShape_{[k]}$, let $\phi_{\ft}\colon \ft\to\{k+1,\ldots,2k-1\}$ be a bijection (e.g.\ one that preserves lexigraphic order).

\begin{proposition}[Section 3.3 of \cite{PW13}]\label{prop:B_ktree}
 Fix $k\ge1$.
 \begin{itemize}[topsep=3pt]
  \item Let $\tau$ be a uniform random element of $\TShape_{[k]}$.
  \item Independently, let $(M_i,1\le i\le 2k-1)\sim \distribfont{Dirichlet}\big(\frac12,\ldots,\frac12\big)$.
  \item Independently, let $(\beta_i,k+1\le i\le 2k-1)$ be a sequence of i.i.d.\ \PDIP[\frac12,\frac12].
 \end{itemize}
 Then $(\tau,(M_i,i\in [k]),(M_{\phi_{\tau}(E)}\beta_{\phi_{\tau}(E)},E\in\ft))$ is a Brownian reduced $k$-tree, as described in the introduction. In particular, its distribution is invariant under the permutation of labels.
\end{proposition}

Denote by $Q_{z,A}(dU)$ the distribution on $\TInt_A$ of the $A$-tree obtained from the distribution $Q_{z,[k]}(dU)$ of a Brownian reduced $k$-tree scaled to have total mass $z$, with leaves then relabeled by the increasing bijection $[k]\rightarrow A$, for $k=\#A$. The resampling kernel $\Lambda_{j,A}$ of Definition \ref{def:resamp_1} satisfies
\begin{equation}\label{eq:B_ktree_resamp}
 \int_{T\in\TInt_{k}}Q_{z,[k]}(dT)f(T) = \int_{(T_i,i\in [2,k])}\Lambda_{2,[1]}(z,dT_2)\cdots\Lambda_{k,[k-1]}(T_{k-1},dT_k)f(T_k),
\end{equation}
where $z\in\TInt_{1}$ denotes the $1$-tree with weight $z$ on its sole component, leaf 1. This formula indicates that the Markov chain that begins with $z$ and at each step, adds a successive label via the resampling kernel, has as its path a consistent system of Brownian reduced $k$-trees, $k\ge1$, each scaled to have total mass $z$. Like the above proposition, this formula follows from the development in \cite[Section 3.3]{PW13}.
%is related to these laws.

\subsection{Pseudo-stationarity}\label{sec:pseudo}

In this section, we will prove Proposition \ref{prop:pseudo:resamp}. Recall that type-0 evolutions do not degenerate (and are reflected when reaching zero total mass), while we say that type-1 evolutions degenerate when they reach the (absorbing) state of zero total mass and type-2 evolutions degenerate when they reach a single-top-mass state on an empty interval partition. In particular, total mass evolutions conditioned on no degeneration up to time $y$ are unaffected by the conditioning for type-0 evolutions as we are conditioning on an event of probability 1, while they are conditioned to be positive for type-1 evolutions and conditioned on an event that depends on the underlying type-2 evolution for type 2.

\begin{proposition}\label{prop:pseudo:pre_D}
 Let $(\cT^y,y\ge0)$ be a killed/non-resampling/resampling $k$-tree evolution starting from a Brownian reduced $k$-tree scaled by an independent random initial mass. Then for $y\ge0$, given $\{D_1>y\}$, the tree $\cT^y$ is again conditionally a Brownian reduced $k$-tree scaled by an independent random mass. In the special case that $\|\cT^0\|\sim\GammaDist[k-\frac12,\gamma]$, given $\{D_1>y\}$, $\|\cT^y\|$ has conditional law $\GammaDist[k-\frac12,\gamma/(2\gamma y+1)]$. 
\end{proposition}
\begin{proof}
 First, suppose $M\sim \GammaDist[k-\frac12,\gamma]$. Note that
 \begin{equation}
  \bP(D_1>y) = (2y\gamma+1)^{-k},
 \end{equation}
 since each type-1 compound contributes $(2y\gamma+1)^{-1}$ by \cite[Equation (6.3)]{Paper1} and each type-2 compound contributes $(2y\gamma+1)^{-2}$, by Proposition \ref{prop:012:pseudo}, all independently, with $k$ top masses altogether. Conditioning on non-degeneration means conditioning each independent type-$i$ evolution, $i=1,2$, not to degenerate; thus, this conditioning does not break the independence of these evolutions. By Proposition \ref{prop:012:pseudo}, the conditional distribution of each edge partition and top mass at time $y$ is the same as the initial distribution, but with each mass and partition scaled up by a factor of $2\gamma y+1$, as claimed. 
 
 The result for deterministic initial total mass follows by Laplace inversion, and for general random mass by integration. We leave the details to the reader and refer to \cite[Proposition 6.4 and the proof of Theorem 6.1]{Paper1} or \cite[Propositions 32, 34]{Paper3} for similar arguments.
\end{proof}

\begin{lemma}\label{lem:scaling}
 If $(\cT^y,y\ge0)$ is a killed/non-resampling/resampling $k$-tree evolution, then for any $c>0$, so is $(c\cT^{y/c},y\ge0)$. 
\end{lemma}

\begin{proof}
 This follows from the corresponding scaling properties for type-0/1/2 evolutions, noted in \cite[Lemma 6.3]{Paper1} and \cite[Lemma 35]{Paper3}.
\end{proof}

%{\tt \cite[Theorem 1]{Problem16again} extends this to independent random times.}

\begin{proposition}\label{prop:pseudo:degen}
 Let $(\cT^y,y\ge 0)$ be a killed/non-resampling/resampling $k$-tree evolution starting from a Brownian reduced $k$-tree, scaled by any independent initial mass $M$, and let $y\ge 0$. Then the following hold.
 \begin{enumerate}[label=(\roman*), ref=(\roman*)]
  \item \label{item:pseudoD:J}
  	The label $J = J(\cT^{D-})$ dropped at the first degeneration time $D = D_1$ has law $\bP(J=2)=2/k(2k-3)$ and $\bP(J=j)=(4j-5)/k(2k-3)$, $j\in\{3,\ldots,k\}$.
  \item \label{item:pseudoD:swapred}
  	Conditionally given $J\!=\!j$, the normalized tree $\varrho\big(\cT^{D-}\big) \big/ \big\|\cT^{D-}\big\|$, which is simply $\cT^D / \big\|\cT^D\big\|$ in the non-resampling case, is a Brownian reduced $([k]\!\setminus\!\{j\})$-tree.
  \item \label{item:pseudoD:indep}
  	The pair $\big(J\big(\cT^{D-}\big),\varrho\big(\cT^{D-}\big) \big/ \big\|\cT^{D-}\big\|\big)$ is independent of $\big(M,D,\big\|\cT^{D-}\big\|\big)$.
  \item \label{item:pseudoD:gamma_mass}
  	In the special case that $M\sim\GammaDist[k-\frac12,\gamma]$, conditionally given $D=y$,
  	$$\|\cT^y\|\sim\GammaDist[k-\textstyle\frac32,\gamma/(2y\gamma+1)].$$
  \item \label{item:pseudoD:resamp}
  	In the resampling case, properties \ref{item:pseudoD:J}, \ref{item:pseudoD:swapred}, and \ref{item:pseudoD:indep} also hold at each subsequent degeneration time $D=D_n$, $n\ge1$. Moreover, $\cT^{D_n}/\big\|\cT^{D_n}\big\|$ is a Brownian reduced $k$-tree.
  	%with $\widetilde{X}_y$ distributed like the sum of total masses of the type-0/1/2 edge evolutions, each conditioned on non-degeneration, with the exception of the degenerate edge, which contributes nothing if type 1 and the mass conditioned on degeneration at $y$ if type 2.   
 \end{enumerate}
\end{proposition}
\begin{proof}
 First, we derive \ref{item:pseudoD:resamp} as a consequence of the other assertions. Equation \eqref{eq:B_ktree_resamp}, along with exchangeability of labels in the Brownian reduced $k$-tree, implies that taking a Brownian reduced $([k]\!\setminus\!\{j\})$-tree and inserting label $j$ via the resampling kernel results in a Brownian reduced $k$-tree. Thus, \ref{item:pseudoD:swapred} gives us $\cT^{D_1}/\big\|\cT^{D_1}\big\| \stackrel{d}{=} \cT^0/\big\|\cT^0\big\|$. Assertion \ref{item:pseudoD:resamp} for subsequent degeneration times then follows by induction and the strong Markov property of resampling $k$-tree evolutions at degeneration times.
 
 It remains to prove \ref{item:pseudoD:J}, \ref{item:pseudoD:swapred}, \ref{item:pseudoD:indep}, and \ref{item:pseudoD:gamma_mass}. We begin with the special case $M\!\sim\!\GammaDist[k\!-\!\frac12,\gamma]$. In this case, by Proposition \ref{prop:B_ktree}, each type-$i$ compound has initial mass $\GammaDist[(i+1)/2,\gamma]$, with all initial masses being independent. For $y>0$, each type-1 compound avoids degeneration prior to time $y$ with probability $(2y\gamma+1)^{-1}$ by \cite[equation (6.3)]{Paper1}. For type-2 the corresponding probability is $(2y\gamma+1)^{-2}$ by \cite[Proposition 38]{Paper3}. Moreover, when a type-2 compound degenerates, each of the two labels is equally likely to be the one to cause degeneration \cite[Proposition 39]{Paper3}. Thus, the first label $I$ to cause degeneration is uniformly random in $[k]$ and is jointly independent with the initial tree shape $\tau_k$ and the time of degeneration $D$. But recall that this does not necessarily mean that label $I$ is dropped at degeneration; we must account for the swapping part of the swap--and-reduce map $\varrho$. 
 
 % our study of projections of the discrete Markov chain
 This places us in the setting of our study of a modified Aldous chain in \cite{Paper2}, where we begin with a uniform random tree shape and select a uniform random leaf for removal, with the same label-swapping dynamics as in the definition of $\varrho$ in Section \ref{sec:non_resamp_def}. In particular, \cite[Corollary 5]{Paper2} gives $p_1:=\bP(J=1)=0$, $p_2:=\bP(J=2)=2/k(2k-3)$ and $p_j:=\bP(J=j)=(4j-5)/k(2k-3)$ for $j\in\{3,\ldots,k\}$; and \cite[Lemma 4]{Paper2} says that given $\{J=j\}$, the tree shape $\tau_{k-1}$ after swapping and reduction is conditionally uniformly distributed on $\TShape_{[k]\setminus\{j\}}$. 
 
 Since $D$ is independent of $(\tau_k,I)$ and since $\tau_{k-1} = \varrho(\tau_k,J(\tau_k,I))$, if we additionally condition on $D$ then the above conclusion still holds: given $\{D=y,J=j\}$, the resulting tree shape $\tau_{k-1}$ is still conditionally uniform on $\TShape_{[k]\setminus\{j\}}$. Moreover, by the independence of the evolutions of the type-0/1/2 compounds in the tree (prior to conditioning), each type-1 compound that does not degenerate is conditionally distributed as a type-1 evolution in pseudo-stationarity, conditioned not to die up to time $y$, and correspondingly for type-2 and type-0 compounds. As noted in Proposition \ref{prop:012:pseudo}, the law at time $y$ is the same as the initial distribution, but with total mass scaled up by a factor of $2y\gamma+1$, meaning that each top mass $m_a^y$ in these compounds is conditionally independent with law $\GammaDist[\frac12,\gamma/(2y\gamma+1)]$, and each internal edge partition $\alpha_E^y$ is a conditionally independent \PDIP[\frac12,\frac12] scaled by a $\GammaDist[\frac12,\gamma/(2y\gamma+1)]$ mass. Similarly, if the degeneration occurs in a type-2 compound in $\tau_k$, then the remaining top mass in that compound also has conditional law $\GammaDist[\frac12,\gamma/(2y\gamma+1)]$ \cite[Proposition 39]{Paper3}. Thus, $\varrho\big(\cT^{D-}\big)/\big\|\cT^{D-}\big\|$ is conditionally a Brownian reduced $([k]\!\setminus\!\{j\})$-tree, as claimed, and is conditionally independent of $\|\cT^D\| \sim \GammaDist[k-\frac32,\gamma/(2y\gamma+1)]$.
 
  This completes the proof of \ref{item:pseudoD:gamma_mass} as well as of \ref{item:pseudoD:J} and \ref{item:pseudoD:swapred} in the special case when the initial total mass is $M\sim\GammaDist[k-\frac12,\gamma]$. Moreover, since the above conditional law of the normalized tree does not depend on the particular value $D$, we find that the pair in \ref{item:pseudoD:indep} is independent of $\big(D,\big\|\cT^D\big\|\big)$ in this case; it remains to show independence from $M$.
 
 Now consider a $k$-tree evolution $(\wbT^{y},y\ge 0)$ starting from a unit-mass Brownian reduced $k$-tree, with degeneration time $\wbD$. Let $M$ be independent of this evolution with law $\GammaDist[k-\frac12,\gamma]$. By the scaling property of Lemma \ref{lem:scaling}, $\cT^y = M\wbT^{y/M}$, $y\ge0$, is a $k$-tree evolution with initial mass $M$, as studied above. In particular,
 \begin{equation*}
  \left(J\big(\cT^{D-}\big),\frac{\varrho\big(\cT^{D-}\big)}{\big\|\cT^{D-}\big\|}\right) = \left(J\big(\wbT^{\wbD-}\big),\frac{\varrho\big(\wbT^{\wbD-}\big)}{\big\|\wbT^{\wbD-}\big\|}\right) \quad\text{and}\quad  \big(D,\big\|\cT^{D-}\big\|\big) = \big(M\wbD,M\big\|\wbT^{\wbD-}\big\|\big).
 \end{equation*}
 %This immediately shows that if instead of conditioning on $\wbD = y$ we condition on $D = M_k\wbD = y$, then the joint conditional distribution of
 %\begin{equation}\label{eq:pseudo:degen_pair}
 % \left( J\big(\wbT^{\wbD-}\big) , \frac{\varrho\big(\wbT^{\wbD-}\big)}{\big\|\varrho\big(\wbT^{\wbD-}\big)\big\|} \right)
 % = \left( J\big(\cT^{D-}\big) , \frac{\varrho\big(\cT^{D-}\big)}{\big\|\varrho\big(\cT^{D-}\big)\big\|} \right)
 %\end{equation}
 %is as claimed in \ref{item:pseudoD:J} and \ref{item:pseudoD:swapred}, with both objects being conditionally independent of $\big\|\varrho\big(\cT^{D-}\big)\big\| = M_k\big\|\varrho\big(\wbT^{\wbD-}\big)\big\|$. But the pair in \eqref{eq:pseudo:degen_pair} is a function of $(\widebar\cT^y,y\ge0)$ and is thus independent of $M_k$, so its distribution given $\wbD = y$ is also as claimed, and is conditionally independent of $\|\varrho\big(\cT^{D-}\big)\|$.
 %
 %By integration, \ref{item:pseudoD:J}, \ref{item:pseudoD:swapred}, and \ref{item:pseudoD:resamp} extend to any independent random initial mass.
 %
 We showed that 
 \begin{equation*}
 \begin{split}
  &\int 
  \bE \left[ f\left( J\big(\wbT^{\wbD-}\big) , \frac{\varrho\big(\wbT^{\wbD-}\big)}{\big\|\wbT^{\wbD-}\big\|}\right)
  		g\left(x\left\|\wbT^{\wbD-}\right\|,x\wbD\right)\right] \frac{\gamma^{k-\frac12}}{\Gamma(k-\frac12)}x^{k-\frac32}e^{-\gamma x} dx\\
  	&\ \ = \bE [ f(J^*,\wbT^*) ] \int \bE\left[g\left(x\left\|\wbT^{\wbD-}\right\|,x\wbD\right)\right] \frac{\gamma^{k-\frac12}}{\Gamma(k-\frac12)}x^{k-\frac32}e^{-\gamma x} dx,
 \end{split}
 \end{equation*}
 where $J^*$ and $\wbT^*$ have the laws described in \ref{item:pseudoD:J} and \ref{item:pseudoD:swapred} for the dropped label and the normalized tree. If we cancel out the constant factors of $\gamma^{k-\frac12}/\Gamma(k-\frac12)$ on each side, appeal to Laplace inversion, and then cancel out factors of $x^{k-\frac32}$, then we find
 \begin{equation*}
  \bE \left[ f\left( J\big(\wbT^{\wbD-}\big) , \frac{\varrho\big(\wbT^{\wbD-}\big)}{\big\|\wbT^{\wbD-}\big\|}\right)
  		g\left(x\left\|\wbT^{\wbD-}\right\|,x\wbD\right)\right] = \bE [ f(J^*,\wbT^*) ] \bE\left[g\left(x\left\|\wbT^{\wbD-}\right\|,x\wbD\right)\right]
 \end{equation*}
 for Lebesgue-a.e.\ $x>0$. By Lemma \ref{lem:scaling}, it follows that this holds for every $x>0$, thus proving \ref{item:pseudoD:J} and \ref{item:pseudoD:swapred} for fixed initial mass, or for any independent random initial mass by integration. This formula also demonstrates the independence of the dropped label and normalized tree from the degeneration time and mass at degeneration. Since the laws that we find for the dropped index and the normalized tree do not depend on the initial mass $x$, this also proves \ref{item:pseudoD:indep}.
\end{proof}

This means that for any scaled Brownian reduced $k$-tree, the $k$-tree evolution without resampling runs through independent multiples of trees $\widebar{\cT}_k^{*},\widebar{\cT}_{k-1}^{*},\ldots,\widebar{\cT}_2^{*},\widebar{\cT}_1^{*},0$, where each of the $m$-trees for $m<k$ has as its distribution the appropriate mixture of Brownian reduced trees with label sets of size $m$. We now combine the previous results to establish the laws of independently scaled Brownian reduced $k$-trees as pseudo-stationary laws for resampling $k$-tree evolutions.

\begin{proof}[Proof of Proposition \ref{prop:pseudo:resamp}]
 By Proposition \ref{prop:pseudo:degen} and \eqref{eq:B_ktree_resamp}, conditional on $D_1 = z>0$, the tree $\cT^{z}$ is distributed as a Brownian reduced $k$-tree scaled by an independent random mass. By the strong Markov property at degeneration times and induction, the same holds conditional on $D_n=z>0$, for any $n\ge1$.
 
 While we have not yet proved Proposition \ref{prop:non_accumulation}, that $D_\infty$ is the hitting time at zero for the total mass process, it is easier to prove it in this special setting. Indeed, the rescaled inter-degeneration times $(D_{n+1}-D_n)/\|\cT^{D_n}\|$ are independent and identically distributed. As noted at the end of Section \ref{sec:Markov_mass}, the total mass $\|\cT^y\|$ evolves as a \BESQ[-1] up until time $D_\infty$, so this time must be a.s.\ finite so as to not exceed the time of absorption for the \BESQ. From this we conclude that the masses $\|\cT^{D_n}\|$ must tend to zero.
 
 Let $(\cT^x_n,x\ge0)$ denote a resampling $k$-tree evolution with $\cT^0_n=\cT^{D_n}$, conditionally independent of $(\cT^y,y\ge0)$ given $\cT^{D_n}$. Now, we condition on $D_n = z\le y < D_{n+1}$. Then by the strong Markov property at time $D_n$, the tree $\cT^y$ is conditionally distributed according to the conditional law of $\cT^{y-z}_n$, given that $(\cT^x_n)$ does not degenerate prior to this time. By Proposition \ref{prop:pseudo:pre_D}, this too is a Brownian reduced $k$-tree scaled by an independent random mass. Integrating out this conditioning preserves the property of $\cT^y$ being a Brownian reduced $k$-tree scaled by an independent mass.
\end{proof}

%\section{Consistency proofs}\label{sec:consistency}
\section{Proof of consistency results for resampling $k$-tree evolutions}\label{sec:consistency}

%This section is dedicated to proving projective consistency results: Theorem \ref{thm:consistency} and Proposition \ref{prop:resamp_to_non}. First, we will prove the following partial result towards Theorem \ref{thm:consistency}\ref{item:cnst:resamp}.

In this section, we prove the following partial result towards Theorem \ref{thm:consistency}\ref{item:cnst:resamp}.

\begin{proposition}\label{prop:consistency_0}
 Fix $T\in\TInt_{k}$. If $(\cT^y_{k+1},y\ge0)$ is a resampling $(k\!+\!1)$-tree evolution with $\cT^0_{k+1}\sim \Lambda_{k+1,[k]}(T,\cdot\,)$ then $\big(\pi_k\big(\cT^y_{k+1}\big),y\ge0\big)$ is a resampling $k$-tree evolution killed at the limit $D_\infty$ of degeneration times in $\big(\cT_{k+1}^y\big)$.
\end{proposition}

Once we have proved this proposition, we use it in Section \ref{sec:non_acc} to prove Proposition \ref{prop:non_accumulation}, that $D_\infty$ is a.s.\ the time at which the total mass $\|\cT^y\|$ first approaches zero. This, in turn, completes the proof of Theorem \ref{thm:consistency}\ref{item:cnst:resamp} from the above proposition.

In Section \ref{sec:const:intertwin} we introduce a process that is intermediate between a resampling $(k\!+\!1)$-tree evolution and a resampling $k$-tree evolution. We appeal to a novel lemma presented in Appendix \ref{sec:intertwining_lem} to prove that the intermediate process is intertwined below the $(k+1)$-tree evolution. In Section \ref{sec:const:Dynkin1}, we show that the projection from the intermediate process to a $k$-tree-valued process satisfies Dynkin's criterion, thus completing a two-step proof of Proposition \ref{prop:consistency_0}, similar to the approach to an analogous discrete result in \cite[proof of Theorem 2]{Paper2}. This pair of relations is illustrated in the commutative diagram in Figure \ref{fig:inter_Dynkin}.

%In Section \ref{sec:non_acc}, we apply this result to prove Proposition \ref{prop:non_accumulation}, that the accumulation point $D_\infty$ of the degeneration times is the first time that the total mass of the evolution hits zero, thereby completing our proof of Theorems \ref{thm:total_mass} and \ref{thm:consistency}\ref{item:cnst:resamp}. In Section \ref{sec:dePoi} we recall arguments from \cite{Paper1} to prove Theorem \ref{thm:dePoi}, that the de-Poissonized $k$-tree evolutions are Markovian and the resampling de-Poissonized $k$-trees are stationary starting from the laws of the Brownian reduced $k$-trees. Finally, in Section \ref{sec:const:other} we prove the remaining assertions of Theorem \ref{thm:consistency}, as well as Proposition \ref{prop:resamp_to_non}.

\begin{figure}[t]
 \centerline{
 \xymatrix@=3em{
  \cT^0_{k+1} \ar[rr]^{P^y_{k+1}} && \cT^y_{k+1}  \\
  \ensuremath{\Ast\cT^0_k} \ar[u]^{\ensuremath{\Ast\Lambda_k}}\ar[d]^{\Phi_1} \ar[rr]^{\ensuremath{\Ast P^y_k}} && \ensuremath{\Ast\cT^y_k} \ar[u]^{\ensuremath{\Ast\Lambda_k}}\ar[d]^{\Phi_1}\\
  \cT^0_k  \ar[rr]^{P^y_k} && \cT^y_k
 }
 } 
 \caption{Combining intertwining with Dynkin's criterion to show, via an intermediate process, that the $k$-tree projection of a resampling $(k+1)$-tree evolution is in turn a Markov process. Intertwining and Dynkin's criterion can be thought of as the upper and lower boxes in this diagram commuting, respectively. The kernels denoted above are introduced in Section \ref{sec:const:intertwin}.\label{fig:inter_Dynkin}}
\end{figure}

\subsection{Intermediate process intertwined below a resampling $(k+1)$-tree}\label{sec:const:intertwin}

We define marked $k$-trees as $k$-trees with one block of the tree ``marked.'' In particular, we are interested in projecting from $(k\!+\!1)$-trees and marking the block of the resulting $k$-tree into which label $k\!+\!1$ must be inserted to recover the $(k\!+\!1)$-tree from the $k$-tree. See Figure \ref{fig:mark_k_tree}.

\begin{figure}[t]
 \centering
 \scalebox{.965}{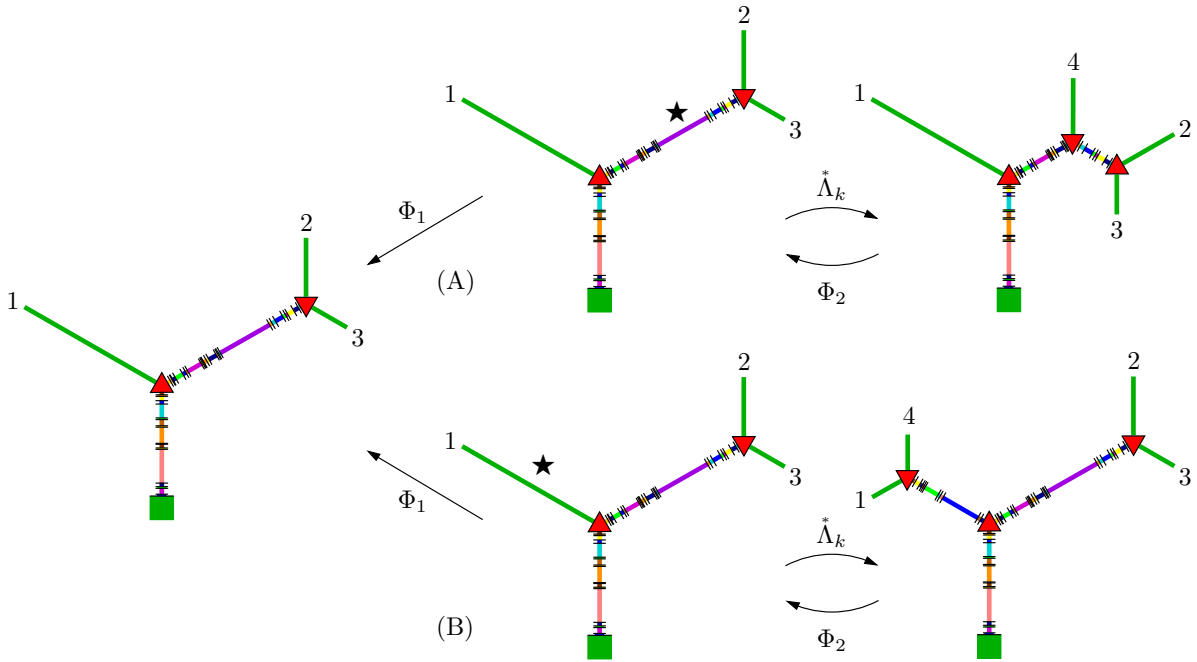}
 \caption{Two marked $k$-trees based on the same $k$-tree, with marked blocks indicated with a ``$\bigstar$.'' In example (A), the marking is on an internal block. Then the kernel $\Ast\Lambda_k$ inserts label $k\!+\!1$ into the block. In (B), the marking is on a leaf block. Then $\Ast\Lambda_k$ splits the marked block into a Brownian reduced 2-tree.\label{fig:mark_k_tree}}
\end{figure}

\begin{definition}\label{def:mark}
 We define the set of \emph{marked $k$-trees}
%  \widebar\TMarkk &:= \left\{(T,\ell)\ \middle|\ T\in \TInt_{k},\ \ell\in\block(T)\cup\left\{(F,a,a)\colon F\in\ft,\,a\in [0,\|\beta_F\|]\setminus\bigcup\nolimits_{U\in\beta_F}U\right\}\right\}\\
%  	&\subset \TInt_{k}\times \big([k]\cup (2^{[k]}\times[0,\infty)^2)\big)
 \begin{equation}\label{eq:marked}
 \begin{split}
  \bTMarkk &:= \left\{\!(T,\ell)\,\middle|\, %\begin{array}{c}
  		T\!\in\! \bTInt_{k}\!\setminus\!\{0\},\ \ell \in \block(T)\cup\left\{(F,a,a)\colon F\!\in\!\ft,\,a\!\in\!\big[0,\|\beta_F\|\big]\setminus\bigcup\nolimits_{U\in\beta_F}\!U\right\}\!%\!\!\\[3pt]
  		%\text{and if }\ell=\big(\parent{\{i\}},0,0\big)\text{ or }\ell = i\text{ for some }i\in [k]\text{ then }x_i>0,
  	%\end{array}
  \right\}\cup\{0\}
 \end{split}
 \end{equation}
 where we take $\ft$, $\beta_F$, and $x_i$ to denote the tree shape and an interval partition and leaf component, respectively, of $T$. 
 We view marked $k$-trees as intermediate objects between $(k\!+\!1)$-trees and $k$-trees, via a pair of projection maps. First, $\phi_1\colon \bTMarkk\to \bTInt_{k}$ is the projection $\phi_1(T,\ell) = T$. We define $\phi_2\colon \bTInt_{k+1}\to\bTMarkk$ as follows.
 \begin{enumerate}[label=(\roman*),ref=(\roman*)]
  \item If in $T\in\bTInt_{k+1}$ we have $\longparent{\{k\!+\!1\}} = \{j,k+1\}$ for some $j\in [k]$, then $\phi_2(T) = (\pi_{k}(T),j)$.
  \item Otherwise, if $E = \longparent{\{k\!+\!1\}}$ is \emph{not} a type-2 edge, then recall part \ref{item:proj:merge} of Definition \ref{def:proj}, in which the interval partitions of $E$ and $F := E\setminus\{k\!+\!1\}$ are combined with the leaf mass $x_{k+1}$ to form the partition $\beta'_F = \beta_F\concat (0,x_{k+1})\concat\beta_{E}$ of $F$ in the projected tree. In this case we define $\phi_2(T) = \big(\pi_{k}(T),(F,\|\beta_F\|,\|\beta_F\|+x_{k+1})\big)$, where the marked block is the block in $\pi_{k}(T)$ corresponding to the leaf mass $x_{k+1}$ in $T$.
 \end{enumerate} 
 We also define a stochastic kernel from $\bTMarkk$ to $\bTInt_{k+1}$. Recall the label insertion operator, $\oplus$, of Section \ref{sec:resamp_def}. Let $\Ast\Lambda_k$ denote the kernel from $\bTMarkk$ to $\bTInt_{k+1}$ that associates with each $(T,\ell)\in \bTMarkk$ the law of $T\oplus (\ell,k\!+\!1,U)$, where $U\sim Q$ is a Brownian reduced $2$-tree of unit mass.
 
 We adopt the convention that $\phi_2(0) = \phi_1(0) = 0$ and $\Ast\Lambda_k(0,\cdot\,) = \delta_0(\,\cdot\,)$.
\end{definition}

The term in \eqref{eq:marked} in which we allow a marking $\ell$ of the form $(F,a,a)$ allows the description of a $(k\!+\!1)$-tree in which the leaf mass $x_{k+1}$ equals $0$ and sits in a type-1 compound. The case $\ell = (F,\|\beta_F\|,\|\beta_F\|)$ corresponds to a degenerate zero leaf mass above a zero interval partition; correspondingly, if $\ell = \big(\parent{\{i\}},0,0\big)$ and $x_i=0$, this corresponds to label $i$ being degenerate in the $(k\!+\!1)$-tree. We write $\TMarkk:=\phi_2(\TInt_{k+1})$ and $\tdTMarkk:=\phi_2(\tdTInt_{k+1})$ for the spaces of marked trees with respectively no degenerate labels or at most one degenerate label, which could be label $k+1$, as discussed above.

We may think of the resampling kernel $\Lambda_{k+1,[k]}$ of Section \ref{sec:resamp_def} as representing a two-step transition, in which a block is first selected at random and then, if a leaf block was chosen, it is split into a scaled Brownian reduced 2-tree. Then $\Ast\Lambda_k$ represents the second of these steps:
\begin{equation}\label{eq:resamp_vs_mark}
 \Ast\Lambda_k\left(\left(T,\ell\right),\cdot\,\right) = \Lambda_{k+1,[k]}\left(T,\cdot\ \middle|\ k\!+\!1\text{ is inserted into }\ell\right)\quad \text{for }(T,\ell)\in\TMarkk\text{ with }\|\ell\|>0.
\end{equation}

We metrize $\bTMarkk$ by
\begin{equation}\label{eq:markk:dist_def}
 d_{\Ast\bT}\left(\Ast T_{k,1},\Ast T_{k,2}\right) := \inf\left\{ d_{\bT}(T_{k+1,1},T_{k+1,2})\colon \phi_2(T_{k+1,1}) = \Ast T_{k,1},\,\phi_2(T_{k+1,2}) = \Ast T_{k,2}\right\}.
\end{equation}

We note that, if $T_{k,1}$ and $T_{k,2}$ have the same tree shape as each other and both are marked in corresponding leaf blocks $i\in [k]$, then
\begin{equation}\label{eq:markk:dist_1}
 d_{\Ast\bT}\left((T_{k,1},i),(T_{k,2},i)\right) = d_{\bT}(T_{k,1},T_{k,2}).
\end{equation}
Indeed, $d_{\bT}(T_{k,1},T_{k,2})$ is a general lower bound for distances between marked $k$-trees. In this special case, the bound can be seen to be sharp by splitting block $i$ in each of the marked $k$-trees into a very small block $k\!+\!1$, a small edge partition with little diversity, and a massive block $i$, in order to form $(k\!+\!1)$-trees that project down as desired. In the limit as block $k\!+\!1$ and the edge partition on its parent approach mass and diversity zero, the $d_{\bT}$-distance between the resulting $(k\!+\!1)$-trees converges to $d_{\bT}(T_{k,1},T_{k,2})$.

On the other hand, if two marked $k$-trees have equal tree shape but the marked blocks lie in different leaf components or internal edge partitions, then
\begin{equation}\label{eq:markk:dist_2}
 d_{\Ast\bT}\left((T_{k,1},\ell_1),(T_{k,2},\ell_2)\right) = d_{\bT}(T_{k,1},0) + d(0,T_{k,2}).
\end{equation}
If $T_{k,1}$ and $T_{k,2}$ have different tree shapes, then both \eqref{eq:markk:dist_1} and \eqref{eq:markk:dist_2} hold, as the right hand sides are then equal, by \eqref{eq:ktree:metric_2}.

This leaves only the case where the two marked $k$-trees have the same shape and the marked blocks each lie in corresponding internal edge partitions in the two trees. Then each marked $k$-tree is as in example (A) in Figure \ref{fig:mark_k_tree}: for $i=1,2$, there is a unique $(k\!+\!1)$-tree $T_{k+1,i}$ for which $\phi_2(T_{k+1,i}) = (T_{k,i},\ell_i)$. Then
%Say, for example, $\ell_1 = (E,a_1,b_1)$ and $\ell_2 = (E,a_2,b_2)$. Let's say also that block $\ell_1$ splits the partition $\beta_{1,E}$ marking edge $E$ in $T_{k,1}$ into $\beta_{1,E} = \alpha_1 \concat (a_1,b_1)\concat \gamma_1$, while $\ell_2$ splits the partition $\beta_{2,E}$ marking edge $E$ in $T_{k,2}$ into $\beta_{2,E} = \alpha_2 \concat (a_2,b_2)\concat \gamma_2$. Then, in this case,
\begin{equation}\label{eq:markk:dist_3}
%\begin{split}
 d_{\Ast\bT}\big((T_{k,1},(E,a_1,b_1)),\, (T_{k,2},(E,a_2,b_2))\big) = d_{\bT}(T_{k+1,1},T_{k+1,2}).%\\
 	%&\quad = d_{\bT}(T_1,T_2) - \dI (\beta_{1,E},\beta_{2,E}) + \dI(\alpha_1,\alpha_2) + |(b_1-a_1) - (b_2-a_2)| + \dI(\gamma_1,\gamma_2).
%\end{split}
\end{equation}

\begin{lemma}\label{lem:Lstar_cont}
 The kernel $\Ast\Lambda_k$ is weakly continuous in its first coordinate.
\end{lemma}

\begin{proof}
 We separately check continuity at zero, at $k$-trees with a marked leaf, and at $k$-trees with the mark in a block of an interval partition. In each case, we consider a sequence $((T_n,\ell_n),n\ge1)$ of marked $k$-trees converging to a limit $\Ast T_\infty$ of that type.
 
 Case 1: $\Ast T_\infty = 0$. Then the total mass $\|T_n\|$ and the diversities of all interval partition components of the $T_n$ must go to zero. Let $U = (m_1,m_2,\alpha)\sim Q$ denote a Brownian reduced 2-tree of unit mass. For $n\ge1$, let $\cT_n := T_n\oplus (\ell_n,U)$. Then $\cT_n$ has law $\Ast\Lambda_k((T_n,\ell_n),\cdot\,)$. We recall from \cite[Lemma 2.12]{Paper1} that scaling an interval partition by $c$, causes its diversity to scale by $\sqrt{c}$. Thus,
 $$d_{\bT}(\cT_n,0) \le d_{\bT}(T_n,0) + \sqrt{\|\ell_n\|}\IPLT_{\alpha}(\infty) \le d_{\bT}(T_n,0) + \sqrt{\|T_n\|}\IPLT_{\alpha}(\infty),$$
 which goes to zero as $n$ tends to infinity. We conclude that $\Ast\Lambda_k$ is weakly continuous at 0.
 
 Case 2: $\Ast T_\infty = (T_\infty,i)$ for some $i\in [k]$. Then by \eqref{eq:markk:dist_2}, for all sufficiently large $n$, $\ell_n = i$ and $T_n$ has the same tree shape as $T_\infty$; call this tree shape $\ft$. Let $U$ and $(\cT_n,n\ge1)$ be as in Case 1. Let $x_{n,i}$ denote the mass of block $i$ in $T_n$. By the bounds on $\dI$-distance between rescaled interval partitions in \cite[Lemma 2.12]{Paper1},
 $$d_{\bT}(T_n,T_m) \le d_{\bT}(\cT_n,\cT_m) \le d_{\bT}(T_n,T_m) + \big|\sqrt{x_{n,i}} - \sqrt{x_{m,i}}\big|\IPLT_{\alpha}(\infty).$$
 Since the sequence $(x_{n,i},n\ge1)$ is Cauchy, it follows from the above bounds that $(\cT_n,n\ge1)$ is a.s.\ Cauchy as well. Thus, we conclude that $\Ast\Lambda_k(T_n,\cdot\,)$ converges weakly.
 
 Case 3: $\ell_n = (E,a_n,b_n)$ for some $E\in \ft$, for all sufficiently large $n$. Then for each such large $n$, there is some $(k\!+\!1)$-tree $T_{k+1,n}$ such that $\Ast\Lambda_k((T_n,\ell_n),\cdot\,) = \delta_{T_{k+1,n}}(\,\cdot\,)$. By \eqref{eq:markk:dist_3}, if the marked $k$-trees $(T_n,\ell_n)$ converge then so do the $(k+1)$-trees $T_{k+1,n}$.
 
 This proves that $\Ast\Lambda_k$ is weakly continuous in its first coordinate everywhere on $\bTMarkk$.
\end{proof}

When composing stochastic kernels, we adopt the standard convention that sequential transitions are ordered from left to right:
$$\int PQ(x,dz)f(z) = \int P(x,dy)\int Q(y,dz)f(z).$$

\begin{definition}\label{def:intertwining}
 Consider a Markov process $X$ with transition kernels $(P_t,t\ge0)$ on a state space $(S,\cS)$, a measurable map $\phi\colon S\to T$, and a stochastic kernel $\Lambda\colon T\times\cS\to [0,1]$. Following \cite{RogersPitman}, the image process $Y(t) := \phi(X(t))$, $t\ge0$, is \emph{intertwined below} $X$ (with respect to $\Lambda$) if:
 \begin{enumerate}
  \item \label{item:intertwin:up} $\Lambda\Phi$ is the identity kernel on $(T,\cT)$, where $\Phi$ is the kernel $\Phi(x,\cdot\,) = \delta_{\phi(x)}(\,\cdot\,)$, and
  \item \label{item:intertwining} $\Lambda P_t = Q_t\Lambda$, $t\ge0$, where $Q_t := \Lambda P_t\Phi$.
 \end{enumerate}
\end{definition}

If $Y$ is intertwined below $X$ and $X$ additionally satisfies
\begin{enumerate}[start=3]
 \item \label{item:intertwin:init} $X(0)$ has regular conditional distribution (r.c.d.) $\Lambda(Y(0),\cdot\,)$ given $Y(0)$,
\end{enumerate}
then $(Y(t),t\ge0)$ is a Markov process \cite[Theorem 2]{RogersPitman}. %Plugging the formula for $Q_t$ into the equation in \ref{item:intertwining} leads to this reformulation:
If conditions \ref{item:intertwin:up} and \ref{item:intertwin:init} are satisfied, then \ref{item:intertwining} is equivalent to
\begin{enumerate}[start=2, label=(\roman*'), ref=(\roman*')]
 \item \label{item:intertwining_v2} $X(t)$ is conditionally independent of $Y(0)$ given $Y(t)$, with r.c.d.\ $\Lambda(Y(t),\cdot\,)$.%, i.e.
% \begin{equation*}
% \begin{split}
%  \Pr\{X(t)\in A\ |\ Y(0) = y\} &= \int_{z\in T} \bP\{X(t)\in A\ |\ Y(0) = y,\,Y(t) = z\}\bP\{Y(t) \in dz\ |\ Y(0) = y\}\\
%  	&= \int_{z\in T} \Lambda(z,A)Q_t(y,dz).
% \end{split}
% \end{equation*}
\end{enumerate}

Now, fix $(T,\ell)\in\TMarkk$. Let $(\cT_{k+1}^y,y\ge0)$ denote a resampling $(k\!+\!1)$-tree evolution with initial distribution $\Ast\Lambda_k((T,\ell),\cdot\,)$. Let $\Ast\cT_{k}^y = (\cT_{k}^y,\ell^y) = \phi_2(\cT^y_{k+1})$, $y\ge0$. We denote the components of these evolutions by
\begin{equation}\label{eq:const:component_notation}
\begin{split}
 \cT^y_{k+1} &= \big(\tau^y,(m_i^y,i\in [k+1]), (\alpha_E^y,E\in\tau^y)\big),\\
 \cT^y_k &= \big(\widetilde\tau^y,(\widetilde m_i^y,i\in [k]), (\widetilde\alpha_E^y,E\in\widetilde\tau^y)\big).
\end{split}
\end{equation}
Note that conditions \ref{item:intertwin:up} and \ref{item:intertwin:init} above are satisfied with $\phi = \phi_2$ and $\Lambda = \Ast\Lambda_k$. To conclude that $\Ast\cT_k$ is Markovian, it remains to check condition \ref{item:intertwining} or \ref{item:intertwining_v2}.

Our approach to proving intertwining goes by way of Lemma \ref{lem:intertwin_jump}, which is a novel, general result for proving intertwining between processes that jump away from some set of ``boundary'' states, which arise only as left limits of the processes at a discrete sequence of times. The lemma states that it suffices to prove that the processes killed upon first approaching the ``boundary'' are intertwined, and that the Markov chains of states that the processes jump into at those times are intertwined as well.

Let $D_1,D_2,\ldots$ denote the sequence of degeneration times of $\big(\cT^y_{k+1}\big)$. Recall the definition of the swap-and-reduce map, $\varrho$, in Section \ref{sec:non_resamp_def}. When a label $i$ in a $(k\!+\!1)$-tree degenerates, it swaps places with $\max\{i,a,b\}$, where $a$ is the least label descended from the sibling of $i$ and $b$ is the least label descended from its uncle. Since $k\!+\!1$ is the greatest label in the tree, there are three cases in which it will resample:
%\begin{enumerate}[label=Case \arabic*:, ref= \arabic*, itemindent=1.4cm, leftmargin=0pt,topsep=4pt]
\begin{enumerate}[label=(D\arabic*), ref= (D\arabic*), topsep=4pt, itemsep=3pt]
 \item\label{case:degen:type2} $k\!+\!1$ belongs to a type-2 compound, and either it (in the case $k\!+\!1=i$) or its sibling (in the case $k\!+\!1=a$) causes degeneration;
 \item\label{case:degen:self} $k\!+\!1$ belongs to a type-1 compound and causes degeneration, so $k\!+\!1=i$; or
 \item\label{case:degen:nephew} $k\!+\!1=b$, as $k\!+\!1$ belongs to a type-1 compound and its sibling in the tree shape is an internal edge that belongs to a type-1 or type-2 compound that degenerates.
\end{enumerate}
In case \ref{case:degen:type2}, the resampling kernel $\Lambda_{k+1,[k]}(\varrho(\cT_{k+1}^{D_n-}),\cdot\,)$ may select the block of the former sibling of $k\!+\!1$ as the point of reinsertion. In this case, the marked $k$-tree remains continuous at this degeneration time, $\Ast\cT^{D_n}_k = \Ast\cT^{D_n-}_k$, and the degeneration time is not a stopping time in the filtration generated by the marked $k$-tree process $\big(\Ast\cT_k^y\big)$. Then, we say that $D_n$ is an \emph{invisible} degeneration time. Otherwise, the degeneration time is said to be \emph{visible}. In particular, in cases \ref{case:degen:self} and \ref{case:degen:nephew} the degeneration is always visible.

Consider the subsequence $D'_1,D'_2,\ldots$ of visible degeneration times. We define $D_0 = D'_0 := 0$. Let $\Phi_2$ denote the stochastic kernel from $\TInt_{k+1}$ to $\TMarkk$ given by $\Phi_2(T,\cdot\,) = \delta_{\phi_2(T)}(\,\cdot\,)$.  Let $P^{\circ,k+1}_y$, $y\ge0$, denote the transition kernels for a resampling $(k\!+\!1)$-tree evolution that jumps and is absorbed at the zero tree at time $D'_1$.

\begin{proposition}\label{prop:inter:killed}
 $\Ast\cT^{\circ,y}_{k} := \Ast\cT^y_{k}\cf\{y<D'_1\}$, $y\ge 0$, is a Markov process intertwined below $\cT^{\circ,y}_{k+1} := \cT^y_{k+1}\cf\{y<D'_1\}$, $y \ge0$, with kernel $\Ast\Lambda_k$.
% \begin{equation}\label{eq:inter:killed}
%  \Ast\Lambda_k P^{\circ,k+1}_y = \Ast P^{\circ,k}_y \Ast\Lambda_k, \quad \text{where} \quad \Ast P^{\circ,k}_y := \Ast\Lambda_k P^{\circ,k+1}_y \Phi_2.
% \end{equation}
\end{proposition}

\begin{proof}
 As noted above, conditions \ref{item:intertwin:up} and \ref{item:intertwin:init} in Definition \ref{def:intertwining} are satisfied by our construction. Thus, it suffices to check condition \ref{item:intertwining_v2}, that $\cT^{\circ,y}_{k+1}$ has r.c.d.\ $\Ast\Lambda_k(\Ast\cT^{\circ,y}_k,\cdot\,)$ given $\Ast\cT^{\circ,y}_k$.
 
 By construction, we have the a.s.\ equality of events $\big\{\cT^{\circ,y}_{k+1} = 0\big\} = \big\{\Ast\cT^{\circ,y}_k = 0\big\} = \big\{D'_1 < y\big\}$. Indeed, $\Ast\Lambda_k(0,\cdot\,) = \delta_0$, so this gives the correct conditional distribution on this event. It remains to check the claimed r.c.d.\ over the domain $\Ast\cT^y_k = \Ast T\in \TMarkk\setminus\{0\}$.
 
 Recall that $(T,\ell)\in\TMarkk$ denotes the fixed initial state of $\Ast\cT^y_k = \big(\cT^y_k,\ell^y\big)$, $y\ge0$.
 
 Case 1: $\ell\notin [k]$. Then label $k\!+\!1$ belongs to a type-1 compound in $\cT^0_{k+1}$, so the first degeneration cannot be a case \ref{case:degen:type2} degeneration, and it must be visible, $D_1 = D'_1$. Thus, $\{\Ast\cT^y_k \neq 0\} = \{y<D_1\}$. On the event $\{y<D_1\}$, the tree shape of $\Ast\cT^y_k$ must equal that of $T$, and its marked block must lie in the same internal edge as the initial marked block $\ell$. For any marked $k$-tree $\Ast T$ with the latter of these properties, there is a unique $(k\!+\!1)$-tree $T'$ that satisfies $\Ast T = \phi_2(T')$, as in example (A) in Figure \ref{fig:mark_k_tree}, and $\Ast\Lambda_k\big(\Ast T,\cdot\,\big) = \delta_{T'}$. Thus, if $\ell\notin [k]$ then $\cT^{\circ,y}_{k+1}$ is in fact a deterministic function of $\Ast\cT^{\circ,y}_k$, and $\Ast\Lambda_k(\Ast\cT^{\circ,y}_k,\cdot\,)$ is a r.c.d.\ for $\cT^{\circ,y}_{k+1}$ given $\Ast\cT^{\circ,y}_k$ as claimed, albeit in a trivial sense.
 
 Case 2: $\ell = i\in [k]$. By the discussion above this proposition, for $n\ge1$, on the event $\{D_n < D'_1\}$, the time $D_n$ is an invisible degeneration time, so label $k\!+\!1$ has just been dropped from a type-2 compound $(m_i^{D_n-},m_{k+1}^{D_n-},\alpha_{\{i,k+1\}}^{D_n-})$ and it resamples into block $i$. Formally,
 \begin{equation}\label{eq:degen_filtration}
  \{D_n < D'_1\} = \{D_{n-1}< D'_1\} \cap\{ \ell^{D_n} = i \}.
 \end{equation}
 For $n\ge1$ let $\cF_{n-} := \sigma\big(\cT^{\circ,y}_{k+1},\,y \in [0,D_n)\big)$ and $\cF_n := \sigma\big(\cT^{\circ,y}_{k+1},\,y \in [0,D_n]\big)$. The first event on the right hand side in \eqref{eq:degen_filtration} is measurable in $\cF_{n-}$ while the second is measurable in $\cF_n$. Moreover, by definition of invisible degeneration times, on the event $\{D_n < D'_1\}$ we get $\Ast\cT^{\circ,D_n}_k = \Ast\cT^{\circ,D_n-}_k$, the latter of which is $\cF_{n-}$-measurable. We also get $\varrho\big(\cT^{\circ,D_n-}_{k+1}\big) = \cT^{\circ,D_n-}_k = \cT^{\circ,D_n}_k$. Thus, by \eqref{eq:resamp_vs_mark}, 
 $\Ast\Lambda_k(\Ast\cT^{\circ,D_n}_k,\cdot\,)$ is a r.c.d.\ for $\cT^{\circ,D_n}_{k+1}$ given $\cF_{n-}$.
 
 Let $N := \sup\{n\ge0\colon D_n<y\}$. Note that $\{n = N\} = \{D_n < y\}\cap \{y \le D_{n+1}\}$. 
 Thus, for bounded, measurable $f\colon \TInt_{k+1}\to\BR$ with $f(0)=0$,
 \begin{equation}\label{eq:inter:killed:calc}
 \begin{split}
  	\bE\big[f\big(\cT^{\circ,y}_{k+1}\big)\big] &= \sum_{n\ge 0}\int_{[0,y)\times\TMarkk}\bP\!\left\{D_n\in dz,\,\Ast\cT^{\circ,D_n}_k \in d\Ast T\right\}\int_{\TInt_{k+1}} \Ast\Lambda_k\!\left(\Ast T,dT'\right)\\ 
  		%\bP\big\{N = n\ \big|\ D_n=z,\,\cT^z_{k+1} = T'\big\}\
  		&\qquad\qquad \int_{\TInt_{k+1}}\bP(\cT^{\circ,y}_{k+1}\cf\{y<D_{n+1}\} \in dT''\ |\ \cT^{\circ,z}_{k+1} = T',\,D_n = z)f(T'').
  \end{split}
 \end{equation}
 By the strong Markov property of $\big(\cT^{\circ,y}_{k+1},y\ge0\big)$,
 \begin{equation*}
  \bP\left(\cT^{\circ,y}_{k+1}\cf\{y<D_{n+1}\} \in dT''\ \middle|\ \cT^{\circ,z}_{k+1} = T',\,D_n = z\right) = \bP\!\left(\widehat\cT^{y-z}_{k+1}\cf\{y-z < \widehat D_1\} \in dT''\ \middle|\ \widehat\cT^0_{k+1} = T'\right),
 \end{equation*}
 where $(\widehat\cT^y_{k+1},y\ge0)$ is a $(k\!+\!1)$-tree evolution and $\widehat D_1$ its first degeneration time. Now, suppose we start this evolution with initial distribution $\Ast\Lambda_k\big(\Ast T,\cdot\,\big)$, where the marked block in $\Ast T$ is a leaf block $i\in [k]$. Then in $\widehat\cT^0_{k+1}$, the type-2 compound $\big(\widehat m^y_i,\widehat m_{k+1}^y,\widehat\alpha_{\{i,k+1\}}^y\big)$ begins in a pseudo-stationary initial distribution. Conditioning the $(k\!+\!1)$-tree evolution not to degenerate prior to time $y-z$ means conditioning this type-2 compound not to degenerate, and likewise for the other, independently evolving compounds that comprise this $(k\!+\!1)$-tree evolution. Thus, by Proposition \ref{prop:012:pseudo}, $\big(\widehat m^{y-z}_i,\widehat m_{k+1}^{y-z},\widehat\alpha_{\{i,k+1\}}^{y-z}\big)$ is conditionally a Brownian reduced 2-tree scaled by an independent random mass, and this is independent of all other compounds. Equivalently, $\Ast\Lambda_k\big(\phi_2\big(\widehat\cT^{y-z}_{k+1}\cf\{y-z < \widehat D_1\}\big),\cdot\,\big)$ is a r.c.d.\ for $\widehat\cT^{y-z}_{k+1}\cf\{y-z < \widehat D_1\}$ given $\phi_2\big(\widehat\cT^{y-z}_{k+1}\cf\{y-z < \widehat D_1\}\big)$. Plugging this back into \eqref{eq:inter:killed:calc}, we get
 \begin{align*}
  \bE\big[f\big(\cT^{\circ,y}_{k+1}\big)\big] &= \sum_{n\ge 0}\int_{[0,y)\times\TMarkk}\bP\!\left\{D_n\in dz,\,\Ast\cT^{\circ,D_n}_k \in d\Ast T\right\}\\
  		&\qquad \int_{\TMarkk} \bP\left(\Ast\cT^{\circ,y}_{k}\cf\{y<D_{n+1}\} \in d\Ast T'\ \middle|\ D_n = z,\,\Ast\cT^{\circ,D_n}_k = \Ast T\right)\int_{\TInt_{k+1}}\Ast\Lambda_k\left(\Ast T',dT''\right)f(T'')\\
  		%&= \int_{\TMarkk} \bP\left\{\Ast\cT^{\circ,y}_{k} \in d\Ast T'\right\}\int_{\TInt_{k+1}}\Ast\Lambda_k\left(\Ast T',dT''\right)f(T'')\\
  		&= \bE\left[\int\Ast\Lambda_k\left(\Ast\cT^{\circ,y}_{k},dT''\right)f(T'')\right].
 \end{align*}
 Thus, $\Ast\Lambda\big(\Ast\cT^{\circ,y}_k,\cdot\,\big)$ is a r.c.d.\ for $\cT^{\circ,y}_{k+1}$ given $\Ast\cT^{\circ,y}_k$, as claimed.
\end{proof}

%\begin{corollary}\label{cor:inter:preswap}
% For $n\ge1$, $\Ast\Lambda_k\big(\Ast\cT_k^{D'_n-},\cdot\,\big)$ is a r.c.d.\ for $\cT_{k+1}^{D'_n-}$ given $\big(\Ast\cT_k^y,y\in [0,D'_n))$.
%\end{corollary}

%\begin{proof}
% By the strong Markov property and induction, it suffices to prove this for $n=1$. 
% Testing
%% It is easily seen via a coupling argument that $\Ast\Lambda_k(\,\cdot,\cdot\,)$ is weakly continuous in its first coordinate. Thus, the claim follows from Proposition \ref{prop:inter:killed}.
%\end{proof}
%Let $Q^{k+1}$ denote the transition kernel for the Markov chain $((\cT_{k+1}^{D'_n},D'_n),n\ge0)$, and for the purpose of the following, let $I$ denote the identity kernel on $\BR$.

\begin{proposition}\label{prop:inter:skel}
 $\Big(\!\Big(\Ast\cT^{D'_n}_{k},D'_n\Big),n\!\ge\!0\Big)$ is a Markov chain intertwined below $\Big(\!\Big(\cT^{D'_n}_{k+1},D'_n\Big),n\!\ge\!0\Big)$ with kernel $\Ast\Lambda_k\otimes I$, where $I$ denotes the identity kernel on $\BR$.
\end{proposition}

\begin{proof}
 Again, it suffices to verify condition \ref{item:intertwining_v2}, that for $n\ge1$, $\Ast\Lambda_k\otimes I\big(\big(\Ast\cT^{D'_n}_{k},D'_n\big),\cdot\,\big)$ is a r.c.d.\ for $\big(\cT^{D'_n}_{k+1},D'_n\big)$ given $\big(\Ast\cT^{D'_n}_{k},D'_n\big)$. Equivalently, we must show that $\Ast\Lambda_k\big(\Ast\cT^{D'_n}_{k},\cdot\,\big)$ is a r.c.d.\ for $\cT^{D'_n}_{k+1}$ given $\big(\Ast\cT^{D'_n}_{k},D'_n\big)$. In fact, if we can prove this claim for $n=1$, then by the strong Markov property of the resampling $(k\!+\!1)$-tree evolution at time $D'_1$, we will have that $\big(\big(\Ast\cT_{k}^{D'_{n+1}},D'_{n+1}-D'_1\big),n\ge0 \big)$ is another instance of the same Markov chain, but with a different initial distribution that also satisfies condition \ref{item:intertwin:init} of Definition \ref{def:intertwining}. Then, by induction, the full result will follow, for all $n\ge1$.
 
 To verify the claim, we first note that on the event that label $k\!+\!1$ resamples at time $D'_1$, as at any time label $k\!+\!1$ resamples, we have
$\cT^{D'_1}_k = \varrho\big(\cT^{D'_1-}_{k+1}\big)$. Thus, \eqref{eq:resamp_vs_mark} yields  
 \begin{equation*}
  \bE\left[\cf\left\{J\left(\cT^{D'_1-}_{k+1}\right)=k\!+\!1\right\}f\left(\cT^{D'_1}_{k+1}\right)\right]
  	=\bE\left[\cf\left\{J\left(\cT^{D'_1-}_{k+1}\right)=k\!+\!1\right\}\int_{\bT_{k+1}}\Ast\Lambda_k\left(\Ast\cT^{D'_1}_{k},dT\right)f(T)\right].
 \end{equation*} 
 Also, defining $N$ so that $D_{N+1} = D'_1$ and conditioning on this $N$ as in the proof of Proposition \ref{prop:inter:killed}, we can show that
 \begin{equation*}
  \bE\left[\cf\left\{J\left(\cT^{D'_1-}_{k+1}\right)\in[k]\right\}f\left(\cT^{D'_1-}_{k+1}\right)\right]
  	= \bE\left[\cf\left\{J\left(\cT^{D'_1-}_{k+1}\right)\in[k]\right\}\int_{\bT_{k+1}}\Ast\Lambda_k\left(\Ast\cT^{D'_1-}_{k},dT\right)f(T)\right].
 \end{equation*}
 Specifically, conditioning a resampling $(k\!+\!1)$-tree evolution (the post-$D_n$ evolution) to drop a label in $[k]$ at the first degeneration means conditioning the compound containing label $k\!+\!1$ to be non-degenerate at the independent first degeneration time of the other compounds, so any initial pseudo-stationarity of this compound yields pseudo-stationarity at this independent random time, by Proposition \ref{prop:012:pseudo}.
 
 To complete the proof, we want to deduce that we have
 \begin{equation}\label{eq:rcd_small_labels}
  \bE\left[\cf_{A} f\left(\cT^{D'_1}_{k+1}\right)\right]
  	=\bE\left[\cf_{A}\int_{\bT_{k+1}}\Ast\Lambda_k\left(\Ast\cT^{D'_1}_{k},dT\right)f(T)\right].
 \end{equation}
 for $A = A_j := \big\{J\big(\cT^{D'_1-}_{k+1}\big) = j\big\}$, $j\in[k]$. To this end, we recall that, by the definition of resampling 
$(k\!+\!1)$-tree evolutions, $\Lambda_{j,[k\!+\!1]\setminus\{j\}}\big(\varrho\big(\cT^{D'_1-}_{k+1}\big),\cdot\,\big)$, is a r.c.d. of $\cT^{D'_1}_{k+1}$ given $\cT^{D'_1-}_{k+1}$ with $J\big(\cT^{D'_1-}_{k+1}\big) = j$. This means that, on $A_j$, we may write $\cT_{k+1}^{D'_1}$ in terms of 
$\Ast\cT_k^{D'_1-} = \big(\cT_k^{D'_1-},L\big)$ as
 \begin{equation*}
  \cT_{k+1}^{D'_1}=\left(\varrho\left(\cT_k^{D'_1-}\right)\oplus(L,k\!+\!1,\mathcal{U})\right)\oplus(M,j,\mathcal{V}),
 \end{equation*}
 where $\mathcal{U},\mathcal{V}\sim Q$ are two independent Brownian reduced 2-trees and, given $\Ast\cT_k^{D'_1-}$, $\mathcal{U}$ and $\mathcal{V}$, the block $M$ of $\varrho\big(\cT_k^{D'_1-}\big)\oplus(L,k+1,\mathcal{U})$ is chosen at random according to the masses of blocks. We derive 
\eqref{eq:rcd_small_labels} separately for $A=A_{j,i}$, $1\!\le\! i\!\le\! 5$, where $\{A_{j,i}\colon 1\!\le\! i\!\le\! 5\}$ is a partition of $A_j$, which we define in the following.
 
 On $A_{j,1} := \big\{L\!\in\![k], M\!\in\!\block\big(\varrho\big(\cT_k^{D'_1-}\big)\big), M\!\neq\! L\big\}$, we have $\Ast\cT_k^{D'_1} = \big(\varrho\big(\cT_k^{D'_1-}\big)\!\oplus\!(M,j,\mathcal{V}),L\big)$ and $\cT_{k+1}^{D'_1} = \cT_k^{D'_1}\oplus(L,k+1,\mathcal{U})$, with $\mathcal{U}$ independent of $\Ast\cT_k^{D'_1}$. Hence, \eqref{eq:rcd_small_labels} holds for $A=A_{j,1}$.
 
 On $A_{j,2} := \{M=k+1\}$, we have $\cT^{D'_1}_{k+1} = \cT_k^{D'_1}\oplus(j,k+1,\widebar{\mathcal{V}})$, where $\widebar{\mathcal{V}}$ is formed by swapping leaf labels in $\mathcal{V}$. As the Brownian reduced 2-tree has exchangeable labels, $\widebar{\mathcal{V}}\sim Q$, so \eqref{eq:rcd_small_labels} holds for $A = A_{j,2}$.
 
 On $A_{j,3} := \big\{L\not\in[k],\, M\neq k\!+\!1\}$, the tree $\cT^{D'_1}_{k+1}$ is the unique $(k\!+\!1)$-tree with $\phi_2\big(\cT^{D'_1}_{k+1}\big) = \Ast\cT_k^{D'_1}$, as in example (A) in Figure \ref{fig:mark_k_tree}, and indeed $\Ast\Lambda_k\big(\Ast\cT^{D'_1}_k,\cdot\,\big)$ is a point mass at this $(k\!+\!1)$-tree. Then, as in Case 1 in the proof of Proposition \ref{prop:inter:killed}, we conclude trivially.%: \eqref{eq:rcd_small_labels} holds for $A=A_{j,3}$.
 
 Consider $A_{j,4} := \{L\in[k],\, M\in \{(\{L,k\!+\!1\},a,b)\colon a,b\in\bR\}\}$. We view this as a sub-event of $B := A_{j,4} \cup \{L\in [k],\,  M\in \{L,k\!+\!1\}\}$. Then $B$ is the event that label $k\!+\!1$ is inserted into a leaf block of $\varrho\big(\cT_k^{D'_1-}\big)$ forming a type-2 compound in $\varrho\big(\cT_{k+1}^{D'_1-}\big)$, and then label $j$ is inserted into a block in this type-2 compound. On this event, the aforementioned type-2 compound is a Brownian reduced 2-tree. Thus, by \eqref{eq:B_ktree_resamp}, the subtree of $\cT^{D'_1}_{k+1}$ comprising leaf blocks $L$, $j$, and $k\!+\!1$ and the partitions on their parent edges is a Brownian reduced 3-tree. The tree shape of this 3-tree has two leaves at distance 3 from the root, sitting in a type-2 compound, and one leaf at distance 2 from the root, in a type-1 compound. It follows from Proposition \ref{prop:B_ktree} that the type-2 compound is then a Brownian reduced 2-tree $\mathcal{W}$ scaled by an independent random mass. Conditioning on $A_{j,4}$ is equivalent to conditioning label $j$ to sit in the type-1 compound. By Proposition \ref{prop:B_ktree}, the labels in the Brownian reduced 3-tree are exchangeable, so $\mathcal{W}$ remains a Brownian reduced 2-tree under this conditioning. Hence, \eqref{eq:rcd_small_labels} holds for $A = A_{j,4}$.
 
 The event $A_{j,5}:=\{L\in[k],\, M=L\}$ is another sub-event of $B$ introduced in the preceding paragraph. This is the sub-event in which label $k\!+\!1$ sits in the type-1 compound in the Brownian reduced 3-tree. We argue as for $A_{j,3}$.
\end{proof}

As we have mentioned above, Lemma \ref{lem:intertwin_jump} is a novel, general lemma for proving intertwining. Propositions \ref{prop:inter:killed} and \ref{prop:inter:skel} provide the ingredients needed to apply this lemma, as well as Lemma \ref{lem:intertwin_strong}, to deduce the following conclusion. %180919

%In Appendix \ref{sec:intertwining_lem}, we prove a general lemma showing that in order to prove intertwining, it suffices to prove intertwining for a ``skeletal'' Markov chain embedded in the process along a sequence stopping times and for a version of the process killed at such stopping times. In particular, the two propositions above and Lemma \ref{lem:intertwin_jump} imply the following.

\begin{proposition}\label{prop:intertwined}
 $\big(\Ast\cT^y_k,y\ge 0\big)$ is a strong Markov process intertwined below $\big(\cT^y_{k+1},y\ge0\big)$ with kernel $\Ast\Lambda_k$. Moreover, the intertwining criteria \ref{item:intertwining} and \ref{item:intertwining_v2} hold not just at fixed times $y$, but at all stopping times for $\big(\Ast\cT^y_k,y\ge 0\big)$.
\end{proposition}

We call $\big(\big(\Ast\cT^y_k\big),y\ge 0\big)$ the \emph{(resampling) marked $k$-tree evolution from initial state $(T,\ell)$}.

\subsection{Projection from marked $k$-tree}\label{sec:const:Dynkin1}

We continue with the notation of the preceding subsection, but now we study $\cT^y_k = \phi_1\big(\Ast\cT^y_k\big) = \pi_{k}\big(\cT_{k+1}^y\big)$, $y\ge0$. For the moment, we know that this is a $\TInt_k$-valued stochastic process; in the course of this section, we will show that it is a resampling $k$-tree evolution up to time $D_\infty = \sup_n D_n$. Let $(D''_n,n\ge1)$ denote the subsequence of degeneration times of $(\cT_{k+1}^y,y\ge0)$ at which a label other than $k\!+\!1$ is killed. This is a further subsequence of $(D'_n,n\ge1)$. 

%Let $\widehat Q^k$ denote the Markov kernel on $\TMarkk$ that sends the initial state $(T,\ell)$ to the law of $\Ast\cT_k^{D''_1}$. Let $Q^k$ denote the Markov kernel on $\TInt_{k}$ that sends $T$ to the distribution of a resampling $k$-tree started at $T$, evaluated at its first degeneration time (i.e.\ after reduction and resampling).

\begin{proposition}\label{prop:Dynkin:killed}
 $\big(\cT^{y}_k\cf\{y<D''_1\},y\ge0\big)$ is a killed $k$-tree evolution and is Markovian in the filtration generated by $\big(\cT^y_{k+1},y\ge0\big)$.%180917
\end{proposition}

This next proof is where we finally appreciate the usefulness of the swap-and-reduce map of Section \ref{sec:non_resamp_def} in preserving projective consistency at degeneration times.

\begin{proof}
 This will be done by a coupling argument. In particular, we will define a killed $k$-tree evolution $\big(\widetilde \cT^y_k, y\in [0,\wt D)\big)$ with which we will couple a resampling $(k\!+\!1)$-tree evolution $\big(\wt\cT^y_{k+1}, y\in [0,\wt D\wedge D_{\infty})\big)$ in such a way that $\big(\pi_k\big( \wt\cT^y_{k+1}), y\in [0,\wt D\wedge D_{\infty})\big) = \big(\wt \cT^y_{k}, y\in [0,\wt D\wedge D_{\infty}) \big)$ almost surely. 
 
 Let $(T,\ell)\in\TMarkk$ be as in the previous section. We denote the coordinates of $T$ by $(\ft,(x_i,i\in[k]),(\beta_E,E\in\ft))$. Let $(\Omega_{(0)},\cF_{(0)},\bP_{(0)})$ denote a probability space on which we have defined an independent type-$d$ evolution corresponding to each type-$d$ compound in $T$, with initial state equal to that compound, for $d=0,1,2$. We denote the top mass evolution corresponding to each leaf $i$ by $\big(\wt m_i^y,y\ge0\big)$, and we denote the interval-partition-valued process associated with each internal edge $E$ by $\big(\wt\alpha_E^y,y\ge0\big)$. Let $\widetilde D$ denote the minimum of the degeneration times of these type-$d$ evolutions. As in Definition \ref{def:killed_ktree}, $\widetilde\cT^y_k := \big(\ft,\big(\widetilde m_j^y,j\in[k]\big), \big(\widetilde\alpha_E^y,E\in\ft\big)\big)$, $y\in [0,\widetilde D)$, is a killed $k$-tree evolution.
 
 %We construct the desired $(k\!+\!1)$-tree evolution by recursively extending this probability space. Before explaining this construction, we set out our notation. For $n\ge 1$, we will denote by $A_n$ the event that label $k\!+\!1$ resamples at least $n$ times before a lower label would resample, should we extend $\big(\wt\cT^y_{k+1}\big)$ to an unstopped resampling $(k\!+\!1)$-tree evolution. We will set $\wt D_n$ equal to this $n^\text{th}$ degeneration time on $A_n$, or equal to $\widetilde D$ on $A_n^c$.
 
 %We extend each successive probability space $(\Omega_n,\cF_n,\bP_n)$ to $(\Omega_{n+1},\cF_{n+1},\bP_{n+1})$ to include the following.
% \begin{enumerate}[label = (\arabic*),ref=(\arabic*)]
%  \item A random block $L_{n+1}$. On $A_n^c$ we set $L_{n+1}$ equal to a cemetary state $\partial$. Conditionally given $A_n$ and sub-$\sigma$-algebra of $\cF_{n+1}$ corresponding to $\cF_n$, we have $L$
%  conditionally distributed as a size-biased random block in $\cT^{\wt D_n}$ given $A_n$ and the sub-$\sigma$-algebra of $\cF_{n+1}$ corresponding to $\cF_n$, or $L_{n+1} = \partial$ on $A_n^c$;
%  \item $(U_{(n+1)}^y,y\ge0)$ conditionally distributed as a pseudo-stationary type-2 evolution conditioned to have total mass process $\big\|U_{n+1}^y\big\| = m_{L_{n+1}}^{\wt D_n+y}$, $y\ge0$, given $\cF_n$ and the event $A_{n,1} := A_n\cap \{L_{n+1}\in [k]\}$;
%  \item a pair of 
% \end{enumerate}
 
 Case 1: the initial marked block $\ell$ is a leaf block in $T$. Then we extend our probability space to $(\Omega_{(1)},\cF_{(1)},\bP_{(1)})$ to include a process $\big(U_{(1)}^y,y\ge0\big)$ so that, given the sub-$\sigma$-algebra of $\cF_{(1)}$ corresponding to $\cF_{(0)}$, it is distributed as a pseudo-stationary type-2 evolution conditioned to have total mass evolution $\big\|U_{(1)}^y\big\| = m_\ell^y$ for $0\le y \le \inf\big\{z\ge 0\colon m_\ell^z = 0\big\}$. Such a conditional distribution exists as $(\cI,d_{\cI})$ is Lusin \cite[Theorem 2.7]{Paper1}, and type-2 evolutions are c\`adl\`ag. As noted in Lemma \ref{lem:type2:symm} and Proposition \ref{prop:012:concat}\ref{item:012concat:BESQ+0}, this total mass process is a \BESQ[-1], so after integrating out this conditioning, $\big(U_{(1)}^y,y\ge0\big)$ is a type-2 evolution by Proposition \ref{prop:012:mass}. We define $D_{(1)}$ to be the degeneration time of $\big(U_{(1)}^y,y\ge0\big)$ and set
 \begin{equation}\label{eq:Dynkin:D1}
  \widetilde D_1 := \widetilde D\wedge D_{(1)}.
 \end{equation}
 Recall the label insertion operator $\oplus$ defined in Section \ref{sec:resamp_def}. We define 
 \begin{equation}\label{eq:Dynkin:insert_ext}
  \wt\cT^y_{k+1} := \wt\cT^y_k\oplus \left(\ell,U_{(1)}^y/\left\|U_{(1)}^y\right\|\right), \quad y\in \left[0,\widetilde D_1\right).
 \end{equation}
 Then this is a killed $(k\!+\!1)$-tree evolution with initial law $\Ast\Lambda_k((T,\ell),\cdot\,)$, in which the type-2 compound containing label $k\!+\!1$ equals $\big(U_{(1)}^y,y\ge0\big)$, up to relabeling.
 
 Case 2: $\ell$ is an internal block $(E,a,b)$ in a type-$d$ edge in $T$, $d = 0,1,2$. We consider the case $d=2$, so $E = \{i,j\}$ for some $i,j\in [k]$; the other cases can be handled similarly. As noted in Remark \ref{rmk:decomp_ker}, there exists a stochastic kernel $\kappa$ that takes the path of a type-$2$ evolution and a block in the interval partition component at time zero and specifies the conditional law of a type-2 evolution and a type-1 evolution, stopped at the lesser of their two degeneration times, conditioned to concatenate to form the specified path split around the specified block. Moreover, after mixing over the law of the type-$2$ evolution the constituent evolutions are independent.  In this case, we are interested in such a pair with conditional law
 \begin{equation}
  \left( \Gamma_{(1)}^y,\; \left(m_{(1)}^y,\alpha_{(1)}^y\right),\; y\in [0,D_{(1)})\right) \sim
  	\kappa\left( \left( \left( \left(\wt m_i^y,\wt m_j^y,\wt\alpha_E^y \right),y\ge0\right),\,(a,b)\right),\,\cdot\,\right).
 \end{equation}
 To be clear, concatenating these two processes in the sense of \eqref{eq:012concat:2+1} would yield $\big( \big(\wt m_i^y,\wt m_j^y,\wt\alpha_E^y \big),y\in [0,D_{(1)})\big)$ prior to the first time $D_{(1)}$ that one of $\big(\Gamma_{(1)}^y\big)$ or $\big(\big(m_{(1)}^y,\alpha_{(1)}^y\big)\big)$ degenerates, and in this concatenation, the top mass $m_{(1)}^0$ corresponds to the block $(a,b)\in \wt\alpha_E^0$. We extend our probability space to $(\Omega_{(1)},\cF_{(1)},\bP_{(1)})$ to include a pair with this conditional law given the sub-$\sigma$-algebra of $\cF_{(1)}$ corresponding to $\cF_{(0)}$.
 
 We define $a_{(1)}^y$ to equal the mass of the interval partition component of $\Gamma_{(1)}^y$ and we set $b_{(1)}^y := a_{(1)}^y + m_{(1)}^y$ for $y\in [0,D_{(1)})$. We define $\wt D_1$ as in \eqref{eq:Dynkin:D1}. Then we set
 \begin{equation}\label{eq:Dynkin:insert_int}
  \wt\cT^y_{k+1} := \wt\cT^y_k\oplus \left(\left(E,a_{(1)}^y,b_{(1)}^y\right),\,k+1,\,U \right),\quad y\in \big[0,\widetilde D_1\big),
 \end{equation}
 where $U$ is an arbitrary 2-tree, say $(1/2,1/2,\emptyset)$, which, we recall from Section \ref{sec:resamp_def}, is redundant in the label insertion operator when inserting into an internal block. 
 \medskip
 
 In each case, the constructed process $\big(\widetilde\cT^y_{k+1},y\in [0,\widetilde D_1)\big)$ is a killed $(k\!+\!1)$-tree evolution with initial distribution $\Ast\Lambda_k((T,\ell),\cdot\,)$. Recall the three cases, \ref{case:degen:type2}, \ref{case:degen:self}, and \ref{case:degen:nephew}, in which label $k\!+\!1$ may resample in a resampling evolution. Case \ref{case:degen:type2} corresponds to Case 1 above, and the event that in that case, $D_{(1)} < \widetilde D$. Cases \ref{case:degen:self} and \ref{case:degen:nephew} correspond to Case 2 above, and the events that in that case, $\big(\big(m_{(1)}^y,\beta_{(1)}^y\big)\big)$ or $\big(\Gamma_{(1)}^y\big)$, respectively, are the first to degenerate among $\big(\big(m_{(1)}^y,\beta_{(1)}^y\big)\big)$, $\big(\Gamma_{(1)}^y\big)$, and $\big(\widetilde\cT^y_k\big)$. In other words, we have the equality of events
\begin{equation}
  A_1 := \{D_{(1)} < \widetilde D\} = \left\{J\left(\wt\cT_{k+1}^{\widetilde D_1-}\right) = k\!+\!1\right\}.
 \end{equation}
 
 We now confirm that on $A_1$,
 \begin{equation}\label{eq:swapred:k+1}
  \varrho\left(\wt\cT^{\widetilde D_1-}_{k+1}\right) = \wt\cT^{\wt D_1}_k.
 \end{equation}
 In Case 1 above, this is clear: we had split the block $\ell$ into a type-2 compound, and at time $\widetilde D_1$, regardless of whether label $k\!+\!1$ or label $\ell$ was the cause of degeneration in the compound, after applying the swap-and-reduce map $\varrho$, label $k\!+\!1$ get dropped and the single remaining mass bears label $\ell$. In Case 2, on the event that label $k\!+\!1$ was the cause of degeneration, then this is again clear: $\big(m_{(1)}^{\wt D_1-},\beta_{(1)}^{\wt D_1-}\big) = (0,\emptyset)$, so all that remains on the compound containing edge $E$ is $\Gamma_{(1)}^{D_1-}$.
 
 It remains to confirm \eqref{eq:swapred:k+1} in Case 2, on the event that one of the nephews of label $k\!+\!1$, corresponding to one of the labels in $\big(\Gamma_{(1)}^y\big)$, is the cause of degeneration. We will address the case where the edge $E$ containing the marked block is a type-2 edge, $E = \{u,v\}$, and say label $u$ causes degeneration; the type-1 case is similar. Then, in $\wt\cT_{k+1}^{\wt D_1-}$, block $u$ and edge $\{u,v\}$ both have mass zero; but in $\varrho\big(\cT_{k+1}^{\wt D_1-}\big)$, label $u$ displaces label $k\!+\!1$, and the edge that was formerly $\longparent{\{k\!+\!1\}} = \{u,v,k\!+\!1\}$ gets relabeled as $\{u,v\}$, so that the newly labeled block $u$ has mass $m_{(1)}^{\wt D_1}$ while edge $\{u,v\}$ bears the partition $\alpha_{(1)}^{\wt D_1}$. This is consistent with the second line of the formula in Proposition \ref{prop:012:concat}\ref{item:012concat:2+1}, so that $\big(m_{(1)}^{\wt D_1},\alpha_{(1)}^{\wt D_1}\big) = \big(\wt m_u^{\wt D_1},\wt\alpha_E^{\wt D_1}\big)$. Thus, again, $\varrho\big(\wt\cT^{\widetilde D_1-}_{k+1}\big) = \wt\cT^{\wt D_1}_k$.
 
 We now extend this construction recursively. Suppose that we have defined $\big(\wt\cT^y_{k+1},y\in [0,\wt D_n)\big)$ on some extension $(\Omega_{(n)},\cF_{(n)},\bP_{(n)})$ of $(\Omega_{(0)},\cF_{(0)},\bP_{(0)})$ so that this is distributed as a resampling $(k\!+\!1)$-tree evolution stopped either just before the first time that a label in $[k]$ resamples or the $n^{\text{th}}$ time that label $k\!+\!1$ resamples, whichever comes first. Suppose also that $\wt\cT^y_k = \pi_k\big(\wt\cT^y_{k+1}\big)$ for $y\in [0,\wt D_n)$ and that, on the event $A_n$ that the $n^{\text{th}}$ resampling time for label $k\!+\!1$ precedes the first time that a lower label would resample, also 
$\varrho\big(\wt\cT^{\widetilde D_n-}_{k+1}\big) = \wt\cT^{\wt D_n}_k$.
%\eqref{eq:swapred:k+1} holds also at time $\wt D_n$.
 
 We further extend our probability space to $(\Omega_{(n+1)},\cF_{(n+1)},\bP_{(n+1)})$ to include random objects with the following conditional distributions given the sub-$\sigma$-algebra of $\cF_{(n+1)}$ that corresponds to $\cF_{(n)}$.
 \begin{itemize}
  \item A random block $L_{n}$ conditionally distributed as a size-biased pick from $\block\big(\wt\cT_k^{\wt D_n}\big)$ given $A_n$, and that equals a cemetery state $\partial$ on $A_n^c$.
  \item A process $\big(U_{(n+1)}^y,y\in [0,D_{(n+1)}\big)$ that, if we additionally condition on $A_{n,1} := \{L_n\in [k]\}$, is conditionally distributed as a pseudo-stationary type-2 evolution stopped at degeneration, conditioned to have total mass process $\big\|U_{(n+1)}^y\| = m_{L_n}^{\wt D_n + y}$, $y\in [0,D_{(n+1)})$. On $A_{n,1}^c$ we instead define this to be the constant process at $\partial$.
  \item A pair of processes that, if we additionally condition on $L_n = (E,a,b)\in \block\big(\wt\cT^{\wt D_n}_k\big)\setminus [k]$, have conditional law 
   \begin{equation}
    \left( \Gamma_{(n+1)}^y,\; \left(m_{(n+1)}^y,\alpha_{(n+1)}^y\right),\; y\in [0,D_{(n+1)})\right) \sim
    	\kappa\left( \left( \left( \wt \Gamma_E^{\wt D_n + y},y\ge0\right),\,(a,b)\right),\,\cdot\,\right),
   \end{equation}
   where $\big(\wt\Gamma_E^y,y\ge0\big)$ denotes the evolution on the type-0/1/2 compound in $\big(\wt\cT^y_k\big)$ containing edge $E$. We define these to be constant $\partial$ processes on the event $A_n^c\cup A_{n,1}$.
 \end{itemize}
 
 We define $\wt D_{n+1} := \wt D\wedge (\wt D_n + D_{(n+1)})$. On $A_{n,1}$, we define $\big(\wt\cT^y_{k+1},y\in [\wt D_n,\wt D_{n+1})\big)$ as in Case 1 above, with the obvious modifications. On $A_{n,2} := A_n\setminus A_{n,1}$, we define this process as in Case 2 with the obvious modifications. On $A_n^c$, we get $\wt D_{n+1} = \wt D_n = \wt D$, and we do not define $\wt\cT^y_{k+1}$ for $y\ge \wt D$.
 
 As required for our induction, $\big(\cT^y_{k+1},y\in [0,\wt D_{n+1})\big)$ is distributed as a resampling $(k\!+\!1)$-tree evolution stopped at either the first time a label in $[k]$ would resample or the $n\!+\!1^{\text{st}}$ time that label $k\!+\!1$ would resample, and $\wt\cT^y_k = \pi_k\big(\wt\cT^y_{k+1}\big)$ for $y\in [0,\wt D_{n+1})$. By the same arguments as before, \eqref{eq:swapred:k+1} holds at $\wt D_{n+1}$ on the event $A_{n+1}$ that label $k\!+\!1$ resamples an $n\!+\!1^{\text{st}}$ time before the first time that a lower label would resample. 
 
 By the Ionescu Tulcea theorem \cite[Theorem 6.17]{Kallenberg}, there is a probability space $(\Omega_\infty,\cF_{\infty},\bP_{\infty})$ on which we can define a resampling $(k\!+\!1)$-tree evolution $\big(\wt\cT^y_{k+1}\big)$ run until either the times at which label $k\!+\!1$ resamples have an accumulation point $D_{(\infty)}$ or a lower label would resample for the first time, with $\pi_k\big(\wt\cT^y_{k+1}\big) = \wt\cT^y_k$ in this time interval $[0,\wt D_\infty)$ and  $\wt\cT^0_{k+1}\sim \Ast\Lambda_k((T,\ell),\cdot\,)$. The claim that $\big(\wt\cT^y_{k},y\in [0,\wt D)\big)$ is Markovian in the filtration generated by itself and $\big(\wt\cT^y_{k+1},y\in [0,\wt D_\infty)\big)$ follows from our construction and the assertions concerning filtrations at the ends of Propositions \ref{prop:012:concat} and \ref{prop:012:mass}.%180917
\end{proof}

We now prove a pair of results, one intertwining-like and the other Dynkin-like, regarding the behavior of $\cT_{k+1}^y$, $\Ast\cT_k^y$, and $\cT_k^y$ at the degeneration times $D''_n$. Note that at such times, $\cT_k^{D''_n-} = \pi_{-(k+1)}\big(\cT_{k+1}^{D''_n-}\big)$ is degenerate, with
$$\left(I\!\left(\cT_k^{D''_n-}\right),J\!\left(\cT_k^{D''_n-}\right)\right) = \left(I\!\left(\cT_{k+1}^{D''_n-}\right),J\!\left(\cT_{k+1}^{D''_n-}\right)\right) 
 \quad\text{and}\quad 
 \varrho\!\left(\cT_k^{D''_n-}\right) = \pi_{-(k+1)}\circ\varrho\!\left(\cT_{k+1}^{D''_n-}\right)\!.$$

\begin{lemma}\label{lem:preswap}
 Fix $n\ge1$. Given $\big(\Ast\cT_k^y,y\in [0,D''_n)\big)$ and the event $\{D_n''<D_\infty\}$, with $\Ast\cT_k^{D''_n-} = (U,L)$ and $J\big(\cT_{k}^{D''_n-}\big) = j\in [k]$,
 \begin{enumerate}
  \item $\cT_{k+1}^{D''_n-}$ has conditional law $\Ast\Lambda_k((U,L),\cdot\,)$ and\label{item:preswap:inter}
  \item $\cT_{k}^{D''_n}$ has conditional law $\Lambda_{j,[k]\setminus\{j\}}(\varrho(U),\cdot\,)$.\label{item:preswap:Dynkin}
 \end{enumerate}
\end{lemma}

These assertions are illustrated in Figure \ref{fig:preswap}.

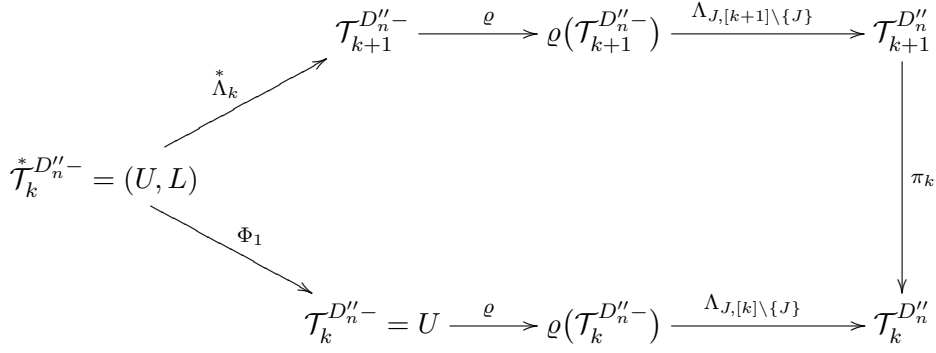
\begin{figure}[t]
 \centerline{
 \xymatrix@=3em{
  & \cT^{D''_n-}_{k+1} \ar[r]^{\varrho} &  \varrho\big(\cT^{D''_n-}_{k+1}\big) \ar[rr]^{\Lambda_{J,[k+1]\setminus\{J\}}} && \cT^{D''_n}_{k+1} \ar[dd]^{\pi_k}\\%\ar@/^/[d]^{\phi_2} \\
  \Ast\cT^{D''_n-}_k = (U,L) \ar[ur]^{\Ast\Lambda_k}\ar[dr]^{\Phi_1} &&& \\%\Ast\cT^{D''_n}_k\ar@/^/[u]^{\Ast\Lambda_k}\ar[d]^{\phi_1}\\
  & \cT^{D''_n-}_k = U  \ar[r]^{\varrho} & \varrho\big(\cT^{D''_n-}_{k}\big) \ar[rr]^{\Lambda_{J,[k]\setminus\{J\}}} && \cT^{D''_n}_k
 }
 } 
 \caption{Lemma \ref{lem:preswap}\ref{item:preswap:inter} asserts that the conditional law of $\cT^{D''_n-}_{k+1}$ given $\Ast{\cT}^{D''_n-}_k$ is as in the upper left arrow in this diagram. By construction, the conditional law of $\cT^{D''_n}_k$ given $\cT^{D''_n-}_{k+1}$ is as in the upper path in this diagram; Lemma \ref{lem:preswap}\ref{item:preswap:Dynkin} claims that this diagram commutes.\label{fig:preswap}}
\end{figure}

\begin{proof}
 \ref{item:preswap:inter} Observe that $D_n''$ is previsible for $\Ast\cT^y_{k}$ because, if we let $C_a''$ denote the first time after $D_{n-1}''$ that for some label $i$ other than $k\!+\!1$ or its sibling or nephews, $m_i^y+\|\beta_{\parent{\{i\}}}^y\|$ is less than $1/a$, for $a\ge 1$, then 
 $C_a''<D_n''$ and $\lim_{a\to\infty}C_a'' = D_n''$.  From Proposition \ref{prop:inter:killed},  
 \[ \bP\left(\cT^{C_a''}_{k+1} \in \,\cdot\, \middle|\, \big(\Ast\cT_k^y,y\in [0,C''_a]\big)\right) = \Ast\Lambda_k(\Ast\cT_k^{C''_a},\cdot\,). \]
 Since $\cT^{C_a''}_{k+1} \rightarrow \cT^{D_n''-}_{k+1}$ almost surely, and $\sigma \big(\Ast\cT_k^y,y\in [0,C''_a]\big) \uparrow \sigma \big(\Ast\cT_k^y,y\in [0,D''_n)\big)$ by \cite[Lemma VI.17.9]{RogersWilliams2}, it follows from \cite[Theorem 5.5.9]{Durrett} and Lemma \ref{lem:Lstar_cont} that
 \[\bE \left[ f\left(\cT^{D_n''-}_{k+1}\right)\, \middle|\, \big(\Ast\cT_k^y,y\in [0,D''_n-)\big)\right] = \int f(T)  \Ast\Lambda_k(\Ast\cT_k^{D''_n-},dT\,)\]
 for every bounded continuous function $f$, as desired.
 
 \ref{item:preswap:Dynkin}  To compute $\bE(F(\Ast\cT_k^y,y\in[0,D_n''))G(\cT_k^{D_n''}))$, we apply \ref{item:preswap:inter} to write this expectation as
  $$\bE\left(F\!\left(\Ast\cT_k^y,y\!\in\![0,D_n'')\right)\!\int_{W_1\in\bT_{[k+1]}}\int_{W_2\in\bT_{[k+1]}}\!\!G(\pi_k(W_2))\Lambda_{j,[k+1]\setminus\{j\}}(\varrho(W_1),dW_2)\Ast\Lambda_k(\Ast\cT_k^{D_n''-},dW_1)\!\right)\!,$$
for functionals $F(\Ast\cT_k^y,y\in[0,D_n''))$ that vanish outside $\{J(\cT_k^{D_n''-})=j\}$, for some $j\in[k]$. Thus,   
it suffices to show 
 $$\int_{\tdTInt_{k+1}}\Ast\Lambda_k((U,L),dT')\int_{\TInt_{k+1}}\Lambda_{j,[k+1]\setminus\{j\}}(\varrho(T'),dT'')f\circ\pi_k(T'') = \int_{\tdTInt_{k}}\Lambda_{j,[k]\setminus\{j\}}(\varrho(U),dT''')G(T''').$$
 
 This is trivial in the case that $L$ marks an internal block in $U$: in that case, $\cT_{k+1}^{D''_n-}$ is a deterministic function of $\Ast\cT_k^{D''_n-} = (U,L)$, as in Figure \ref{fig:mark_k_tree}(A), and there is a natural weight- and tree-structure-preserving bijection between the blocks of the former and those of the latter, allowing us to couple $\Lambda_{j,[k+1]\setminus\{j\}}\big(\varrho\big(\cT^{D''_n-}_{k+1}\big),\cdot\,\big)$ with $\Lambda_{j,[k]\setminus\{j\}}\big(\varrho\big(\Ast\cT^{D''_n-}_{k}\big),\cdot\,\big)$.
 
 Henceforth, we assume that $L = i\in [k]$ with block mass $x_i$ in $U$. Let $H$ denote the event that, after resampling, leaf $j$ is the sibling, uncle, or nephew of leaf $k\!+\!1$ in $\cT_{k+1}^{D''_n}$. Equivalently, $H$ is the event that in the marked $k$-tree, label $j$ resamples into the marked block $i$ of $\varrho\big(\cT_k^{D''_n-}\big)$ so that after resampling, the marking sits somewhere in a type-2 compound in $\Ast\cT_k^{D''_n}$ containing label $j$. Again, the assertion is trivial on the event $H^c$, as then there is a weight- and tree-structure-preserving bijection between the \emph{remaining} blocks of the trees, i.e.\ the unmarked blocks in $\cT_k^{D''_n-}$ and the blocks outside of the type-2 compound containing $k\!+\!1$ in $\cT_{k+1}^{D''_n-}$. The only remaining event is when $L = i\in [k]$ and $H$ holds.
 
 Given $\Ast\cT_k^{D''_n-} = (U,L)$, the event $H$ has conditional probability $x_i / \|U\|$; in particular, it is conditionally independent of the normalized ``internal structure'' of the type-2 compound containing label $k\!+\!1$ in $\cT_{k+1}^{D''_n-}$,
 $$x_i^{-1}\left(m_i^{D''_n-},m_{k+1}^{D''_n-},\alpha_{\{i,k+1\}}^{D''_n-}\right) \quad \text{where }m_i^{D''_n-} + m_{k+1}^{D''_n-} + \left\|\alpha_{\{i,k+1\}}^{D''_n-}\right\| = x_i.$$
 By assertion \ref{item:preswap:inter}, given $\Ast\cT^{D''_n-}_{k} = (U,L)$ and $L=i$, this internal structure is conditionally a Brownian reduced 2-tree. By \eqref{eq:B_ktree_resamp} and the exchangeability of labels in Brownian reduced $k$-trees noted in Proposition \ref{prop:B_ktree}, this means that after inserting label $j$, blocks $i$, $j$, and $k\!+\!1$, along with the partitions marking their parent edges, comprise an independently scaled Brownian reduced 3-tree. Thus, the $\pi_{-(k+1)}$-projection of this 3-tree is an independently scaled Brownian reduced 2-tree, as required for $\Lambda_{j,[k]\setminus\{j\}}$ on the event $H$.
 %the 2-tree projection of this tree that projects away label $k\!+\!1$ 
\end{proof}

\begin{proof}[Proof of Proposition \ref{prop:consistency_0}]
 By Proposition \ref{prop:Dynkin:killed} and Lemma \ref{lem:preswap}\ref{item:preswap:Dynkin}, $\cT^{y}_k = \pi_k\big(\cT^y_{k+1}\big)$ evolves as a resampling $k$-tree evolution up to time $D''_1$. This holds up to time $D_\infty$ by induction and the strong Markov property applied at the degeneration times $D''_n$, $n\ge1$. %The full assertion of Proposition \ref{prop:consistency_0}, that we can project away multiple labels and, given a suitable initial distribution, still arrive back at a resampling $j$-tree evolution, follows by another induction argument.
\end{proof}

In fact, we have shown that the map $\phi_1$ on $\big(\Ast\cT^y_k,y\ge0\big)$ satisfies Dynkin's criterion.

\section{Accumulation of degeneration times as mass hits zero}\label{sec:non_acc}

We now possess all major ingredients needed to prove Proposition \ref{prop:non_accumulation}, that $D_\infty := \sup_nD_n$ equals $\inf\{y\ge0\colon \|\cT^{y-}\|=0\}$ for a resampling $k$-tree evolution. This proposition immediately completes the proofs of Theorems \ref{thm:total_mass} and \ref{thm:consistency}\ref{item:cnst:resamp}, that total mass of a resampling $k$-tree evolves as a \BESQ[-1] and the projective consistency of resampling $k$-tree evolutions, from the partial results in Propositions \ref{prop:total_mass_0} and \ref{prop:consistency_0}, respectively.

We require the following lemma.

\begin{lemma}\label{lem:degen_diff}
 Fix $k\ge 3$ and $\epsilon>0$. Let $T\in\TInt_{k-1}$ with $\|T\|>\epsilon$ and let $(\cT^y,y\ge0)$ be a resampling $k$-tree evolution with $\cT^0\sim\Lambda_{k,[k-1]}(T,\cdot\,)$. Let $(D^*_n,n\ge1)$ denote the subsequence of degeneration times at which label $k$ is dropped and resamples. Assume that with probability one we get $D_\infty > D^*_2$. Then there is some $\delta = \delta(k,\epsilon)>0$ that does not depend on $T$ such that $\bP(D^*_2>\delta)>\delta$.
\end{lemma}

The proof of this lemma is somewhat technical, so we postpone it until Appendix \ref{sec:non_acc_2}.

\begin{proof}[Proof of Proposition \ref{prop:non_accumulation} using Lemma \ref{lem:degen_diff}]
 By Proposition \ref{prop:total_mass_0}, the total mass $\|\cT^y\|$ of a resampling $k$-tree evolution evolves as a \BESQ[-1] stopped at a random stopping time $D_\infty$ that is not necessarily measurable in the filtration of the total mass process. By the independence of the type-$i$ evolutions in the compounds of the $k$-tree in Definition \ref{def:killed_ktree} and continuity of the distributions of their degeneration times, there is a.s.\ no time at which two compounds degenerate simultaneously. Thus, since a \BESQ[-1] a.s.\ hits zero in finite time, there are a.s.\ infinitely many degenerations in finite time: $D_\infty < \infty$. In fact, the lifetime of a $\besq_x(-1)$ before absorption has law \InvGammaDist[3/2,x/2] by \cite[equation (13)]{GoinYor03}. Thus, $\bE[D_\infty] < \infty$.
 
 We will prove this by showing that for every $\epsilon\in (0,\|\cT^0\|)$ the time $H_\epsilon := \inf\{y\ge0\colon 0<\|\cT^y\|\le\epsilon\}$ is a.s.\ finite. This implies that these times have a limit in $[0,D_\infty]$ at which time $\|\cT^y\|$ converges to zero, by continuity. Thus, by the argument of the previous paragraph, this limit equals $D_\infty$, which completes the proof. We prove $H_\epsilon<\infty$ by showing that 
 \begin{equation}\label{eq:degen_diff}
  \bP\big(D_{j+2}-D_{j} > \delta\ \big|\ H_{\epsilon}>D_{j}\big) > \delta \quad \text{for }j\ge n
 \end{equation}
 for some sufficiently small $\delta>0$ and sufficiently large $n$. 
 This implies %\sum_{j\ge0}\delta \bP(D_{2(j+1)}-D_{2j} > \delta\ |\ H_{\delta}>D_{2j})\bP\{H_{\delta} > D_{2j}\}\\
 \begin{equation*}
  \infty > \bE[D_\infty] > \sum_{j \ge n}\delta^2 \bP\{H_{\epsilon} > D_{2j}\}.
 \end{equation*}
 From this, it follows by the Borel--Cantelli Lemma that $H_\epsilon$ is a.s.\ finite, as desired. We proceed to verify \eqref{eq:degen_diff}.
 
 Fix $\epsilon\in (0,\|\cT^0\|)$. We proceed by induction on the number of labels. Consider a resampling 2-tree evolution. There is some $\delta>0$ such that a pseudo-stationary 2-tree evolution with initial mass $\epsilon$ will \emph{not} degenerate prior to time $\delta$ with probability at least $\delta$. By the scaling property, Lemma \ref{lem:scaling}, the same holds for any larger initial mass with the same $\delta$, proving \eqref{eq:degen_diff} in this case.
 
 %\texttt{NOTE: Editing in progress. Resume here.} 
 
 %Suppose for a contradiction that $\lim_{y\to D_\infty}\|\cT^y\| \neq0$. Then by Proposition \ref{prop:total_mass_0}, $(\|\cT^y\|,\,y\in[0,D_\infty))$ is a \BESQ[-1] that doesn't hit or approach zero; thus, it must be bounded away from zero, with some infimum $\delta>0$. There is some $\epsilon>0$ such a 2-tree evolution with pseudo-stationary initial state with mass $\delta$ will \emph{not} degenerate prior to time $\epsilon$ with probability at least $\epsilon$. By the scaling property noted in Lemma \ref{lem:scaling}, the same inequality holds for a pseudo-stationary initial state with greater mass. Thus, there must be infinitely many $n$ for which $D_{n+1}-D_n > \epsilon$. This is a contradiction, as the \BESQ[-1] total mass a.s.\ hits zero in finite time.
 
 Now, suppose that the proposition holds for $k$-tree evolutions and consider a resampling $(k\!+\!1)$-tree evolution $(\cT^y,y\ge0)$. By Proposition \ref{prop:consistency_0}, $(\pi_k(\cT^y),y\ge0)$ is a resampling $k$-tree evolution up to the accumulation time $D_\infty$ of degenerations of the $(k\!+\!1)$-tree evolution. The degeneration times of $(\pi_k(\cT^y))$ are the times at which a label less than or equal to $k$ resamples in $(\cT^y)$. By the inductive hypothesis, these degeneration times do not have an accumulation point prior to the extinction time of the \BESQ[-1] total mass. Thus, $D_\infty$ must equal the accumulation point of degeneration times $(D^*_j,j\ge1)$ at which label $k\!+\!1$ resamples. Lemma \ref{lem:degen_diff} now yields \eqref{eq:degen_diff} with $(D_m,m\ge1)$ replaced by $(D_m^*,m\ge1)$.
\end{proof}

\section{Proofs of remaining consistency results}\label{sec:const:other}

The remaining assertions of Theorem \ref{thm:consistency} largely follow from our arguments in the proof of Proposition \ref{prop:Dynkin:killed}. %and the work already done. 
We summarize the proofs of these results below.

\begin{proof}[Proof of Theorem \ref{thm:consistency}\ref{item:cnst:nonresamp}]%180917
 Suppose $\big(\cT_{k+1}^y,y\ge0\big)$ is a non-resampling $(k\!+\!1)$-tree evolution and let $\cT_{k}^y = \pi_{k}\big(\cT_{k+1}^y\big)$, $y\ge0$. For the purpose of this argument, let $D^*$ denote the time at which label $k\!+\!1$ is dropped in degeneration and let $(D'_n,n\in [k])$ denote the sequence of times at which labels in $[k]$ are dropped. For $y\ge D_*$, $\cT_{k}^y = \cT_{k+1}^y$, and both evolve from time $D^*$ onwards as non-resampling $k$-tree evolutions. From a slight extension of the argument of Proposition \ref{prop:Dynkin:killed}, allowing $\cT_{k+1}^0$ to be any $(k\!+\!1)$-tree satisfying $\pi_k\big(\cT_{k+1}^0\big) = \cT_k^0$ in Case 1 in that proof, we find that $\cT_{k}^y$ evolves as a $k$-tree evolution up to time $D'_1$. At this degeneration time,
 $$\cT_k^{D'_1} = \pi_{k}\circ\varrho\big(\cT_{k+1}^{D'_1-}\big) = \varrho\circ \pi_{k}\big(\cT_{k+1}^{D'_1-}\big) = \varrho\big(\cT_k^{D'_1-}\big),$$
 as in Definition \ref{def:nonresamp_1} of non-resampling $k$-tree evolutions. After this time, given that label $k\!+\!1$ has not yet been dropped, then $\cT_{k+1}^y$ continues as a non-resampling $([k\!+\!1]\!\setminus\!\{j\})$-tree evolution, where $j$ was the first label dropped. Then by the same argument as above, $\cT_k^y$ evolves until the next degeneration time as a $([k]\!\setminus\!\{j\})$-tree evolution. By an induction applying the strong Markov property of $\big(\cT_{k+1}^y\big)$ at the times $(D'_n,n\in [k])$, the process $(\cT_k^y,y\ge0)$ is a non-resampling $k$-tree evolution. A second induction argument allows one to project away multiple labels, rather than just one. 
\end{proof}

\begin{proof}[Proof of Theorem \ref{thm:consistency}\ref{item:cnst:dePoi}]
 Fix $1\le j<k$. Suppose $\fT_k := (\cT_k^y,y\ge0)$ is a resampling $k$-tree evolution with initial distribution as in \eqref{eq:cnst:init}, so that $\fT_j = (\cT_j^y,y\ge0) := (\pi_j(\cT_k^y),y\ge0)$ is a resampling $j$-tree evolution. Then because these evolutions have the same total mass process, they require the same time change for de-Poissonization: $(\rho_u(\fT_k),u\!\ge\!0) = (\rho_u(\fT_j),u\!\ge\!0)$. \linebreak Thus, the associated de-Poissonized processes are also projectively consistent. The same argument holds in the non-resampling case.
\end{proof}

We can now also prove Proposition \ref{prop:resamp_to_non}.

\begin{proof}[Proof of Proposition \ref{prop:resamp_to_non}]
 Suppose $\big(\cT_{k,+}^y,y\ge0\big)$ is a resampling $k$-tree evolution. Let $(D_n,n\ge1)$ denote its sequence of degeneration times, and set $D_0 := 0$. Recall that we consider each edge in a tree shape to be labeled by the set of all labels of leaves in the subtree above that edge.  Recall the definition of $\varrho$ in Section \ref{sec:non_resamp_def}: when a label $i_n := I\big(\cT_{k,+}^{D_n-}\big)$ causes degeneration, it swaps places with label $j_n := J\big(\cT_{k,+}^{D_n-}\big) = \max\{i_n,a_n,b_n\}$, where $a_n$ and $b_n$ are respectively the least labels on the sibling and uncle of leaf edge $\{i_n\}$ in the tree shape of $\cT_{k,+}^{D_n-}$. In the resampling evolution, label $j_n$ is resampled.
 
 We extend this notation slightly. Let $E_{n}^{(a)}$ and $E_n^{(b)}$ denote the sets of labels on the sibling and uncle of edge $\{i_n\}$, so that $a_n = \min\big(E_n^{(a)}\big)$ and $b_n = \min\big(E_n^{(b)}\big)$. Let $\tau_n$ denote the transposition permutation that swaps $i_n$ with $j_n$.
 
 Set $A_0 := B_0 := [k]$ and let $\sigma_0$ denote the identity map on $[k]$. Now suppose for a recursive construction that we have defined $(A_{n-1},B_{n-1},\sigma_{n-1})$. We consider four cases.\smallskip
 
 \begin{enumerate}[label = Case \arabic*:, ref = \arabic*, topsep=5pt, itemsep=5pt, itemindent=1.82cm, leftmargin=0pt]
  \item  $i_n\notin A_{n-1}$ and $j_n\notin A_{n-1}$. In this case, the degeneration, swap-and-reduce map, and resampling in $\cT_k^{D_n}$ are invisible under $\sigma_{n-1}\circ \pi_{A_{n-1}}$, since the projection erases both labels involved. We set $(A_n,B_n,\sigma_n) := (A_{n-1},B_{n-1},\sigma_{n-1})$.
  
  \item $i_n\notin A_{n-1}$ and $j_n\in A_{n-1}$. In this case, the label $i_n$ that has caused degeneration is invisible under $\pi_{A_{n-1}}$, so there is no degeneration in the projected process, but $i_n$ displaces a label that \emph{is} visible. To maintain continuity in the projected process at this time, $i_n$ takes the place of $j_n$ in such a way that $\sigma_n(i_n) = \sigma_{n-1}(j_n)$. In particular, $A_n:= (A_{n-1}\setminus\{j_n\})\cup \{i_n\}$, $B_n := B_{n-1}$, and $\sigma_n := \sigma_{n-1}\circ\tau_n|_{A_n}$.
  
  \item \label{case:r2n:degen} $i_n\in A_{n-1}$ and $E_n^{(a)}$ and $E_n^{(b)}$ both intersect $A_{n-1}$ non-trivially. Let $\wi_n := \sigma_{n-1}(i_n)$. In this case, the degeneration caused by $i_n$ in $\cT_{k,+}$ corresponds to a degeneration caused by $\wi_n$ in $\cT_{k,-}$.
  
  Let $\tilde a_n := \min\!\big(\sigma_{n-1}\big(E_n^{(a)}\cap A_{n-1}\big)\!\big)$ and $\tilde b_n = \min\!\big(\sigma_{n-1}\big(E_n^{(b)}\cap A_{n-1}\big)\!\big)$. Let $\wj_n := \max\{\wi_n,\tilde a_n,\tilde b_n\}$ and let $\tilde\tau_n$ denote the transposition permutation that swaps $\wi_n$ with $\wj_n$. If $j_n\in A_{n-1}$ then we set $A_n := A_{n-1}\setminus\{j_n\}$; otherwise, we set $A_n := A_{n-1}\setminus\{i_n\}$. In either case, we define $B_n := B_{n-1}\setminus\{\wj_n\}$ and $\sigma_n := \tilde\tau_n\circ\sigma_{n-1}\circ\tau_n|_{A_n}$.
  
  \item \label{case:r2n:shrink} $i_n\in A_{n-1}$ and $E_n^{(a)}$ is disjoint from $A_{n-1}$. Then leaf block $i_n$ and the subtree that contains label set $E_n^{(a)}$ in $\cT_{k,+}^y$ project down to a single leaf block, $\sigma_{n-1}(i_n)$, in $\cT_{k,-}^y$ as $y$ approaches $D_n$. By leaving open the possibility that $E_n^{(b)}$ may be disjoint from $A_{n-1}$ as well, we include in this case the possibility that the subtree of $\cT_{k,+}^y$ with label set $E_n^{(b)}$ projects to this same leaf block as well. Regardless, this degeneration is ``invisible'' in $\cT_{k,-}$. In order to keep label $\sigma_{n-1}(i_n)$ in place in the projected process, if label $i_n$ resamples or swaps with a label in $E_n^{(b)}$, then we choose a label in $E_n^{(a)}$ to map to $\sigma_{n-1}(i_n)$ under $\sigma_n$.
  
  \begin{enumerate}[label = Case \theenumi.\arabic*: , ref=\theenumi.\arabic*, topsep=0pt, itemsep=0pt, itemindent=2cm, leftmargin=0pt]
   \item $j_n = a_n$. Then we define $(A_n,B_n,\sigma_n) := (A_{n-1},B_{n-1},\sigma_{n-1})$.\label{case:r2n:shrink:null}
   
   \item $j_n = i_n$ or $j_n = b_n$. Then let $\hat\tau_n$ denote the transposition that swaps $i_n$ with $a_n$. If $j_n\in A_{n-1}$, as is always the case when $j_n = i_n$, then we set $A_n := (A_{n-1}\setminus\{j_n\})\cup\{a_n\}$. Otherwise, if $j_n\notin A_{n-1}$ then we set $A_n := (A_{n-1}\setminus\{i_n\})\cup\{a_n\}$. In either case, we define $B_n := B_{n-1}$ and $\sigma_n := \sigma_{n-1}\circ\hat\tau_n\circ\tau_n|_{A_n}$.\label{case:r2n:shrink:swap}
  \end{enumerate}
  
  \item $i_n\in A_{n-1}$ while $E_n^{(a)}$ intersects $A_{n-1}$ non-trivially but $E_n^{(b)}$ does not. This degeneration time in $\cT_{k,+}^y$ corresponds to a time at which labeled leaf block $\sigma_{n-1}(i_n)$ in $\cT_{k,-}^y$ has mass approaching zero (really, it is a.s.\ an accumulation point of times at which this mass equals zero) while the interval partition on its parent edge has a leftmost block. The subtree of $\cT^{D_n-}_{k,+}$ that contains the leaf labels $E_n^{(b)}$ maps to a single internal block, the aforementioned leftmost block, in $\cT_{k,-}^{D_n-}$. Therefore, we define $\sigma_n$ in such a way that some label that sits in the subtree corresponding to that block gets mapped to $\sigma_{n-1}(i_n)$, so that this latter label ``moves into'' the leftmost block in the projected process, as in a type-1 or type-2 evolution; see Proposition \ref{prop:012:pred}. In fact, we can accomplish this with the same definitions of $(A_n,B_n,\sigma_n)$ as in Cases \ref{case:r2n:shrink:null} and \ref{case:r2n:shrink:swap}, but with roles of $a_n$ and $b_n$ reversed.
%  
%  \begin{enumerate}[label = \theenumi\alph*: , ref=\theenumi\alph*, topsep=0pt, itemsep=0pt, itemindent=2cm, leftmargin=0pt]
%   \item $j_n\in E_n^{(b)}$. Then we define $(A_n,B_n,\sigma_n) := (A_{n-1},B_{n-1},\sigma_{n-1})$.
%   
%   \item $j_n\notin E_n^{(b)}$. Then let $\hat\tau_n$ denote a transposition that swaps $i_n$ with some member $\hat c_n$ of $E_n^{(b)}$. If $j_n\in A_{n-1}$ then we set $A_n := (A_{n-1}\setminus\{j_n\})\cup\{\hat c_n\}$. Otherwise, if $j_n\notin A_{n-1}$ then we set $A_n := (A_{n-1}\setminus\{i_n\})\cup\{\hat c_n\}$. In either case, we define $B_n := B_{n-1}$ and $\sigma_n := \sigma_{n-1}\circ\hat\tau_n\circ\tau_n|_{A_n}$.
%  \end{enumerate}
 \end{enumerate}
 
 It follows from the consistency result of Theorem \ref{thm:consistency}\ref{item:cnst:nonresamp} that for each $n$, the projected process evolves as a stopped non-resampling $k$-tree evolution (or $B_n$-tree evolution) during the interval $[D_n,D_{n+1})$. By our construction, we have $\cT_{k,-}^{D_n} = \varrho\big(\cT_{k,-}^{D_n-}\big)$ in Case \ref{case:r2n:degen}, as in Definition \ref{def:nonresamp_1} of non-resampling evolutions. In the other cases, it follows from the arguments in the proof of Proposition \ref{prop:Dynkin:killed} that each type-0/1/2 compound in $\cT_{k,-}^{D_n}$ attains the value required by the type-0/1/2 evolution in that compound, given its left limit in $\cT_{k,-}^{D_n-}$; see Proposition \ref{prop:012:pred}. Thus, $\big(\cT_{k,-}^y,y\ge0\big)$ is a non-resampling $k$-tree evolution. 
\end{proof}

\appendix

\section{Intertwining for processes that jump from branch states}\label{sec:intertwining_lem}

In this appendix we state and prove a general lemma that can be used to prove intertwining for a function of a Markov process that attains ``forbidden states'' as left limits in its path from which it jumps away, in the manner of the resampling and non-resampling $k$-tree evolutions at degeneration times. More precisely, we consider a strong Markov process $(X^\bullet(t),t\ge0)$ constructed as follows.

%\texttt{NOTE: The following is copied from Soumik's intertwining file; may need to do something to bring notation in line.}

\newcommand{\wLambda}{\widetilde{\Lambda}}

Let $(\widebar{\bX},d_\bX)$ be a metric space, $\bX\subset\widebar{\bX}$ a Borel subset, $\partial\not\in\widebar{\bX}$ a cemetery state and
set $\bX_\partial:=\bX\cup\{\partial\}$, extending the topology of $\bX$ so that $\partial$ is isolated in $\bX$. We denote the Borel sigma algebra
on $\bX$ by $\cX$. Consider an $\bX_\partial$-valued Borel right Markov process $(X^\circ(t),t\ge 0)$ with transition kernels $(P_t^\circ,t\ge 0)$. 
Suppose that $(X^\circ(t),t\ge 0)$ has left limits in $\widebar{\bX}$ and is absorbed in $\partial$ the first time a left limit is in 
$\widebar{\bX}\setminus\bX$, or by an earlier jump to $\partial$ from within $\bX$. We denote this absorption time by $\zeta$ and refer to 
$(X^\circ(t),t\ge 0)$ as the \em killed Markov process\em. We use the standard setup where our basic probability space supports a family 
$(\bP_x,x\in\bX)$ of probability measures under which $X^\circ$ has initial state $X^\circ(0)=x$. Suppose for simplicity that $\bP_x(\zeta<\infty)=1$
for all $x\in\bX$.

Let $\kappa\colon\widebar{\bX}\times\cX\rightarrow[0,1]$ be a stochastic kernel. We use $\kappa$ as a \em regeneration kernel \em by sampling
$X^\bullet(\zeta)$ from $\kappa(X^\circ(\zeta-),\,\cdot\,)$ and continuing according to the killed Markov process starting from $X^\bullet(\zeta)$. 
More formally, let $X^\bullet(0)=x\in\bX$ and $S_0=0$. Inductively, given $(X^\bullet(t),0\le t\le S_n)$ for any $n\ge 0$, let $(X^\circ_n(t),t\ge 0)$
be a killed Markov process starting from $X^\bullet(S_n)$ with absorption time $\zeta_n$, set $X^\bullet(S_n+t)=X^\circ_n(t)$, $0\le t<\zeta_n$,
and $S_{n+1}=S_n+\zeta_n$, and then sample $X^\bullet(S_{n+1})$ from the regeneration kernel $\kappa(X^\bullet(S_{n+1}-),\,\cdot\,)$. Finally, set
$X^\bullet(t)=\partial$ for $t\ge S_\infty:=\lim_{n\rightarrow\infty}S_n$. 

Then $\widetilde{P}((x,s),\,\cdot\,)=\bP_x((X^\bullet(\zeta),s+\zeta)\in\,\cdot\,)$ is clearly a Markov transition kernel. Meyer \cite{Mey75} showed
that $(X^\bullet(t),t\ge 0)$ is a (Borel right) Markov process, and we denote its transition kernels by $(P_t^\bullet,t\ge 0)$.  

We recall Definition \ref{def:intertwining} of intertwining. Consider a measurable map $\phi\colon\widebar{\bX}\to\widebar{\bY}$ to another metric
space $(\widebar{\bY},d_\bY)$ with $\partial\notin\widebar{\bY}$. We extend this map, defining $\phi(\partial) = \partial$ and set 
$\bY_\partial:=\phi(\bX)\cup\{\partial\}$. Let $\Phi(x,\cdot)=\delta_{\phi(x)}$ denote the trivial kernel associated with $\phi$. Let $\Lambda$ denote a stochastic kernel from $\bY$ to $\bX$ such that $\Lambda \Phi$ is the identity.

We define $Y^\bullet(t) := \phi(X^\bullet(t))$ and $Y^\circ(t) := \phi(X^\circ(t))$, $t\ge0$. We set $Q^\circ_t := \Lambda P^\circ_t \Phi$, $t\ge0$, $Q^\bullet_t:=\Lambda P^\bullet_t\Phi$ and $\widetilde Q := \widetilde\Lambda\widetilde P\widetilde\Phi$, where
\[
 \wLambda((y,s), dxdt)= \Lambda(y, dx) \delta_s(dt) \quad \text{and}\quad \widetilde{\Phi}((x,t), dyds) = \delta_{\phi(x),t}(dyds).
\]

Following \cite{RogersPitman}, the criteria for a discrete-time process to be intertwined below a Markov \emph{chain} are the same as in Definition \ref{def:intertwining}, but with single-step transition kernels in place of $P_t$ and $Q_t$ in \ref{item:intertwining}. This definition is sufficient for the same conclusion as in the continuous setting: if the processes additionally satisfy criterion \ref{item:intertwin:init} noted after Definition \ref{def:intertwining}, then the image process is also Markovian.

\begin{lemma}\label{lem:intertwin_jump}
 Suppose the pair of triplets of stochastic kernels $(P_t^\circ,Q_t^\circ,\Lambda)$ and $(\widetilde{P},\widetilde{Q},\widetilde{\Lambda})$
  satisfy the intertwining conditions
 \begin{equation}\label{eq:assumed_intertwin}
  \Lambda P_t^\circ = Q_t^\circ\Lambda,\ t\ge0,\quad \text{and}\quad \widetilde\Lambda \widetilde P = \widetilde Q\widetilde\Lambda
 \end{equation}
 for all $t\!>\!0$. Then $(P^\bullet_t,Q^\bullet_t,\Lambda)$ also satisfies the intertwining condition $\Lambda P_t^\bullet=Q_t^\bullet\Lambda$ for all
 $t\!>\!0$.
%   $(Y^\circ(t),t\ge0)$ and $((Y^\bullet(S_n),S_n),n\ge0)$ are respectively intertwined below $(X^\circ(t),t\ge0)$ and $((X^\bullet(S_n),S_n),n\ge0)$ via respective kernels $\Lambda$ and $\widetilde\Lambda$; i.e.Then $(Y^\bullet(t),t\ge0)$ is intertwined below $(X^\bullet(t),t\ge0)$ via $\Lambda$ and, in particular, if $X^\bullet(0)$ has conditional law $\Lambda(Y^\bullet(0),\cdot\,)$ given $Y^\bullet(0)$, then $(Y^\bullet(t),t\ge0)$ is a Markov process.
\end{lemma}

\begin{proof}
 Let $f\colon \bX_\partial\rightarrow [0, \infty)$ be bounded and measurable such that $f(\partial)=0$ and fix $t>0$. Now, for any $y \in \bY$,
 \begin{align*}
  \int_{\bX} \Lambda(y,dx) & \int_{\bX_\partial}  P_t^\bullet (x,dz) f(z)= \int_{\bX} \Lambda(y, dx) \bE_x\left[f\left( X_t^\bullet\right)\right]\\
 	&= \sum_{k=0}^\infty \int_{\bX} \Lambda(y, dx) \bE_x \left[ f\left( X_t^\bullet\right) \cf\{ S_k \le t < S_{k+1}\} \right]\\
 	&=  \sum_{k=0}^\infty \int_{\bX} \Lambda(y, dx) \bE_x \left[ \cf\{ S_k < t\} \bE_{X^\bullet(S_k)}\left( f\left( X_{t-S_k}^\bullet\right) \cf\{ t - S_k < \zeta\} \right) \right]. 
 \end{align*}
 In the last line, we have used the strong Markov property for $X^\bullet$ at the stopping times $(S_k)$. 
 Now, we write the above expectations in terms of the Markov chain $\big( \big(X^\bullet(S_n), S_n\big) \big)$, on a probability space where, under
 $\bP_{x,s}$ this Markov chain starts from $(x,s)$. Define a function $h\colon\bX\times [0, \infty) \rightarrow \bR$ by 
 \[
 \begin{split}
  h(x,s)&= \cf\{ s < t\} \bE_{x} \left( f\left( X_{t-s}^\bullet\right) \cf\{ t - s < S_1\} \right)
  	= \cf\{ s < t\} \bE_{x} \left( f\left( X_{t-s}^\circ\right)\right).
 \end{split}
 \]
 Applying \eqref{eq:assumed_intertwin} twice, first for the chain and then for the killed process, we get 
 \begin{align*}
  \int_{\bX} & \Lambda(y,dx) \int_{\bX_\partial} P_t^\bullet (x,dz) f(z)
  	= \sum_{k=0}^\infty \int_{\bX\times [0, \infty)} \wLambda((y,0), dx ds) \bE_{x,s}\left[ h(X^\bullet(S_k), S_k) \right]%\displaybreak
  	\\
  	&= \sum_{k=0}^\infty \int_{\bX\times [0, \infty)} \wLambda((y,0), dx ds) \int_{\bX_\partial\times [0, \infty)}\widetilde{P}^k \left( (x,s), dz du\right) h(z,u)\\
  	&= \sum_{k=0}^\infty \int_{\bY\times [0, \infty)} \widetilde{Q}^k((y,0),dw ds) \int_{\bX_\partial \times [0, \infty)} \wLambda\left( (w,s), dzdu \right) h(z,u)\\   
  	&= \sum_{k=0}^\infty \int_{\bY\times [0, \infty)} \widetilde{Q}^k((y,0),dw ds) \cf\{s < t \} \int_{\bX_\partial} \Lambda\left( w, dz \right) \bE_{z} \left[ f\left( X_{t-s}^\circ\right) \right]  \\
  	&=\sum_{k=0}^\infty \int_{\bY\times [0, \infty)} \widetilde{Q}^k((y,0),dw ds) \cf\{s < t \} \int_{\bX_\partial} \Lambda\left( w, dz \right) \int_{\bX_\partial} P^\circ_{t-s} (z, dr) f(r)  \\
  	&=  \sum_{k=0}^\infty \int_{\bY \times [0, \infty)} \widetilde{Q}^k((y,0),dw ds) \cf\{s < t \}  \int_{\bY_\partial} Q_{t-s}^\circ(w,dv) \int_{\bX_\partial} \Lambda(v,dr) f(r).
 \end{align*}
 %by \eqref{eq:assumedintertwin} again. 
 
 We now claim that the last line in the above display is exactly 
 \[
  \int_{\bY_\partial} Q_t^\bullet(y,dw)\int_{\bX_\partial} \Lambda(w,dx) f(x),
 \]
 which will prove the statement of the lemma. To see this let $F:\bX_\partial\rightarrow [0, \infty)$ be given by
 \[
  F(x)= \int_{\bX_\partial} \Lambda(\phi(x), dz) f(z).
 \]
 Then, by definition of $Q_t^\bullet$ and a very similar calculation as before, but replacing $f$ by $F$, we get 
 \begin{align*}
  	\int_{\bY_\partial} &Q_t^\bullet(y,w)\int_\bX \Lambda(w,x) f(x)= \int_{\bX_\partial} \Lambda(y,dx) \int_{\bX_\partial} P^\bullet_t(x,dz) F(z)\\
  	&= \int_\bX \Lambda(y,dx) \sum_{k=0}^\infty \bE_x\left[ F\left( X^\bullet_t \right) \cf\{ S_k \le t < S_{k+1}\}\right]\\
  	&= \sum_{k=0}^\infty \int_\bX \wLambda\left((y,0), dxdu \right) \int_{{\bX_\partial}\times[0, \infty)}\widetilde{P}^k((x,u), dzds) 
  		\cf\{s<t\}\int_{\bX_\partial}P^\circ_{t-s}(z,dr)F(r)\\
  		%H(z,u), \quad \text{for some }H(z,u),\\
  	&=\sum_{k=0}^\infty  \int_{\bY_\partial\times [0, \infty)} \widetilde{Q}^k\left( (y,0), dwds\right)  \cf\{ s < t \} \int_{\bX_\partial} \Lambda(w,dz) \int_{\bX_\partial} P_{t-s}^\circ(z,dr) F(r)\\
  	&= \sum_{k=0}^\infty  \int_{\bY_\partial\times [0, \infty)} \widetilde{Q}^k\left( (y,0), dwds\right)  \cf\{ s < t \} \int_{\bX_\partial} \Lambda(w,dz) \int_{\bX_\partial} P_{t-s}^\circ(z,dr) \int_{\bX_\partial} \Lambda(\phi(r), du) f(u)\\
  	&=\sum_{k=0}^\infty  \int_{\bY_\partial\times [0, \infty)} \widetilde{Q}^k\left( (y,0), dwds\right)  \cf\{ s < t \} \int_{\bY_\partial}  Q_{t-s}^\circ(w,dv) \int_{\bX_\partial} \Lambda(v,du) f(u),
 \end{align*}
 where the final line is applying the definition of $Q_t^\circ$. This proves the result.
\end{proof}

In fact, this proof does not require that $X^\circ$ have c\`adl\`ag paths, but only that we can make some sense of taking a left limit at the absorption time $\zeta$, so that we can carry out the construction.

We require one additional property related to intertwining. For the following, suppose that $(X(t),t\ge0)$ and $(Y(t),t\ge0)$ are strong Markov processes on respective metric spaces $(\bX,d_\bX)$ and $(\bY,d_\bY)$, and that they are intertwined via the measurable map $\phi\colon \bX\to\bY$ and the kernel $\Lambda\colon \bY\times\cX\to [0,1]$, where $\cX$ is the Borel $\sigma$-algebra associated with $d_X$. Let $(\cF^Y(t),t\ge0)$ denote the filtration generated by $(Y(t),t\ge0)$. Rogers and Pitman \cite{RogersPitman} note that, if two processes satisfy criteria (i)-(iii) of Definition \ref{def:intertwining} of intertwining, then they additionally satisfy the following strong form of criterion \ref{item:intertwining_v2}:
\begin{equation}\label{eq:inter:strongish}
 \bP\big\{X(t)\in A\ \big|\ \cF^Y(t)\big\} = \Lambda(Y(t),A) \quad \text{for }t\ge0,\ A\in\cX.
\end{equation}

\begin{lemma}\label{lem:intertwin_strong}
 Suppose that $\phi$ is continous, $\Lambda(y,\cdot\,)$ is weakly continuous in $y\in\bY$, and $X$ and $Y$ have c\`adl\`ag sample paths. Then for any stopping time $\tau$ in $(\cF^Y(t),t\ge0)$, equation \eqref{eq:inter:strongish} is satisfied with $\tau$ in place of $t$.
\end{lemma}

\begin{proof}
 First, we prove the result for a discrete stopping time. Suppose that $\cT:=\{ t_1, t_2, \ldots,\}$ is a countable discrete set such that $\tau\in \cT$ a.s.. Fix an arbitrary $G \in \cF^Y(\tau)$. Then 
 \[
  \bP\left( \{ X(\tau) \in A\} \cap G \right)= \sum_{i=1}^\infty \bP\left( \{ X(t_i) \in A\} \cap G \cap \{ \tau =t_i\} \right). 
 \]
 By definition, $G \cap \{ \tau =t_i\} \in \cF^Y(t_i)$. Hence, by applying \eqref{eq:inter:strongish} at each $t_i$, we get 
 \[
 \begin{split}
  \bP\left( \{ X(\tau) \in A\} \cap G \right)&= \sum_{i=1}^\infty \bE\left[ \cf_{G \cap \{ \tau =t_i\}} \bP\left(X(t_i)\in A\mid \cF^Y(t_i)\right) \right]\\
  	&= \sum_{i=1}^\infty \bE\left[ \cf_{G \cap \{ \tau =t_i\}} \Lambda\left(Y(t_i), A\right) \right]= \bE\left[ \cf_G \Lambda\left(Y(\tau), A \right)\right].
  \end{split}
 \]
 Since $\Lambda\left(Y(\tau), \cdot\, \right)$ is measurable with respect to $\cF^Y(\tau)$, the claim follows in this case. 
 
 For an arbitrary stopping time $\tau$, there exists a sequence of discrete stopping times $\tau_k$, $k \in \bN$, such that $\tau_k \downarrow \tau$, almost surely. In particular $\cF^Y(\tau) \subseteq \cF^Y(\tau_k)$, for all $k \in \bN$. Hence, if $G \in \cF^Y(\tau)$, then by the above paragraph, for each $k\in \bN$, and for a continuous function $f\colon \bX\to [0,\infty)$,
 \[
  \bE\left[ f(X(\tau_k))\cf_G \right] = \bE\left[ \cf_G\int_\bX \Lambda\left(Y(\tau_k),dx\right)  f(x)\right].
 \]
 We take the limit as $k\rightarrow\infty$ above and appeal to the right-continuity of $X$ and $Y$ and continuity of $\Lambda(y, \cdot\,)$ in $y$ to conclude that $\bE\left[ f(X(\tau))\cf_G \right] = \bE\big[ \cf_G\int_\bX \Lambda\left(Y(\tau),dx\right) f(x)\big]$. 
\end{proof}

\section{Proof of Lemma \ref{lem:degen_diff}}\label{sec:non_acc_2}

As in the statement of the lemma, fix $k\ge 3$, $\epsilon>0$, $T\in \TInt_{k-1}$ with $\|T\|>\epsilon$, and let $(\cT^y,y\ge0)$ denote a resampling $k$-tree evolution with initial distribution $\cT^0\sim\Lambda_{k,[k-1]}(T,\cdot\,)$. Let $(D_n,n\ge1)$ denote the sequence of all degeneration times of this evolution and $(D^*_n,n\ge1)$ the subsequence of degeneration times at which label $k$ drops and resamples. We prove this lemma in three cases.
\begin{enumerate}[label = Case \arabic*:, ref = \arabic*, itemindent=1.82cm, leftmargin=0pt]
 \item\label{case:degdif:leaf} $T$ contains a leaf block of mass $x_i > \|T\|/2k$.
 \item\label{case:degdif:IP_big} $T$ contains an edge partition $\beta$ of mass at least $\|T\|/2k$, and $\beta$ contains a block of mass at least $\|\beta\|/2k^2$.
 \item\label{case:degdif:IP_small} $T$ contains an edge partition $\beta$ of mass at least $\|T\|/2k$, and each block in $\beta$ has mass less than $\|\beta\|/2k^2$.
\end{enumerate}

\begin{proof}[Proof of Lemma \ref{lem:degen_diff}, Case \ref{case:degdif:leaf}]
 With probability at least $1/2k$, the kernel $\Lambda_{k,[k-1]}(T,\cdot\,)$ inserts label $k$ into the large leaf block $i$, splitting it into a Brownian reduced 2-tree $(x_i,x_k,\beta_{\{i,k\}})$.  This type-2 compound $(\cU^y,y\in [0,D_1))$ will then evolve in pseudo-stationarity, as in Proposition \ref{prop:012:pseudo}, until the first degeneration time $D_1$ of $(\cT^y,y\ge0)$.
 
 Let $A_1$ denote the event that $\cU^{D_1-}$ is not degenerate, i.e.\ some other compound degenerates at time $D_1$. On $A_1$, some outside label may swap places with $i$ and cause $i$ to resample. However, as noted in the discussion of cases \ref{case:degen:type2}, \ref{case:degen:self}, and \ref{case:degen:nephew} in Section \ref{sec:const:intertwin}, no label will swap places with label $k$ at time $D_1$ on the event $A_1$. Let $\mathcal{R}_1$ denote the subtree of $\varrho(\cT^{D_1-})$ corresponding to $\cU^{D_1-}$. This equals $\cU^{D_1-}$ if no label swaps with $i$. By Proposition \ref{prop:012:pseudo} and exchangeability of labels in Brownian reduced 2-trees, on the event $A_1$ the tree $\cR_1$ is a Brownian reduced 2-tree.
 
 Let $B_1$ denote the event that the label dropped at $D_1$ resamples into a block of $\cR_1$.  Let $\cU^{D_1}$ denote the resulting subtree after resampling. On the event $A_1\cap B_1^c$, $\cU^{D_1} = \cR_1$ is again a Brownian reduced 2-tree, by definition of the resampling $k$-tree evolution. On the event $A_1\cap B_1$, the tree $\cU^{D_1}$ is a Brownian reduced 3-tree, by \eqref{eq:B_ktree_resamp} and the exchangeability of labels.
 
 We extend this construction inductively. Suppose that on the event $\bigcap_{m=1}^n A_m$, the tree $\cU^{D_n}$ is a Brownian reduced $M$-tree, for some (random) $M$. We define $(\cU^y,y\in [D_n,D_{n+1}))$ to be the $M$-tree evolution in this subtree during this time interval. Let $A_{n+1}$ denote the event that $\cU^{D_{n+1}-}$ is non-degenerate, $\cR_{n+1}$ the corresponding subtree in $\varrho(\cT^{D_{n+1}-})$, $B_{n+1}$ the event that the dropped label resamples into a block in $\cR_{n+1}$, and $\cU^{D_{n+1}}$ the corresponding subtree in $\cT^{D_{n+1}}$. Then on $\bigcap_{m=1}^{n+1} A_m$ the tree $\cR_{n+1}$ is again a Brownian reduced $M$-tree, by the same arguments as above, with Proposition \ref{prop:pseudo:pre_D} in place of Proposition \ref{prop:012:pseudo}. On $B_{n+1}^c\cap \bigcap_{m=1}^{n+1} A_m$, the tree $\cU^{D_{n+1}} = \cR_{n+1}$ is a Brownian reduced $M$-tree, and on $B_{n+1}\cap \bigcap_{m=1}^{n+1} A_m$ the tree $\cU^{D_{n+1}}$ is a Brownian reduced $(M\!+\!1)$-tree.
 
 In this manner, we define $(\cU^y,y\in [0,D_N))$ where $D_N$ is the first time that $\cU^{y-}$ attains a degenerate state as a left limit. Let $A^y$ denote the label set of $\cU^y$ for $y\in [0,D_N)$; by the preceding argument and Proposition \ref{prop:pseudo:pre_D}, $\cU^y$ is conditionally a Brownian reduced $(\# A^y)$-tree given $\{y < D_N\}$. Let $(\sigma^y,y\in [0,D_N))$ denote the evolving permutation that composes all label swaps due to the swap-and-reduce map, $\sigma^y = \tau_n\circ\tau_{n-1}\circ\cdots\circ\tau_1$ for $y\in [D_n,D_{n+1})$, where $\tau_m$ is the label swap permutation that occurs at time $D_m$. 
 
 We can simplify this account by considering a pseudo-stationary killed $k$-tree evolution $(\cV^y,y\in [0,D''))$ coupled so that $\pi_{A^y}\circ\sigma^y(\cV^y) = \cU^y$, $y\in [0,D'')$, where $D''$ is the degeneration time of $(\cV^y)$. Such a coupling is possible due to the consistency result of Proposition \ref{prop:consistency_0} and the exchangeability of labels evident in Definition \ref{def:killed_ktree} of killed $k$-tree evolutions. Note that, in particular, $D''$ precedes the first time at which a label in $(\cU^y)$ degenerates. Moreover, following the discussion of cases \ref{case:degen:type2}, \ref{case:degen:self}, and \ref{case:degen:nephew} in Section \ref{sec:const:intertwin}, label $k$ cannot be dropped in degeneration until a label within $(\cU^y)$ degenerates.
 
 There is some $\delta>0$ sufficiently small so that a pseudo-stationary $k$-tree evolution with initial mass $\epsilon/2k$ will avoid degenerating prior to time $\delta$ with probability at least $2k\delta$. By the scaling property of Lemma \ref{lem:scaling}, this same $\delta$ bound holds for pseudo-stationary $k$-tree evolutions with greater initial mass. Applying this bound to $(\cV^y,y\in [0,D''))$ proves the lemma in this case.
\end{proof}

\begin{proof}[Proof of Lemma \ref{lem:degen_diff}, Case \ref{case:degdif:IP_big}]
 In this case, with probability at least $1/4k^3$, label $k$ is inserted into a ``large'' block in $\beta$ of mass at least $\|\beta\|/2k^2$. If another label resamples into this same block prior to time $D^*_1$, then we are in the regime of Case \ref{case:degdif:leaf}, and the same argument applies, albeit with smaller initial mass proportion. However, if no other label resamples into this block then, although it is unlikely for this block to vanish quickly, it is possible for label $k$ to be dropped in degeneration if a label that is a nephew of $k$ causes degeneration (case \ref{case:degen:nephew} in Section \ref{sec:const:intertwin}). In this latter case, however, that label swaps into the block in which label $k$ was sitting. Then, label $k$ resamples and may jump back into this large block with probability bounded away from zero. This, again, puts us in the regime of Case \ref{case:degdif:leaf}. In this case, $D^*_1$ may be small with high probability, but not $D^*_2$.
 
 More formally, a version of the argument for Case \ref{case:degdif:leaf} yields $\delta>0$ for which, with probability at least $\delta$: (i) the kernel $\Lambda_{k,[k-1]}(T,\cdot\,)$ inserts label $k$ into a block in $\beta$ with mass at least $\|\beta\|/2k^2$; (ii) this block, or a subtree created within this block survives to time $\delta$ with its mass staying above $\|\beta\|/3k^2$; (iii) the total mass stays below $2\|T\|$ and either (iv) $D^*_1 > \delta$; or (v) $D^*_{1} \le \delta$ but at time $D^*_{1}$, label $k$ resamples back into this same block, which only holds a single other label at that time; and then (vi) $D^*_{2} - D^*_{1} > \delta$.
\end{proof}

To prove Case \ref{case:degdif:IP_small}, we require two lemmas, one of which recalls additional properties of type-0/1/2 evolutions from \cite{Paper1,Paper3}.

\begin{lemma}\label{lem:012:clade_ish}
 Fix $(x_1,x_2,\beta)\in [0,\infty)^2\times\cI$ with $x_1+x_2>0$. There exist a type-0 evolution $(\alpha_0^y,y\ge0)$, a type-1 evolution $((m^y,\alpha_1^y),y\ge0)$, and a type-2 evolution $((m_1^y,m_2^y,\alpha_2^y),y\ge0)$ with respective initial states $\beta$, $(x_1,\beta)$, and $(x_1,x_2,\beta)$, coupled in such a way that for every $y$, there exists an injective, left-to-right order-preserving and mass-preserving map sending the blocks of $\alpha_2^y$ to blocks of $\alpha_1^y$, and a map with these same properties sending the blocks of $\alpha_1^y$ to blocks of $\alpha_0^y$.
\end{lemma}

These assertions are immediate from the pathwise constructions of type-0/1/2 evolutions in \cite[Definitions 3.21, 5.14]{Paper1} and \cite[Definition 17]{Paper3}. They can alternatively be derived as consequences of Definitions \ref{def:type01} and \ref{def:type2} and the transition kernels described in Proposition \ref{prop:012:transn}.

\begin{lemma}\label{lem:type2:smallblocks}
 Fix $c\in (0,1/2)$ and $x>0$. Consider $u_1,u_2\ge0$ with $u_1+u_2>0$ and $\beta\in\cI$ with $\|\beta\|>x$ and none of its blocks having mass greater than $c\|\beta\|$. For every $\epsilon>0$ there exists some $\delta = \delta(x,c) > 0$ that does not depend on $(u_1,u_2,\beta)$ such that with probability at least $1-\epsilon$, a type-2 evolution with initial state $(u_1,u_2,\beta)$ avoids degenerating prior to time $\delta$.
\end{lemma}

\begin{proof}
 Fix a block $(a,b)\in\beta$ with $a\in [c\|\beta\|,2c\|\beta\|]$ and let
 \begin{equation*}
  \beta_0 := \{(a',b')\in\beta\colon a'<a\},\quad \beta_1 := \{(a'-b,b'-b)\colon (a',b')\in\beta, a'\ge b\}
 \end{equation*}
 so that $\beta = \beta_0\concat (0,b-a)\concat\beta_1$. We follow Proposition \ref{prop:012:concat}\ref{item:012concat:2+1}, in which a type-2 evolution is formed by concatenating a type-2 with a type-1. In particular, let $\wh\Gamma^y := \big( \wh m_1^y, \wh m_2^y,\wh \alpha^y\big)$ and $\widetilde\Gamma^y := (\widetilde m^y,\widetilde\alpha^y)$, $y\ge0$, denote a type-2 and a type-1 evolution with respective initial states $(u_1,u_2,\beta_0)$ and $(b-a,\beta_1)$. Let $\wh D$ denote the degeneration time of $(\wh\Gamma^y,y\ge0)$ and let $\wh Z$ denote the time at which $\|\wh\Gamma^y\|$ hits zero. Let $I$ equal 1 if $\wh m_1^{\wh D}>0$ or 2 if $\wh m_2^{\wh D}>0$, and set $(X_I,X_{3-I}) := \big(\wh m_I^{\wh D},\widetilde m^{\wh D}\big)$. Finally, let $\big(\widebar m_1^y,\widebar m_2^y,\widebar\alpha^y\big)$, $y\ge 0$ denote a type-2 evolution with initial state $(X_1,X_2,\widetilde\alpha^{\widehat D})$, conditionally independent of $((\wh\Gamma^y,\widetilde\Gamma^y)),y\in [0,\widehat D]))$ given this initial state, but coupled to have $\widebar m_I^y = \wh m_I^{\widehat D + y}$ for $y\in [0,\widehat Z-\widehat D]$. By Proposition \ref{prop:012:concat}\ref{item:012concat:2+1}, the following is a type-2 evolution:
 \begin{equation}\label{eq:SB:concat_1}
  \left\{\begin{array}{ll}
   (\wh m_1^y,\;\wh m_2^y,\;\wh\alpha^y\concat(0,\widetilde m^y)\concat\widetilde\alpha^y)	& \text{for }y\in [0,\wh D),\\
   (\widebar m_1^{y-\wh D},\;\widebar m_2^{y-\wh D},\;\widebar\alpha^{y-\wh D})	& \text{for }y\ge \wh D.
  \end{array}\right.
 \end{equation}
 Moreover, by the Markov property of type-1 evolutions, Definition \ref{def:type2} of type-2 evolutions, and the symmetry noted in Lemma \ref{lem:type2:symm}, the following is a stopped type-1 evolution:
 \begin{equation}\label{eq:SB:concat_2}
  \left\{\begin{array}{ll}
   (\widetilde m^y,\widetilde\alpha^y)	& \text{for }y\in [0,\wh D),\\
   (\widebar m_{3-I}^{y-\wh D},\widebar\alpha^{y-\wh D})	& \text{for }y\in [\wh D,\wh Z].
  \end{array}\right.
 \end{equation}
 
 Let $\delta>0$ be sufficiently small so that, with probability at least $\sqrt{1-\epsilon}$, a $\besq_{cx}(-1)$ avoids hitting zero prior to time $\delta$, and likewise for a $\besq_{(1-2c)x}(0)$. Then with probability at least $1 - \epsilon = (\sqrt{1-\epsilon})^2$, both $(\|\wh\Gamma^y\|,y\ge0)$ and the $\besq(0)$ total mass of the type-1 evolution of \eqref{eq:SB:concat_2} avoid hitting zero prior to time $\delta$. On this event, the type-2 evolution of \eqref{eq:SB:concat_1} does not degenerate prior to time $\delta$.
\end{proof}

\begin{proof}[Proof of Lemma \ref{lem:degen_diff}, Case \ref{case:degdif:IP_small}]
  Informally, we proved that degenerations of $k$ may take a long time in Cases \ref{case:degdif:leaf} and \ref{case:degdif:IP_big} by controlling 
  the degeneration times of pseudo-stationary structures inserted into large blocks in repeated resampling events. In Case 
  \ref{case:degdif:IP_small}, there are no large blocks, and indeed large blocks may never form. Instead, there must be a large interval partition,
  which we can cut rather evenly into $2k-1$ sub-partitions. We will control the degeneration of evolving sub-partitions and the probability
  that insertions are into distinct non-adjacent internal sub-partitions.
  
  Specifically, let us follow the notation introduced at the start of this appendix. We can decompose $\beta=\beta_1\star\cdots\star\beta_{2k-1}$
  into sub-partitions $\beta_i$ with $\|\beta_1\|\ge \|\beta\|/k$ and $\|\beta\|/4k\le\|\beta_i\|\le\|\beta\|/2k$ for $2\le i\le 2k-1$, since no 
  block exceeds mass $\|\beta\|/2k^2\le\|\beta\|/4k$. 
  
  With probability at least $1/8k^2$, the kernel $\Lambda_{k,[k-1]}(T,\cdot\,)$ inserts label $k$ into $\beta_2$, splitting 
  $\beta_2=\beta_2^-\star(0,x_k)\star\beta_2^+$. Then label $k$ is in the type-1 compound 
  $\cU^0=(x_k,\beta_2^+\star\beta_3\star\ldots\star\beta_{2k-1})$, while $\beta_1\star\beta_2^-$ is the interval partition of the 
  compound associated with the sibling edge of $k$. Consider the concatenation $\cV^y:=\cV^y_3\star\cV^y_4\star\cdots\star\cV^y_{2k-1}$
  of type-1 evolutions $(\cV^y_i,y\ge 0)$, $3\le i\le 2k-1$, starting respectively from $\beta_2^+\star\beta_3,\beta_4,\ldots,\beta_{2k-1}$, for
  times $y$ up to the first time $D_\cV$ that one of them reaches half or double its initial mass. We denote by $D_\cW$ the degeneration time of
  the sibling edge of $k$ as part of this resampling $k$-tree evolution. If this sibling edge of $k$ is a type-2 edge, 
  denote by $u_1$ and $u_2$ its top masses and consider a type-2 evolution $(\cW^y,y\ge 0)$ starting from $(u_1,u_2,\beta_1\star\beta_2^-)$, with
  degeneration time $D_\cW$. 
  Otherwise, this edge has three or more labels, so one or both children of this edge have more than one label. For each of these children, we 
  choose as $u_1$ or $u_2$, respectively, the top mass of its smallest label. Then the degeneration time of a type-2 evolution 
  $(\widetilde{\cW}^y,y\ge 0)$ starting from $(u_1,u_2,\beta_1\star\beta_2^-)$ is stochastically dominated by the time $D_\cW$ at which the sibling edge of $k$ degenerates. We denote by $D_\cT$ the first time that $\|\cT^y\|$ reaches half or double its initial mass.
  
  Since $\|\beta_1\star\beta_2^-\|\ge\|\beta_1\|\ge\|\beta\|/k\ge\|T\|/2k^2\ge\epsilon/2k^2$, Lemma \ref{lem:type2:smallblocks} yields 
  $\delta_\cW=\delta(\epsilon/2k^2,1/2k)>0$ such that $\bP(D_\cW>\delta_1)\ge 1-1/(32k^2)^k$. 
  Let $\delta_\cT>0$ be such that a $\besq_{\epsilon}(-1)$ stays in $(\epsilon/2,2\epsilon)$ up to time $\delta_\cT$ with probability at least
  $1-1/(32k^2)^k$. Then $\bP(D_\cT>\delta_\cT)\ge 1-1/(32k^2)^k$. Finally, let $\delta_\cV>0$ be such that the probability that 
  $\besq_{\epsilon/8k^2}(0)$ does not exit $(\epsilon/16k^2,\epsilon/4k^2)$ before time $\delta_\cV$ exceeds $(1/2)^{1/2k}$. Since the $2k-3$ 
  independent type-1 evolutions are starting from greater initial mass, we obtain from Proposition \ref{prop:012:mass} and the scaling property of 
  Lemma \ref{lem:scaling} that $\bP(D_\cV>\delta_\cV)>1/2$. 
  
  We proceed in a way similar to Case 1 and inductively construct a subtree evolution $(\cU^y,y\in[0,D_\cU))$ coupled to $(\cV^y,y\in[0,D_\cV))$ on 
  events $A_{n+1}$, $n\ge 0$, on which $D_{n+1}<\min\{D_\cV,D_\cW,D_\cT\}$ and any resampling of a label at time $D_{n+1}$ when $\cU^{D_{n+1}-}$ 
  has $j-1$ labels occurs into a block of $\cV^{D_{n+1}}_{2j}$, $j=2,\ldots,k-1$. Given that $D_{n+1}<\min\{D_\cV,D_\cW,D_\cT\}$, such a block is 
  chosen by the resampling kernel with (conditional) probability exceeding $\|\beta_{2j}\|/(4\|T\|)\ge 1/32k^2$. 
  
  Thus, with probability at least $\delta:=\min\{\delta_\cV,\delta_\cW,\delta_\cT,1/(32k^2)^{k-2}\}$, we have $D_1^*>\delta$ and so $D_2^*>\delta$.
\end{proof}

\bibliographystyle{abbrv}
\bibliography{AldousDiffusion4}

\end{document}

\begin{proposition}[Propositions 4.30, 5.4, 5.16 of \cite{Paper1}]\label{prop:012:transn}
 Fix a time $y>0$. Fix $(x,\beta)\in [0,\infty)\times\cI$, with the constraint that either $x>0$ or $\beta$ has no leftmost block (i.e.\ 0 is an accumulation point of small blocks in $\beta$).  For each block $(a,b)\in\beta$,
 \begin{itemize}
  \item let $\chi_{(a,b)}\sim \distribfont{Bernoulli}(1-e^{-w/2y})$, where $w = b-a$,
  \item let $L_{(a,b)}$ be a $(0,\infty)$-valued random variable with probability density function
  \begin{equation}
   \bP\{L_{(a,b)}\in du\} = \frac{1}{\sqrt{2\pi}}\frac{\sqrt{y}}{u^{3/2}}\frac{e^{-u/2y}}{e^{w/2y}-1}\left( 1 - \cosh\!\left(\frac{\sqrt{wu}}{y}\right) + \frac{\sqrt{wu}}{y} \sinh\!\left(\frac{\sqrt{wu}}{y}\right) \right)du,
  \end{equation}
  \item let $S_{(a,b)}\sim\ExpDist[(2y)^{-1/2}]$, and
  \item let $R_{(a,b)}$ be a subordinator with Laplace exponent 
  %\begin{equation}
   $\lambda\mapsto \left(\lambda+\frac{1}{2y}\right)^{1/2}-\left(\frac{1}{2y}\right)^{1/2}$.
  %\end{equation}
 \end{itemize}
 We define $(\chi_x,L_x,S_x,R_x)$ correspondingly, with $x$ in place of $w = b-a$. We take these objects to all be jointly independent. For the purpose of the following, we take
 $$\textsc{ip}(R,S) := \big\{(R(t-),R(t))\colon t < S,\ R(t-)<R(t)\big\}.$$
 
 The map that sends an initial state $\beta$ and time $y$ to the law of 
 %If $(\alpha^z,z\ge0)$ is a type-0 evolution with initial state $\beta$, then
 \begin{equation}\label{eq:type0:kernel}
  \textsc{ip}(R_x,S_x)\concat \Concat_{U\in \alpha\colon \chi_U=1}\big((0,L_U)\concat \textsc{ip}(R_U,S_U)\big)
 \end{equation}
 is a Markov transition kernel on $\cI$. The same holds for the map sending $(x,\beta)$ and $y$ to the law of 
 %If $((m^z,\alpha^z),z\ge0)$ is instead a type-1 evolution with initial state $(x,\beta)$, then
 \begin{equation}\label{eq:type1:kernel}
  \Concat_{U\in \{x\}\cup\alpha\colon \chi_U=1}\big((0,L_U)\concat \textsc{ip}(R_U,S_U)\big).
 \end{equation}
\end{proposition}